\documentclass[a4paper,10pt]{article}
\usepackage{fullpage}
\RequirePackage{natbib}
\RequirePackage[OT1]{fontenc}
\RequirePackage{amsthm,amsmath}
\RequirePackage[colorlinks=true]{hyperref}
\hypersetup{
	linkcolor = {blue},
	citecolor = {blue},
	urlcolor = {blue}
}

\usepackage[utf8]{inputenc}
\usepackage{times}
\usepackage{amsfonts,amssymb}	
\usepackage{mathrsfs}	
\usepackage{dsfont}
\usepackage{authblk}
\usepackage{soul} 
\usepackage{bm}
\usepackage{bbm}
\usepackage{color}
\usepackage[table]{xcolor}
\usepackage{euscript,graphicx}
\usepackage{float}
\usepackage{subfigure}
\usepackage{xspace}
\usepackage{booktabs}
\usepackage[font=small,labelfont=bf]{caption}

\usepackage{appendix}
\usepackage{tikz}
\usepackage{pgfplots}
\usepgfplotslibrary{fillbetween}
\usetikzlibrary{positioning,fadings,arrows,babel}
\usetikzlibrary{tikzmark,decorations.pathreplacing}
\pgfplotsset{width=10cm,compat=1.9}

\usepackage{ upgreek }

\usepgfplotslibrary{fillbetween}

\newcommand*{\doi}[1]{\href{http://dx.doi.org/#1}{\small{doi:#1}}} 
\definecolor{Green}{RGB}{0,150,1} 
\definecolor{bleu}{RGB}{0,57,128}
\definecolor{orange}{RGB}{255,128,0}
\definecolor{violet}{RGB}{200,0,200}
\definecolor{vert}{RGB}{12,120,20}
\definecolor{bleuc}{RGB}{10,30,90}
\definecolor{rouge}{RGB}{200,0,0}
\newcommand{\Blue}[1]{{\color{blue}{#1}\color{black}\xspace}}

\newcommand{\tabhead}[1]{\textbf{#1}}

\newtheorem{theorem}{Theorem}[section]

\newtheorem{lem}[theorem]{Lemma}
\newtheorem{prop}[theorem]{Proposition}
\newtheorem{rem}[theorem]{Remark}
\newtheorem{defn}[theorem]{Definition}

\newtheorem{hypo}{\bf Hypothesis}

\numberwithin{equation}{section}

\newcommand{\av}[1]{{{\displaystyle{\langle} {#1}\displaystyle{\rangle}}}}

\def\BB{\mathcal{B}}

\def\RR{\mathbb{R}}

\def\Ff{\mathcal{F}}

\def\EE{\mathbb{E}}
\def\PP{\mathbb{P}}

\def\KK{\mathcal{K}}

\def\expSch{\textrm{Exp-EM\xspace}\xspace}

\def\ind{\mathds{1}}

\def\qn1{{X}_{t_{n+1}}}

\newcommand{\Dt}{\Delta t}
\newcommand{\dt}{\delta(t)}
\newcommand{\ds}{\delta(s)}

\newcommand{\oX}{\overline{X}}

\newcommand{\alphab}{{\beta}}

\newcommand{\pntCompact}{\upzeta}
\newcommand{\pntDisc}{\upchi}
\newcommand{\texp}{\overline{S}^{\Delta t}}

\newcommand{\sgn}{\displaystyle {\rm sgn}}
\newcommand{\diffc}{{\text{\scriptsize{$\Upsigma$}}}}
\newcommand{\diffcd}{{\text{\scriptsize{$\Upsigma'$}}}}

\newcommand{\bexp}{b^{\text{exp}}}
\newcommand{\sexp}{{\sigma^{\text{exp}}}}

\newcommand{\Lg}{L_{\textit{G}}\,}

\newcommand{\Lips}{\mathcal{C}_{\textit{max}}}
\newcommand{\osLip}{L_{b_{\text{loc}}}}

\newcommand{\Cczero}{\mathcal{L}^{\scriptsize{\textcircled{\tiny{0}}}}}

\newcommand{\uu}{\frac{\pntDisc_m}{4}}

\newcommand{\wkappa}{\kappa_{\texttt{Weak}}}
\newcommand{\mkappa}[1]{\kappa_{\texttt{Strong}({#1})}}
\newcommand{\hhh}{h}

\begin{document}

\title{
Strong convergence of the exponential Euler scheme for SDEs with superlinear growth coefficients and one-sided Lipschitz drift
}

\author[1]{Mireille Bossy\footnote{mireille.bossy@inria.fr}}
\author[2]{Kerlyns Mart\'inez\footnote{kermartinez@udec.cl}}
\affil[1]{Universit{\'e} C{\^o}te d'Azur, Inria, France}
\affil[2]{Departamento de Ingenier\'ia Matem\'atica, Universidad de Concepci\'on, Chile}
\date{}
\maketitle
\begin{abstract}
We consider the problem of the discrete-time approximation of the solution of a one-dimensional SDE with piecewise locally Lipschitz drift and continuous diffusion coefficients with polynomial growth. 
In this paper, we study the strong convergence of a (semi-explicit) exponential-Euler scheme previously introduced in \cite{BoJaMa2021}.  We show the usual 1/2 rate of convergence for  the  exponential-Euler scheme when the drift is continuous. When the drift is discontinuous, the convergence rate is penalised by a factor $\varepsilon$ decreasing with the time-step. We examine the case of the diffusion coefficient vanishing at zero, which adds a positivity preservation condition and a convergence analysis that exploits the negative moments and exponential moments of the scheme with the help of change of time technique introduced in \cite{BBD08}. Asymptotic behaviour and theoretical stability of the exponential scheme, as well as numerical experiments, are also presented.

\end{abstract}
\tableofcontents

\section{Introduction}\label{sec:intro}
In this paper, we analyse the strong convergence of the exponential Euler scheme for the simulation of  the solution of SDE  with coefficients having superlinear  growth. More precisely, we consider the one-dimensional SDE
\begin{equation}\label{eq:IntroSDE}
dX_t = b(X_t)dt + \sigma(X_t) dW_t,~X_0=x_0>0,
\end{equation}
where $(W_t;0\leq t\leq T)$ is a standard Brownian motion on the probability space $(\Omega,\Ff,\PP)$ equipped with its natural filtration $(\Ff_t;\,0\leq t\leq T)$. 
The diffusion coefficient $\sigma:[0,+\infty)\rightarrow\RR$ is assumed to be continuous with polynomial growth:
\begin{equation}\label{eq:IntroHyposigma}
	\sigma^2(x)\leq \diffc^2 x^{2\alpha},\quad\text{for } x\in\RR^+, \quad\text{ for some } \diffc>0 \text{ and } \alpha >1.
\end{equation}
In particular, we consider the situation where  $\sigma(0) =0$, and wellposedness hypotheses leading to a positive solution to \eqref{eq:IntroSDE}. We show that the existence  and control of moments of the positive solution to \eqref{eq:IntroSDE} is granted  when the drift  $b:[0,+\infty)\rightarrow\RR$  is assumed to be  piecewise locally Lipschitz  function,  with a polynomial growth bound of the form (see Proposition~\ref{prop:XMoments}):
\begin{equation}\label{eq:IntroHypo}
b(x)\leq b(0) + B_1 x - B_2 x^{\alphab},~\forall x\in\RR^+,\quad\text{with } \alphab\geq 2\alpha-1 > 1, 
\end{equation}
for some positive constants $b(0), B_1$ and $B_2$.

Under assumptions ensuring the wellposedness and strict positivity of $X$, the exponential scheme strategy draws on the dynamics artificially transformed into a semi-linear form \eqref{eq:IntroSDE_expoform}. 
A semi-exact scheme in exponential form is then introduced, maintaining the strict positivity of the approximate solution. Moreover, it possesses identical finite moments as the exact solution under the same set of sufficient conditions.  

We refer to  \cite{BoJaMa2021} for some specific examples of applications motivating the study of the exponential scheme  for \eqref{eq:IntroSDE}, as well as for elements of review of the literature on numerical schemes adapted to SDEs with both non-Lipschitz  drift and diffusion  coefficients.   We complete this review with works address the strong error for similar SDE situations.  First, we emphasise that \cite{Kloeden} established the $L^p$-strong divergence, for $p\in[1,+\infty)$, related to the Euler-Maruyama scheme for SDEs with both drift and diffusion satisfying some superlinear growth condition. In \cite{Kloeden2}, the authors proposed a time-explicit tamed-Euler scheme to overcome this \emph{divergence} problem of the Euler approximation, based on renormalised-increments to the scheme.
Recently \cite{Jentzen-b} proved  the ${1}/{2}$ rate of the $L^p$-strong convergence for the tamed-Euler scheme for a family of SDE  that includes some locally Lipschitz cases for both  continuous drift  and diffusion coefficients. A truncated version of the Euler-Maruyama scheme has also been introduced and adapted  for positivity preserving approximation (see \cite{mao2021positivity} and references therein).  
\medskip

While \cite{BoJaMa2021} introduces the exponential Euler scheme and analyses its convergence rate in the weak sense, this paper aims to prove a convergence rate of 1/2 for the $L^p(\Omega)$-$\sup_{t\in[0,T]}$ error for the exponential scheme, as described below. The main objectives are as follows: 
\textit{(i)} To expand the $C^4$ regularity requirement for the drift $b$ in weak error analysis and start to encompass  discontinuous case.
\textit{(ii)} To enhance the control condition concerning the growth and domination of the drift and diffusion to achieve convergence rates.
\textit{(iii)}  To explore and expand the proposition of schemes  applicable to contexts of superlinear coefficients.  
Rather than a priori control of the schema by a tamed/truncated strategy, we aim to identify a robust a posteriori threshold for the approximated process, defining a value range for the scheme that ensures its convergence.   By obtaining a threshold  that expands as we refine the time-step, we pave the way for exploring natural and explicit strategies for adaptive time-step schemes handling increasingly explosive cases of SDEs. 

\subsubsection*{Exponential scheme}

The  ex\-po\-nen\-tial-Euler Maruyama scheme (\expSch for short), originates  from rewriting the SDE \eqref{eq:IntroSDE} into a quasi-linear SDE, taking part of the strict positivity of the solution
\begin{equation}\label{eq:IntroSDE_expoform}
dX_t = X_t \Big( \frac{b(X_t) - b(0)}{X_t}dt + \frac{\sigma(X_t)}{X_t} dW_t\Big) + b(0) dt,\quad X_0=x_0>0.
\end{equation}
The semi-linear integration produces, for an  homogeneous $N$-partition of the time interval $[0,T]$   with  time-step $\Delta t=t_{n+1}-t_n$, the approximation scheme:
\begin{equation}\label{eq:IntroNumericScheme0}
	\overline{X}_{t_{n+1}}=\overline{X}_{t_n}\exp\left\{\frac{\sigma(\oX_{t_n})}{\oX_{t_n}}\;(W_{t_{n+1}}-W_{t_n})+
	\left(\frac{b(\overline{X}_{t_n})-b(0)}{\overline{X}_{t_n}}-\frac12 \frac{\sigma^2(\oX_{t_n})}{\oX_{t_n}^2}\right)\Delta t\right\}+ b(0)\Delta t, \quad \oX_0 = x_0, 
\end{equation}
that preserves the positiveness of the solution. 
We refer the reader to \cite{BoJaMa2021} for a detailed construction of  \eqref{eq:IntroNumericScheme0} in the particular case $\sigma(x) = \diffc x^{\alpha}$. As pointed out, a significant advantage of the \expSch scheme is its ability to preserve the conditions under which the boundedness of moments holds compared to the continuous time process (see Proposition \ref{prop:XMoments} and Lemma \ref{lem:Schememoments}) in contrast with the classical Euler-Maruyama time-discretisation for some locally Lipschitz coefficients (see, for instance, \cite{Kloeden}).

The preservation properties (positivity, moments) and the ease of computational handling (both numerical and analytical) make the \expSch version a robust method that can be employed in diverse simulation problems, including SPDE as proposed in \cite{BDU_positivity_2023} and  \cite{erdogan2023weak}. Specifically, it can prove highly beneficial in a splitting simulation strategy aimed at accurately decomposing the equation to be solved into its exact or semi-exact form.
\medskip

While this paper primarily examines the \expSch scheme applied to  1D positive-valued SDEs (involving possibly non-integer powers $\alpha$ and $\alphab$), it is worth noting that  the \expSch scheme can  readily be extended to handle   SDEs in $\RR^d$, $ d\geq 1$,  featuring more generalized diffusion coefficient class $\sigma(x)$. In such case, the \expSch scheme  can be roughly expressed as:
\begin{equation*}
dX_t = X_t   \Big( ((b(X_t) - b(0)) \div X_t )dt + ((\sigma(X_t) - \sigma(0))\div X_t) dW_t\Big) + b(0) dt + \sigma(0) dW_t,
\end{equation*}
where  $b$ is $\RR^d$- valued and $\sigma$ is a $d\times r$ matrix,  consistent with the $r$-dimensional Brownian motion $W$.   The notation $\div$ is used here for component-wise division.  The \expSch scheme is then driven by the following SDE 
\begin{equation}\label{eq:SDE_intro_extension}
\begin{aligned}
d\overline{X}^{i}_t= & \left(\overline{X}_t-b(0){\dt} -  \sigma(0)  \delta(W_t) \right)^{i} \Big(\frac{(b(\overline{X}_{\eta(t)})-b(0))^{i}}{\overline{X}^{i}_{\eta(t)}} dt+ \frac{(\sigma(\overline{X}_{\eta(t)})-\sigma(0))^{i,k}}{\overline{X}_{\eta(t)}^{i}}   dW^k_t\Big)\\
& + (b(0) dt + \sigma(0) dW_t)^{i}, 
\end{aligned}
\end{equation}
with $\dt = t - \eta(t)$, and  $\delta(W_t)  = W_{t} - W_{\eta(t)}$.  The simplicity of the 1D case is widely exploited in our proofs. Nevertheless, the extension of the analysis to the multidimensional case is straightforward in view of the main argument of stochastic time change to circumvent the problem of stochastic Gronwall estimation. 
The extension to multidimensional  drift with discontinuity strongly depends on the topology of discontinuity sets.  Here also the 1D proof handles discontinuity points based on the ideas of sojourn time estimation of $\RR^d$-valued discrete time diffusion in small ball (\cite{bernardin2009passage}). 

\subsubsection*{Preserving positivity}

Preserving positivity in approximation schemes for multidimensional problems with polynomial coefficients poses a significant challenge, particularly in the context of epidemiological modelling. In this regard, we refer to the works of  \cite{cai2023positivity},  \cite{GREENHALGH2016218}, and the references cited therein.

The preserving positivity issue addressed here presents an additional challenge for the scheme. Indeed,  although chosen in the set of parameters that ensure sufficient conditions for the  Feller test on process $X$ (defining the Feller zone of non-explosion and uniqueness), some parameters strongly impact the behaviour of the exact and approximated processes.  Large values of the processes $X$ and $\oX$ can  abruptly return to zero.  However, the behaviour  of $X$ in the vicinity of zero is well controlled in the Feller zone, in particular when $b(0) =0$,  allowing  to upper bound all  negative moments of the exact process $X$ without additional restriction (see the first item of Proposition \ref{prop:XMoments}).   
Unlike the control of positive moments,   the control of negative moments on the \expSch scheme requires constraining certain parameters or limiting the maximum value of $\oX$ below a threshold (see Proposition \ref{prop:negative_moments}). 

Note that the  $L^p$-hypothesis (or Hasmisnskii type condition), which is commonly used to establish the existence of a superlinear diffusion in $L^p$ (see also Remark \ref{rem:exponentialPrototype} and the discussion in \cite{kelly2022adaptive}), and that typically provides also $L^p$-control for a reasonable  approximation scheme, is insufficient to get control on the scheme negative moments, at least for the \expSch scheme.

\subsection{Summary of the contributions and plan of the paper}

\paragraph{Convergence results. } 
Beyond the usual motivations for the study of strong convergence of numerical schemes -- a crucial role for the convergence of parameter estimators for the SDE model, as well as for the convergence of multilevel methods -- a main objective here is to broaden the prediction of the $L^p$-convergence of the scheme concerning the SDE parameters, further than the theoretical conditions established in \cite{BoJaMa2021} which seem far from being necessary. In section \ref{sec:num}  with the help of numerical tests, we explore the new theoretical sufficient conditions for convergence obtained on the parameters of $X$.  
\medskip
 
The polynomial domination bounds \eqref{eq:IntroHyposigma}-\eqref{eq:IntroHypo} are deeply exploited in our proof -- together with polynomial growth \ref{H:polygrowth} and piecewise locally Lipschitz behaviour \ref{H:pieceloclip} on $b$ -- to identify explicit conditions on $p$ and on the set of parameters  $b(0)$, $B_2$, $\diffc$, $\beta$ and  $\alpha$, for which the $L^p(\Omega)$-theoretical rate is applicable.  

In Theorem \ref{thm:strong_rate_stopped}, we state that the convergence is of order $1/2$ for the  $L^p(\Omega)$-$[\sup_{t\in[0,T\wedge \texp]} \cdot ]$ norm, where $\texp$ is the first exit-time of the \expSch scheme $\oX$ from the set $(0,\Dt^{-\frac{1}{\alphab-1}}]$. 
This result is obtained under a large set of parameters for the coefficients, including the case $b(0)=0$. In such a case, the \expSch scheme \eqref{eq:IntroNumericScheme0} is still positive but the control of negative moments of $\oX$ escapes our analysis (see Proposition \ref{prop:negative_moments}).  The threshold $\Dt^{-\frac{1}{\alphab-1}}$  is going to infinity with $\Dt$. It can be used in numerical experiments as an indicator to locally decrease and adapt   the time-step until the updated value of   $\oX$ falls below the  threshold. This approach will be the subject of future research.

Convergence of the unstopped process is stated in Theorem \ref{thm:strong_rate_non_stopped},  allowing an additional  domination bound on the derivative $b'$  outside a compact.  This additional condition improves overpass the   \textit{stochastic Gronwall estimation} step for which our technique requires the control of some exponential moments for the \expSch process.

Theorems  \ref{thm:strong_rate_stopped} and \ref{thm:strong_rate_non_stopped} are stated in Section  \ref{sec:strong_error_analysis}. Their proofs are decomposed in several steps that we summarise below. 
\begin{description}

\item[\quad The first step] lies in Section \ref{sec:scheme}, where we analyse some properties of $\oX$ such as positive, negative and exponential moments, as well as rate of convergence of the local error $\|\oX^\gamma_t - \oX^\gamma_{\eta(t)}\|_{L^p(\Omega)}$.  Due to the polynomial dominance of the coefficients, the local error is exploited for different values of $\gamma$  (see Proposition \ref{prop:local_error}). 

It is essential to highlight the importance of controlling exponential moments of $\oX$ in establishing the convergence conditions of the scheme, particularly when the drift $b$ lacks adequate regularity. However, it is important to note that we attain only partial control, as indicated in Proposition \ref{prop:schem_ExpoMoment}, which specifically pertains to the stopped process $\oX_{\cdot \wedge \texp}$.

\item[\quad Discontinuity. ] Both Theorem \ref{thm:strong_rate_stopped} and \ref{thm:strong_rate_non_stopped} allow the drift  $b$ to be  discontinuous.  The convergence analysis  consists of isolating the neighbourhood of this discontinuity and showing that the time that the approximation process spends in this small neighbourhood is itself small.  This technique draws inspiration from the methods employed in \cite{bernardin2009passage} and relies primarily on the occupation time formula and Bernstein's inequality. This particular treatment induces a $\varepsilon$ penalty in the convergence rate. 
Note that the result of \cite{bernardin2009passage} applies originally to the analysis of weak convergence.  
We detail this aspect  in the proof of Theorem \ref{thm:strong_rate_stopped} in Lemma~\ref{lem:disc_contrib}. 

\item[\quad Circumvent the stochastic Gronwall estimation step. ] 
For the $L^p$-error analysis, a prerequisite is the boundedness of the $p$-moments for the solution and for the scheme,  that requires $\ind_{\{\alphab=2\alpha-1\}} \ 2p\leq 1+\tfrac{2B_2}{\diffc^2}$ (see Proposition \ref{prop:XMoments} for $X$ and Lemma \ref{lem:Schememoments} for $\oX$).  
However, the nonlinearity in the error dynamics in the proof of Theorem \ref{thm:strong_rate_stopped} necessitates the application of a {\it{stochastic Gronwall lemma}}, which requires the use of exponential moments.  This issue was already  successfully addressed  in 
 \cite{BBD08} to obtain the 1/2-strong rate of convergence associated with a Symmetrized Euler Scheme  for SDEs with diffusion of type $x^\alpha$, and $\alpha\in[\tfrac12,1)$.  \cite{BBD08}  employs  a time change step that  necessitates the control of exponential moments for  $X^{2\alpha-1}$.   Adapted to our case, this requirement is now on the processes $(X\vee\oX)^{\beta -1}$ and $\frac{1}{X}$. While the exponential moments are  controlled for $X$ with additional conditions identified on the parameter set, as  highlighted  in the first step,  the control over the scheme, such as it is, is only applicable to the  stopped process $\oX_{\cdot \wedge \texp}$.  We also emphasise that the technique of the change of time easily extends to multidimensional cases.

\item[\quad Adding polynomial growth assumption on $b'$ far from zero] allows extending the convergence proof to the unstopped $\oX$ in Theorem \ref{thm:strong_rate_non_stopped} improving the convergence rate and simplifying its proof (see Appendix \ref{sec:proof thm regular}).
 
\end{description}
\paragraph{Asymptotic stability. } We also explore the respective asymptotic behaviours of the trajectories of the process and its approximation by  the \expSch scheme, through a stability analysis.   In Section \ref{sec:stability},  for the polynomial coefficients case, we show that the trajectories of $\oX$ admit an interval of stability points, that converges to the unique stability point of $X$.  

\paragraph{Numerical experiments. }
In Section \ref{sec:num}, we assess the performance of the \expSch scheme in the case of polynomial coefficients. The evaluation of the prototypical case confirms the theoretical rate obtained in Theorem \ref{thm:strong_rate_non_stopped} and shows that the parameter conditions for convergence are largely sufficient. Additionally, we will examine discontinuous scenarios or cases with less regular behaviour covered by Theorem \ref{thm:strong_rate_stopped} through numerical experiments. Lastly, we will illustrate the stability of the numerical scheme established in Proposition \ref{prop:asympt_disc}.

\subsection*{Notation}
Throughout this paper, $T>0$ will refer to an arbitrary finite time horizon,  $C$ will denote a positive constant, possibly depending on the parameters of the dynamic, which may change from line to line. Any process $(Z_t, t\in [0,T])$ will be simply denoted by $Z$.  

For any $a,b\in\mathbb{R}$, $a\vee b$ and $a\wedge b$ denote respectively the maximum and minimum between $a$ and $b$. Given the fixed non-negative  discrete time-step parameter $\Delta t$, we set $t_k=k\Delta t$ for $k\in \mathbb{N}$, $\eta(t) =      \Delta t  \lfloor \frac{t}{\Delta t} \rfloor$, $\dt = t - \eta(t)$, and  $\delta(W_t)  = W_{t} - W_{\eta(t)}$. 

\subsection{Strong wellposedness for the solution to SDE \eqref{eq:IntroSDE}}\label{sec:Analytical}

This study is restricted to the case where $\alpha>1$. The deterministic initial position $x_0$ is  strictly positive.
\medskip

Going back to the proof proposed in \cite{BoJaMa2021},  in Proposition \ref{prop:XMoments} below we extend the sufficient conditions ensuring the strong wellposedness of \eqref{eq:IntroSDE} as well as conditions ensuring moments. 
We first introduce the global growth condition we consider on the drift $b$, contributing to the solution's positiveness and non-explosion: the growth rate around the origin is allowed to go above the derivative of the squared diffusion, so $2\alpha -1$. 
\begin{hypo}{\hspace{-0.12cm}\Blue{\bf  \textrm Polynomial Growths. (\ref{H:polygrowth}).}}
\makeatletter\def\@currentlabel{ {\bf\textrm H}$_{\mbox{\scriptsize\bf\textrm{PolyGrowth}}}$}\makeatother
\label{H:polygrowth}
\begin{itemize}
\item [$(b)$] The drift $b$  satisfies that $ 0\leq b(0) < +\infty$,  and  there exists some non-negative constant $\Lg$ and $\alphab\geq 2\alpha-1$ such that 
\begin{equation}\label{eq:b_LocallyLips}
|b(x)-b(0)|\leq \Lg \ (x^{\alphab} \vee x),~\forall x\in \RR^+. 
\end{equation}

\item [$(\sigma)$] The diffusion function $\sigma$ is locally Lipschitz continuous in $\RR^+$.    Furthermore, there exists some non-negative constant $\diffc>0$ such that
\begin{equation}\label{eq:sigma_LocallyLips}
|\sigma(x)|^2 \leq \diffc^2 x^{2\alpha},~\forall x\in \RR^+,
\end{equation}
and the map $x\mapsto \sigma^2(x)\,\sigma^2(x^{-1})$ is positively bounded on the interval $(0,1)$.
\end{itemize}
\end{hypo}

Next, we introduce some conditions on the regularity of the drift $b$. First, we introduce the notion of a piecewise  Lipschitz  function  for the drift $b$, allowing a finite number of discontinuities points. This kind of hypothesis was recently considered in \cite{mullergronbach2022existence}  to extend the locally Lipschitz condition to piecewise locally Lipschitz condition on $b$ in the context of convergence rate analysis for the tamed Euler scheme.
\begin{defn}\label{def:pieceloclip}
We say a function $f : \RR^+ \rightarrow \RR$  is piecewise Lipschitz continuous, if
there are finitely many points $ 0 <  \pntDisc_1 < \ldots <  \pntDisc_m  < +\infty$ such that $f$  is Locally Lipschitz on each of the intervals $(\pntDisc_k,\pntDisc_{k+1})$, $k=0,\ldots m$, with $\pntDisc_0 = 0$ and  $\pntDisc_{m+1} = +\infty$. 
\end{defn}
The piecewise locally Lipschitz continuity, applied to the drift $b$, is restricted to  $\RR^+$  since  the \expSch  scheme here preserves the positiveness of the solution. Moreover, the already stated hypothesis \ref{H:polygrowth} constrains the form of the local Lipschitz constant as follows: 

\begin{hypo}{\hspace{-0.15cm}\Blue{ {\bf\textrm{ Piecewise Locally Lipschitz condition.}} (\ref{H:pieceloclip}).}}
\makeatletter\def\@currentlabel{{{\bf\textrm H}$_{\mbox{\scriptsize\bf\textrm{Piec\_LocLip}}}$\xspace}}\makeatother
\label{H:pieceloclip}
\begin{itemize}
\item[\bf{(i)}. ]  The drift $b$ is  continuous at point $0$ (where the diffusion term vanishes). Moreover $b$ is piecewise locally Lipschitz continuous, in the sense of Definition \ref{def:pieceloclip}: for any $k=0,\ldots,m$, 
\[|b(x)-b(y)|\leq C_{\pntDisc_k}(1+x^{\alphab-1}\vee y^{\alphab-1})|x-y|,\quad\forall x,y\in(\pntDisc_k,\pntDisc_{k+1}).\]

We denote $\Lips = \displaystyle{\max_{1\leq i\leq m} }C_{\pntDisc_i}$.

\item[\bf{(ii)}. ]  The discontinuity jumps are decreasing: for any $k=1,\ldots,m$, $b(\pntDisc_k^+) - b(\pntDisc_k^-) <0.$ 
\end{itemize}
\end{hypo}
Note that that \ref{H:pieceloclip}-{\bf{(i)}} and  \ref{H:polygrowth} are compatible by taking  the growth constant equal to $\Lips$.

\medskip
Under \ref{H:pieceloclip}--{\bf{(i)}}  existence and uniqueness up to an explosion time holds for the solution of \eqref{eq:IntroSDE}. We introduce a  third condition that moves the explosion time to infinity.
\begin{hypo}{\hspace{-0.15cm}\Blue{\bf  \textrm Controls on $b$ (\ref{H:control}).}}
\makeatletter\def\@currentlabel{ {\bf\textrm H}$_{\mbox{\scriptsize{\bf\textrm{Control\_$b$}}}}$}\makeatother \label{H:control}
There exist some non-negative constants $b(0)$,   $B_1$, $B_2$, and $\alphab\geq2\alpha-1$, such that,
\begin{equation}\label{eq:control_b}
b(x)\leq b(0) + B_1 \ x -B_2 \ x^{\alphab},~\forall x\in \RR^+.
\end{equation}
Moreover, $b$ is one-sided locally Lipschitz: for all $x\neq y$ positive,  there exists a positive constant $\osLip$ such that 
\begin{equation}\label{eq:monotony_b}
\frac{b(x) - b(y)}{x -y} \leq \osLip (1+x^{\alphab-1}\vee y^{\alphab-1}).  
\end{equation}
\end{hypo}
\begin{rem} We would like to emphasise the consequences of \ref{H:control} on the parameter set under consideration. 
\item[(i)] Note that \ref{H:control}, in particular $B_2\geq 0$ is a sufficient condition for the non-explosion of  $X$, but far to be a necessary one. Typically, when $b(x) = b(0) + B_1 \ x -B_2 \ x^{\alpha}$, the Feller test extends  the non-explosion condition to ${B_2} > - \tfrac{\diffc^2}{2}$.

\item[(ii)]The one-sided locally Lipschitz property of $b$ stated in \eqref{eq:monotony_b}   was already a consequence  of \ref{H:pieceloclip}  with $\osLip=\Lips$.   The additional condition \eqref{eq:control_b} now imposes $\Lips \geq B_1 \vee  B_2$. At the same time,  we expect in many cases a much smaller  constant $\osLip$ in condition \eqref{eq:monotony_b}, as again when $b(x) = b(0) + B_1 \ x -B_2 \ x^{2\alpha-1}$,  we have even a  sort of   negative  $\osLip$ with (applying Lemma \ref{lem:equotient_Y_estimation}). 
\begin{equation*}
\frac{b(x) - b(y)}{x -y} \leq  B_1 - B_2  (x^{2\alpha-2} \vee  y^{2\alpha-2}). 
\end{equation*}
\end{rem}

The following proposition  regroups the properties of $X$ solution of \eqref{eq:IntroSDE}. It is derived from a similar proposition presented in \cite{BoJaMa2021}, but extends both the regularity assumptions on $b$ and the theoretical control of the exponential moments of the solution to \eqref{eq:IntroSDE}.  The proof is postponed to Appendix \ref{sec:appendix_wellposedness}. 
\begin{prop}\label{prop:XMoments}
Assume \ref{H:pieceloclip}--{\bf{(i)}} , \ref{H:polygrowth} and \ref{H:control}. Then there exists a unique (strictly) positive strong  solution $X$  to the SDE \eqref{eq:IntroSDE} with the following moment bounds:  \\
${\footnotesize\bullet}$ Negative moments of any order: for all $q>0$, there exists $C_q >0$, depending on $q$, but not on $x_0$,  such that 
\begin{equation}\label{eq:moments_neq}
\sup_{t\in[0,T]}\EE\big[X_t^{-q}\big]\leq~C_q(1+ x_0^{-q}).
\end{equation}
${\footnotesize\bullet}$ Some positive  moments:  for all exponent $p>0$ such that $\ind_{\{2\alpha-1\}}(\alphab) \ p\leq \frac{1}{2}+\tfrac{B_2}{\diffc^2}$, there exists $C_p >0$, depending on $p$, but not on $x_0$,  such that  
\begin{equation}\label{eq:moments_cont}
\sup_{t\in[0,T]}\EE\big[X_t^{2p}\big]\leq~C_p(1+ x_0^{2p}).
\end{equation}
${\footnotesize\bullet}$  Some exponential moments:   assume $(b(0),\alpha)  \notin (0,\frac{\diffc^2}{2})\times[1,\frac{3}{2}]$. 
\begin{equation}\label{eq:ExpMomentX}
\begin{aligned}
& \mbox{ When  $b(0)\neq 0$, ~for all $\nu\in \RR$,  }  
\quad \sup_{t\in[0,T]}\EE\Big[\exp\left\{\nu\int_0^t\dfrac{ds}{X_s}\right\}\Big] < +\infty.  \\
& \mbox{For all $\mu  < B_2$, ~ for all $v$  such that $3\upsilon \left(3\upsilon - 1\right)\frac{\diffc^2}2< B_2$,}  \\
&\quad  \sup_{t\in[0,T]}\EE\left[\exp\left\{\mu\int_0^tX_s^{\alphab-1}ds\right\}\right]+  \sup_{t\in[0,T]} \EE\left[\exp\left\{ -\upsilon\int_0^t\frac{b(X_s)}{X_s}ds\right\}\right] <+\infty.
\end{aligned}
\end{equation}
\end{prop}
 
\begin{rem}\label{rem:exponentialPrototype}
\item[(i)] It is worth notice that the combination of  \ref{H:control}  with the parameters condition $\ind_{\{2\alpha-1\}}(\alphab) \ p\leq \frac{1}{2}+\tfrac{B_2}{\diffc^2}$, to get $p$-moment bounds could be written differently. 
For example, recently   \cite{mullergronbach2022existence} or  \cite{kelly2022adaptive})  consider the following combination  of $xb(x)$ and quadratic variation term: 
\begin{equation*}
2xb(x)  + (p -1)\sigma^2(x) \leq c (1 + x^2).
\end{equation*}
However, specifying the coefficients in \ref{H:control}  helps keep track of the assumptions to control the sufficient moments for any $L^p$-norm  and delineates the role of each parameter in this control. More precisely, the non-negative constants $B_1$ and $b(0)$ in the SDE \eqref{eq:IntroSDE} serve to propel the solution away from zero. Conversely, the constant $B_2$ counteracts the solution's growth induced by the diffusion term, manifesting a mean-reverting effect.
This reasoning aligns with condition (3.14) of \cite[Ch. 4]{mao2007stochastic}, and the conditions imposed on the parameters to control the moments are consistent with those of Theorem 4.1 of \cite[Ch. 4]{mao2007stochastic}.

\item[(ii)] 
Using the Lenglart Inequality (see Lemma \ref{lem:Lenglart_sharp}), it is possible to  directly derive the moment estimation  \eqref{eq:moments_cont} with the $\sup$ inside the expectation. Then the condition leading to \eqref{eq:moments_cont} is strengthened with $\ind_{\{2\alpha-1\}}(\alphab) \ \frac{p}{\epsilon} \leq \frac{1}{2}+\frac{B_2}{\diffc^2}$, for some $\epsilon\in(0,1)$. For example, considering $\epsilon = \tfrac{p}{p+1}$, the condition is replaced by  $\ind_{\{2\alpha-1\}}(\alphab) \ (p+1)\leq \frac{1}{2}+\frac{B_2}{\diffc^2}$, and  
\begin{equation*}
\EE\big[\sup_{t\in[0,T]} X_t^{2p}\big]\leq~  (p+1)^{\frac{2p+1}{(p+1)^2}} p^{-\frac{p}{(p+1)^2}} C_{p+1}(1+ x_0^{2(p+1)}). 
\end{equation*}
\end{rem}

\section{The exponential scheme, moment bounds  and local error estimations}\label{sec:scheme}

The scheme  $(\overline{X}_{t_{n}}; n\geq1)$, defined in \eqref{eq:IntroNumericScheme0} admits the following continuous version
\begin{equation}\label{eq:DecomposednumericScheme}
\overline{X}_t=b(0){\dt}+\overline{X}_{\eta(t)}\exp\Big\{\tfrac{\sigma(\overline{X}_{\eta(t)})}{\oX_{\eta(t)}}(W_t-W_{\eta(t)})+\Big(\tfrac{b\left(\overline{X}_{\eta(t)}\right)-b(0)}{\overline{X}_{\eta(t)}}-\tfrac{1}2\tfrac{\sigma^2(\overline{X}_{\eta(t)})}{\overline{X}_{\eta(t)}^2} \Big){\dt}\Big\},
\end{equation}
driven by the SDE
\begin{equation}\label{eq:ContExpscheme}
	d\overline{X}_t=\left(\overline{X}_t-b(0){\dt}\right)\Big(\tfrac{b\left(\overline{X}_{\eta(t)}\right)-b(0)}{\overline{X}_{\eta(t)}}dt+\tfrac{\sigma(\overline{X}_{\eta(t)})}{\overline{X}_{\eta(t)}}dW_t\Big)+b(0) dt.
\end{equation}
From  \eqref{eq:DecomposednumericScheme},  almost surely, 
\begin{equation}\label{eq:scheme_minoration}
0\leq \left(1-\frac{{b(0)}\delta(t)}{\oX_t}\right)\leq 1 \  \text{ or equivalently, } \  ~ 0\leq {b(0) \dt} \leq {\oX_t}.
\end{equation}
For the \expSch process $\overline{X}$, we bound the same order of  $2p$\,th-moments than for $X$, with the same sufficient condition $0 \leq  \ind_{\{2\alpha-1\}}(\alphab)\, p \leq \frac{1}{2} +\frac{B_2}{\diffc^2}$: 
\begin{lem}\label{lem:Schememoments}
Assume \ref{H:control} and \ref{H:polygrowth}-($\sigma$). For all exponents $p>0$ such that $\ind_{\{2\alpha-1\}}(\alphab)\, p \leq \frac{1}{2} +\frac{B_2}{\diffc^2}$, there exists a non-negative constant $C_p$, depending on $p$ but not on $x_0>0$, such that
\[
\sup_{t\in[0,T]}\EE\big[{\overline{X}_t^{2p}}\big]\leq C_p(1+x_0^{2p}).
\]
\end{lem}
\begin{proof}
Considering $p>0$, we apply the It\^o formula to $\oX_t^{2p}$ (with a localization argument in a compact set of $\mathbb{R}^{+*}$ omitted here or simplicity): 
\begin{equation*}
\begin{aligned}
\EE[\overline{X}_{t}^{2p}]
& = x_0^{2p}+2pb(0)\EE\big[\int_0^{t}\overline{X}_s^{2p-1}ds\big]\\
& +2p\EE\Big[\int_0^{t}\overline{X}_s^{2p-2}(\overline{X}_s-b(0){\ds})\Big\{\overline{X}_s\tfrac{b(\overline{X}_{\eta(s)})-b(0)}{\overline{X}_{\eta(s)}}+\frac{(2p-1)}2 \left(\overline{X}_s-b(0){\ds}\right)
\tfrac{\sigma^2(\overline{X}_{\eta(s)})}{\overline{X}_{\eta(s)}^2}\Big\}ds
\Big],
\end{aligned}
\end{equation*}
and thus, from \ref{H:control}, \ref{H:polygrowth}-($\sigma$) and inequality \eqref{eq:scheme_minoration}, we get
\begin{equation*}
\begin{aligned}
\EE[\overline{X}_{t}^{2p}]
\leq  & x_0^{2p}+2pb(0)\EE\big[\int_0^{t}\overline{X}_s^{2p-1}ds\big]\\
& +2p\EE\Big[\int_0^{t}\overline{X}_s^{2p-2}(\overline{X}_s-b(0){\ds})\Big\{B_1\oX_s  - B_2\oX_s\oX_{\eta(s)}^{\alphab-1}+(2p-1)\frac{\diffc^2}{2}\left(\overline{X}_s-b(0){\ds}\right)
\overline{X}_{\eta(s)}^{2(\alpha-1)}\Big\}ds
\Big]\\
\leq &  x_0^{2p}+2pb(0)\EE\big[\int_0^{t}\overline{X}_s^{2p-1}ds\big]+ 2B_1p\EE\Big[\int_0^{t}\overline{X}_s^{2p}ds\Big]\\
& +2p\EE\Big[\int_0^{t}\overline{X}_s^{2p-1}(\overline{X}_s-b(0){\ds})\Big\{- B_2\oX_{\eta(s)}^{\alphab-1}+(2p-1)\tfrac{\diffc^2}2\overline{X}_{\eta(s)}^{2(\alpha-1)}\Big\}ds
\Big].
\end{aligned}
\end{equation*}
When $\alphab = 2\alpha-1$,  the last term above is negative provided that $(2p-1)\leq 2B_2/\diffc^2$. Otherwise, if $\alphab > 2\alpha-1$, the map $z\mapsto (2p-1)\tfrac{\diffc^2}2 z^{2\alpha-2} - B_2 z^{\alphab-1}$ is bounded from above. In  the two cases, from Young inequality, there exists a constant $C$, independent of $t$, such that
\begin{equation*}
\EE[\overline{X}_{t}^{2p}]\leq C(1+x_0^{2p})+C\mbox{$\int_0^{t}$} \EE[\overline{X}_{s}^{2p}] ds.
\end{equation*}
The proof ends by applying Gronwall's inequality.
\end{proof}

\subsection{Local error and negative moments for the \expSch scheme} 

We analyse the local error of the \expSch scheme $\|\oX_t^{\gamma} - \oX_{\eta(t)}^{\gamma}\|_{L^{2p}(\Omega)}$ for some exponent $\gamma > 0$. 
Interesting values of $\gamma$ are  indeed $\gamma=1$, and  the exponents appearing in  the  It\^o formula applied to $\oX$, typically $\gamma = \alphab-1$ and $\gamma=\alpha-1$. The convergence rate of the local error for $\gamma=1$ is stated in Proposition \ref{prop:local_error}  below. It is  expected  that the local error bound requires sufficient control on positive moments of $\oX$. 

\begin{prop}\label{prop:local_error}
Assume \ref{H:polygrowth} and \ref{H:control}. For  all $p >0$, integer exponent such that 
\begin{equation}\label{eq:assumption_inlocalerror_0}
\ind_{\{2\alpha-1\}}(\alphab) \times \  p \ ( 2\alpha-1) \leq \frac12+\frac{B_2}{\diffc^2}. 
\end{equation}
Then, there exists a non-negative constant $C_p$, independent of $\Dt$, such that
\begin{equation}\label{eq:local_onehalp} \sup_{t\in[0,T]}   \|\oX_t - \oX_{\eta(t)}\|_{L^{2p}(\Omega)}  \leq C_p  \ \Dt^\frac{1}{2}.
\end{equation}
\end{prop}
In order to prove the strong convergence of the exponential scheme we use the local error approximation stated in Proposition \ref{prop:local_error}. However, we show below a more general lemma that is useful for proving the finiteness of the negative moments of the scheme: 
\begin{lem}\label{lem:local_error_prev} Assume \ref{H:polygrowth} and \ref{H:control}. 
For $\gamma >  0$, for all  $p > 0$ such that $\ind_{\{2\alpha-1\}}(\alphab)\, p \,(\gamma +2(\alpha-1) )\leq \frac12+\frac{B_2}{\diffc^2}$, there exists a non-negative constant $C_p$, independent of $\Dt$, such that
\begin{equation}\label{eq:lem_first_local_error_estimate}
\sup_{t\in[0,T]} \|\oX_t^{\gamma} - \oX_{\eta(t)}^{\gamma}\|_{L^{2p}(\Omega)} \leq C_p \left( \delta(t)^\frac{1}{2} \ \ind_{\{0\}}(b(0)) \  +  \	\delta(t)^{\frac{1}{2} \wedge \gamma} \ \ind_{(0,+\infty)}(b(0))   \right).
\end{equation}
Let $\Dt < 1$. Then, for any $0\leq \epsilon < \frac12 \wedge \gamma$ and all  $p > 0$ such that $\ind_{\{2\alpha-1\}}(\alphab)\, (p + \Dt^{2p\epsilon}) \,(\gamma +2(\alpha-1) )\leq \frac12+\frac{B_2}{\diffc^2}$, we have
\begin{equation}\label{eq:lem_first_local_error_estimate_sup_inside}
\big\|\sup_{t\in[0,T]} |\oX_t^{\gamma} - \oX_{\eta(t)}^{\gamma}|\big\|_{L^{2p}(\Omega)} \leq C_{p+\Dt^{2p\epsilon}}  \left( \delta(t)^{\frac{1}{2}-\epsilon} \ \ind_{\{0\}}(b(0)) \  + \ \delta(t)^{ (\frac{1}{2} \wedge \gamma) -\epsilon} \ \ind_{(0,+\infty)}(b(0))   \right).
\end{equation}	
\end{lem}

\begin{proof}
Applying the It\^o formula to $\overline{X}^\gamma_t$,  we get
\begin{equation*}
\begin{aligned}
d\overline{X}^\gamma_t&= \gamma\oX^\gamma_t \left(1-\tfrac{b(0)\delta(t)}{\oX_t}\right)\left[\tfrac{b(\oX_{\eta(t)})-b(0)}{\oX_{\eta(t)}} + (\gamma -1) \frac{1}{2}\left(1-\tfrac{b(0)\delta(t)}{\oX_t}\right) \tfrac{\sigma^2(\oX_{\eta(t)})}{\oX_{\eta(t)}^2}  \right] dt +\gamma b(0)\oX_t^{\gamma-1}dt\\
&\qquad + \gamma\oX^\gamma_t \left(1-\tfrac{b(0)\delta(t)}{\oX_t}\right) \tfrac{\sigma(\oX_{\eta(t)})}{\oX_{\eta(t)}}dW_t. 
\end{aligned}
\end{equation*}
So for $t$ such that  $\dt >0$, 
\begin{equation}\label{eq:e1_e2}
\begin{aligned}
& \|\oX_t^{\gamma} - \oX_{\eta(t)}^{\gamma}\|^{2p}_{L^{2p}(\Omega)} \\
& =  \EE\left[\left\{ \gamma\int_{\eta(t)}^t\oX^\gamma_s \left(1-\tfrac{b(0)\delta(s)}{\oX_s}\right)\left[\tfrac{b(\oX_{\eta(s)})-b(0)}{\oX_{\eta(s)}}+ (\gamma -1) \tfrac{1}{2}\left(1-\tfrac{b(0)\delta(s)}{\oX_s}\right) \tfrac{\sigma^2(\oX_{\eta(s)})}{\oX_{\eta(s)}^2} 
\right]ds\right.\right.\\
&\qquad  \left.\left. +\gamma b(0)\int_{\eta(t)}^t\oX_s^{\gamma-1}ds+\gamma\int_{\eta(t)}^t\oX^\gamma_s \left(1-\tfrac{b(0)\delta(s)}{\oX_s}\right) \tfrac{\sigma(\oX_{\eta(s)})}{\oX_{\eta(s)}}dW_s\right\}^{2p}\right]\\ 
&\quad  \leq C(p)\gamma^{2p} \ \EE\left[\left\{\int_{\eta(t)}^t\oX^\gamma_s \left(1-\tfrac{b(0)\delta(s)}{\oX_s}\right)\left[\tfrac{b(\oX_{\eta(s)})-b(0)}{\oX_{\eta(s)}}+ (\gamma -1) \tfrac{1}{2}\left(1-\tfrac{b(0)\delta(s)}{\oX_s}\right) \tfrac{\sigma^2(\oX_{\eta(s)})}{\oX_{\eta(s)}^2} \right]ds + b(0)\oX_s^{\gamma-1}ds\right\}^{2p}\right]\\ 
& \qquad + C(p)\gamma^{2p} \ \EE\left[\left(\int_{\eta(t)}^t\oX^\gamma_s \left(1-\tfrac{b(0)\delta(s)}{\oX_s}\right) \tfrac{\sigma(\oX_{\eta(s)})}{\oX_{\eta(s)}}dW_s\right)^{2p}\right]\\
&\qquad  := E_1(t) +E_2(t),
\end{aligned}
\end{equation}
with $E_1$ and $E_2$ separating the Lebesgue from the  Itô integrals.  For $E_1$, using \eqref{eq:scheme_minoration} and \ref{H:polygrowth}, we get
\begin{equation*}
E_1(t) \leq C(p)\gamma^{2p} ~\EE\left[\left(\int_{\eta(t)}^t\oX^\gamma_s \left[\Lg(\oX_{\eta(s)}^{\alphab-1}\vee 1)+|\gamma-1|\tfrac{\diffc^2}2\, \oX_{\eta(s)}^{2(\alpha-1)} \right] ds + b(0)\int_{\eta(t)}^t \oX_s^{\gamma-1}ds\right)^{2p}\right]. 
\end{equation*}
	
\noindent $\bullet$ When $\gamma \geq 1$,   all the terms involved above are integrable for $B_2$ large enough,  and  by Jensen inequality,  
\begin{equation*}
\begin{aligned}
E_1(t)
&\leq C(p)\dt^{2p-1}\EE\left[\int_{\eta(t)}^t\oX^{2p\gamma}_s \left[\Lg (\oX_{\eta(s)}^{\alphab-1}\vee 1) +(\gamma-1)\tfrac{\diffc^2}2\, \oX_{\eta(s)}^{2(\alpha-1)}  \right]^{2p}ds\right]+ C(p)~\dt^{2p}\sup_{0\leq t\leq T}\EE\left[\oX_t^{2p(\gamma-1)}\right]\\
&\leq  C(p)~\dt^{2p}\left\{1+\sup_{0\leq t\leq T}\EE\left[\oX_t^{2p( \gamma + \alphab-1)}\right]+ b(0) \sup_{0\leq t\leq T}\EE\left[\oX_t^{2p(\gamma-1)}\right]\right\},
\end{aligned}
\end{equation*}
since $\alphab - 1\geq 2(\alpha-1)>0$.  Then \eqref{eq:lem_first_local_error_estimate} holds when $(B_2, \alphab)$ are such that $\sup_{0\leq t\leq T}\EE[\oX^{2 p(\gamma+\alphab-1)}_t]$ is bounded. Obviously,   when $b(0)=0$, the last  upper-bound for $E_1$ holds for all $\gamma > 0$.
\medskip
	
\noindent $\bullet$ When $0 < \gamma < 1$ and  $b(0)>0$, using \eqref{eq:scheme_minoration}
\begin{equation*}
b(0)\int_{\eta(t)}^t \oX_s^{\gamma-1}ds \leq b^\gamma(0)\int_{\eta(t)}^t (s - \eta(t)) ^{\gamma-1}ds = b^\gamma(0)\int_0^{\delta(t)} s^{\gamma-1}ds = \frac{b^\gamma(0)}{\gamma} \delta(t)^\gamma.
\end{equation*}
Then
\begin{equation*}
\begin{aligned}
E_1(t)
&\leq C(p)\dt^{2p-1}\EE\left[\int_{\eta(t)}^t\oX^{2p\gamma}_s \left[\Lg (\oX_{\eta(s)}^{\alphab-1}\vee 1) +|\gamma-1|\tfrac{\diffc^2}2\, \oX_{\eta(s)}^{2(\alpha-1)}  \right]^{2p}ds\right]+ C(p)~\dt^{2p\gamma}\\
&\leq C(p)~\dt^{2p}\left\{1+\sup_{0\leq t\leq T}\EE\left[\oX_t^{2p( \gamma + \alphab-1)}\right]\right\} +  C(p)~\dt^{2p\gamma}. 
\end{aligned}
\end{equation*}
	
For the bound of the term  $E_2(t)$ in \eqref{eq:e1_e2}, we use usual arguments from Itô calculus, with some updated  constant $C(p)$,  and using again   \eqref{eq:scheme_minoration} with \ref{H:polygrowth},  obtaining
\begin{equation*}
E_2(t)
\leq C(p)\ \gamma^{2p} \  \diffc^{2p}\ \dt^{p}\left(1+\sup_{0\leq t\leq T}\EE[\oX^{2 p(\gamma+\alpha-1)}_t]\right).
\end{equation*}
Finally, from Lemma \ref{lem:Schememoments}, a sufficient condition to control the highest moment  $\sup_{0\leq t\leq T}\EE[\oX^{2 p(\gamma+\alphab-1)}_t]$ when $\alphab=2\alpha-1$ is  that $2 p(\gamma+2(\alpha-1))\leq 1+\frac{2B_2}{\diffc^2}$, whereas  no conditions are needed in the case  $\alphab>2\alpha-1$. \medskip

To prove the estimation \eqref{eq:lem_first_local_error_estimate_sup_inside}, we consider $\epsilon>0$ and the estimation \eqref{eq:lem_first_local_error_estimate} with $p' = p + \Dt^{2 p \epsilon}$. Then, assuming the parameters satisfy $2 (p+\Dt^{2p\epsilon})(\gamma+2(\alpha-1))\leq 1+\frac{2B_2}{\diffc^2}$, we apply Lenglart's inequality in Lemma \ref{lem:Lenglart_sharp}, with  $r = \frac{p}{p+\Dt^{2p\epsilon}}$, obtaining
\begin{equation*}
\begin{aligned}
\EE\left[\sup_{t\in[0,T]} |\oX_t^{\gamma} - \oX_{\eta(t)}^{\gamma}|^{2p}\ \right] &= \EE\left[\left(\sup_{t\in[0,T]} |\oX_t^{\gamma} - \oX_{\eta(t)}^{\gamma}|^{2(p+\Dt^{2p\epsilon})}\right)^r\ \right]\\
& \leq C_{p+\Dt^{2p\epsilon}} C_r \left( \Dt^{p} \ \ind_{\{0\}}(b(0)) \  +  \	\Dt^{p \wedge 2p\gamma} \ \ind_{(0,+\infty)}(b(0))   \right),
\end{aligned}
\end{equation*}
with $C_r = o(\Dt^{-2p\epsilon})$. Therefore, 
\begin{equation*}
\big\|\sup_{t\in[0,T]} |\oX_t^{\gamma} - \oX_{\eta(t)}^{\gamma}|\big\|_{L^{2p}(\Omega)} \leq C_{p+\Dt^{2p\epsilon}} \left( \Dt^{\frac{1}{2}-\epsilon} \ \ind_{\{0\}}(b(0)) \  +  \	\Dt^{ (\frac{1}{2} \wedge \gamma) - \epsilon} \ \ind_{(0,+\infty)}(b(0))   \right),
\end{equation*}	
under assumption $2 (p+\Dt^{2p\epsilon})(\gamma+2(\alpha-1))\leq 1+\frac{2B_2}{\diffc^2}$, obtaining  \eqref{eq:lem_first_local_error_estimate_sup_inside}. 
\end{proof} 
With this lemma, we deduce some bounds on the negative moments of a stopped version of the scheme.
 More precisely, we introduce the  stopping time
\begin{equation}\label{eq:texp} 
\texp = \inf\{ s\geq 0; \oX_{\eta(s)} > \Dt^{-\frac{1}{\alphab - 1}} \},
\end{equation}
that goes to infinity, when $\Dt$ is going to zero. Also, we observe that $\eta(\texp) = \texp$.  

\begin{prop}\label{prop:negative_moments} 
Assume \ref{H:polygrowth} and   \ref{H:control}. Then,  for  $\texp$ defined in \eqref{eq:texp}, for any $\kappa >0$
\begin{equation}\label{eq:neg_moments_scheme_stopped}
\sup_{t\leq T} \EE[(\overline{X}_{t\wedge \texp}^{-\kappa}) ]  \leq \sup_{t\leq T} \EE[(\overline{X}_{\cdot} - b(0)\delta(\cdot))_{t\wedge \texp}^{-\kappa} ]   \leq C(\kappa,\alpha,T). 
\end{equation}
In addition, assume $b(0)>0$, then for any $\kappa >0$, such that,
\begin{equation*}
\ind_{\{2\alpha -1\}}(\beta) \times \Big(\ind_{(1,\frac{5}{4}]}(\alpha)+4(\alpha-1)\ind_{(\frac{5}{4},+\infty)}(\alpha)\Big) \times \left(\kappa\vee(2\alpha-2) \right) 
\leq \frac{1}{2} + \frac{B_2}{\diffc^2}, 
\end{equation*}
\begin{equation}\label{eq:neg_moments_scheme}
\sup_{t\leq T} \EE[\overline{X}^{-\kappa}_t] \  \leq C(\kappa,\alpha,T). 
\end{equation}
\end{prop}
\begin{proof}
Applying the It\^o formula to $\overline{X}_t^{-\kappa}$, and from hypothesis \ref{H:polygrowth} and \eqref{eq:scheme_minoration}: 
\begin{equation*}
\begin{aligned}
d\overline{X}^{-\kappa}_t&= -\kappa\oX^{-\kappa}_t \left(1-\tfrac{{b(0)}\delta(t)}{\oX_t}\right)\left[ \tfrac{b(\oX_{\eta(t)})-b(0)}{\oX_{\eta(t)}}  - (\kappa+1) \frac{1}{2}\left(1-\tfrac{{b(0)}\delta(t)}{\oX_t}\right)  \tfrac{\sigma^2(\oX_{\eta(t)})}{\oX_{\eta(t)}^2}  \right]dt -\kappa {b(0)}\oX_t^{-(\kappa+1)}dt\\
&\qquad - \kappa\oX^{-\kappa}_t \left(1-\tfrac{{b(0)}\delta(t)}{\oX_t}\right)\tfrac{\sigma(\oX_{\eta(t)})}{\oX_{\eta(t)}}dW_t \\
&\leq -\kappa\oX^{-\kappa}_t \left[ -\Lg(\oX_{\eta(t)}^{\alphab-1} \vee 1)  -(\kappa+1) \tfrac{\diffc^2}{2} \oX_{\eta(t)}^{2(\alpha-1)}  \right]dt -\kappa {b(0)}\oX_t^{-(\kappa+1)}dt - \kappa\oX^{-\kappa}_t \left(1-\tfrac{{b(0)}\delta(t)}{\oX_t}\right) \tfrac{\sigma(\oX_{\eta(t)})}{\oX_{\eta(t)}}dW_t. 
\end{aligned}
\end{equation*}
Taking the expectation, (with a stopping time if needed to handle the local martingale), and keeping only the positive terms, we introduce the local error as follows
\begin{equation*}
\begin{aligned}
d\ \EE[\overline{X}^{-\kappa}_t] &\leq  \kappa \Big( (\kappa+1) \frac{\diffc^2}{2}\EE[\oX^{-\kappa}_t\oX_{\eta(t)}^{2(\alpha-1)} ] + \Lg\,\EE[\oX^{-\kappa}_t(\oX_{\eta(t)}^{\alphab-1}\vee 1) ]\Big)  \ dt \\
&\leq  \kappa \Big( (\kappa+1) \frac{\diffc^2}{2} + (\Lg+(\kappa+1) \tfrac{\diffc^2}{2})\,\EE[\oX^{-\kappa}_t(\oX_{\eta(t)}^{\alphab-1}\vee1) ]\Big)  \ dt \\
&\leq   C(\kappa) \EE[\oX^{-\kappa + \alphab-1}_t]  \ dt + C(\kappa)  \EE[\oX^{-\kappa}_t  |\oX_{\eta(t)}^{\alphab-1}  - \oX_{t}^{\alphab-1}| ]  \ dt  +C(\kappa) (1+\EE[\oX_t^{-\kappa}])dt,
\end{aligned}
\end{equation*}
with $C(\kappa) :=\kappa \big( (\kappa+1) \frac{\diffc^2}{2} + \Lg\big)$.  Assuming first $\kappa \geq \alphab-1$, Jensen inequality in the first term  gives us: 
\[\EE[\oX^{-\kappa + \alphab-1}_t]  \leq \left( \EE[ 
{\oX_t}^{-\kappa} ]\right)^{\frac{\kappa - (\alphab-1)}{\kappa}}.\]
For the second term we use H\"older inequality and \eqref{eq:scheme_minoration}:  assuming $b(0) > 0$, let $\ell<\kappa$, then for 
\begin{equation*}
\EE[\oX^{-\kappa}_t  (\oX_{\eta(t)}^{\alphab-1}  - \oX_{t}^{\alphab-1}) ]  \leq \tfrac{1}{(b(0) \delta(t))^\ell} \EE[\oX^{-\kappa+\ell}_t  |\oX_{\eta(t)}^{\alphab-1}  - \oX_{t}^{\alphab-1}|]  \leq \tfrac{1}{(b(0) \delta(t))^\ell} \EE[\oX^{-\kappa}_t] ^{\frac{\kappa - \ell}{\kappa}} \ \EE[|\oX_{\eta(t)}^{\alphab-1}  - \oX_{t}^{\alphab-1}|^{\frac{\kappa}{\ell}} ]^{\frac{\ell}{\kappa} }.
\end{equation*}
So we can use the previous local error bound, provided that we control the $\tfrac{\kappa}{\ell}4(\alpha-1)$th-moment   of $\oX$ (requiring additional parameter control when $\alphab = 2\alpha-1$, see Lemma  \ref{lem:local_error_prev}), obtaining
\[ \EE[|\oX_{\eta(t)}^{\alphab-1}  - \oX_{t}^{\alphab-1}|^{\frac{\kappa}{\ell}} ]^{\frac{\ell}{\kappa} }  \leq C(\frac{\kappa}{\ell}) \dt^{\frac{1}{2} \wedge (\alphab-1)}.\]
Choosing $\ell = \tfrac{1}{2} \wedge (\alphab-1)$, then for any $\Delta t\leq 1$,
\begin{equation*}
\EE[\oX^{-\kappa}_t  |\oX_{\eta(t)}^{\alphab-1}  - \oX_{t}^{\alphab-1}| ]  \leq  C(\kappa,\alpha) \EE[\oX^{-\kappa}_t] ^{\frac{\kappa - \ell}{\kappa}}.
\end{equation*}
Since $\frac{\kappa - (\alphab-1)}{\kappa} \vee \frac{\kappa - \ell}{\kappa} \leq 1$, we obtain the following bound
\begin{equation*}
d\ \EE[\overline{X}^{-\kappa}_t] \leq  C(\kappa,\alpha) \left[ 1 + 1\vee \EE[\overline{X}^{-\kappa}_t]  \right]dt,
\end{equation*}
from which, applying Gronwall inequality we get the desired inequality
\[\sup_{t\leq T} \EE[\overline{X}^{-\kappa}_t]\leq C(x_0,\kappa,\alpha,T). \]	
To guarantee the estimation, when $\alphab=2\alpha-1$, and according to  the choice $\ell=\tfrac{1}{2} \wedge (\alphab-1)$, we need to impose that  $\frac{\kappa}{\frac12\wedge(2\alpha-2)}(2\alpha-2)\leq \frac{B_2}{\diffc^2} + \frac12$. 

Finally, when $\kappa<\alphab-1$, the first term is a positive moment, and  using the bound on the second term, under the sufficient condition that  $\frac{1}{\frac12\wedge(2\alpha-2)}(2\alpha-2)^2\leq \frac{B_2}{\diffc^2} + \frac12$ we have the same conclusion. 
Putting together the required conditions leads to  the following 
$$\kappa\vee(2\alpha-2)\left(\ind_{(1,\frac{5}{4}]}(\alpha)+4(\alpha-1)\ind_{(\frac{5}{4},+\infty)}(\alpha)\right)\leq \frac12 + \frac{B_2}{\diffc^2}. $$

We prove now the estimation  \eqref{eq:neg_moments_scheme_stopped} for any negative moment of the stopped \expSch scheme.  Assume now, $b(0)\geq 0$. 
For $\kappa>0$, we apply Itô formula to the process 
\begin{equation*}
d (\overline{X}_t - b(0) \dt) =\left(\overline{X}_t-b(0){\dt}\right)\Big(\tfrac{b\left(\overline{X}_{\eta(t)}\right)-b(0)}{\overline{X}_{\eta(t)}}dt+\tfrac{\sigma(\oX_{\eta(t)})}{\oX_{\eta(t)}}dW_t\Big).
\end{equation*}
\begin{equation*}
d (\overline{X}_t - b(0) \dt)^{-\kappa} = -\kappa \left(\overline{X}_t-b(0){\dt}\right)^{-\kappa}\Big(\tfrac{b\left(\overline{X}_{\eta(t)}\right)-b(0)}{\overline{X}_{\eta(t)}}\ dt+\tfrac{\sigma(\oX_{\eta(t)})}{\oX_{\eta(t)}} \ dW_t  - \tfrac12(\kappa +1)\tfrac{\sigma^2(\oX_{\eta(t)})}{\oX_{\eta(t)}^2} \ dt \Big).
\end{equation*}
We rewrite $\overline{X}_t - b(0) \dt =  \overline{X}_{\eta(t)} \ Z(\dt)$, denoting 
\begin{equation*}
Z(\dt) :=  \exp\Big\{\tfrac{\sigma(\overline{X}_{\eta(t)})}{\oX_{\eta(t)}}(W_t-W_{\eta(t)}) + \left(\tfrac{b(\oX_{\eta(t)}) - b(0)}{\oX_{\eta(t)}}-\tfrac{1}2 \tfrac{\sigma^2(\overline{X}_{\eta(t)})}{\oX_{\eta(t)}^2}\right){\dt}\Big\}.
\end{equation*}
From \ref{H:polygrowth}, observe that, for any $\Dt\leq 1$:
\begin{equation*}
\begin{aligned}
\EE\left[ \ind_{\{  \oX_{\eta(t)} \leq \Dt^{-\frac{1}{\alphab-1}} \}}   Z^{-\kappa}(\dt) \Big|\mathcal{F}_{\eta(t)} \right] 
&\leq 
\ind_{\{  \oX_{\eta(t)} \leq \Dt^{-\frac{1}{\alphab-1}} \}} \exp\{L_G(\Dt^{-1}\vee1)\Dt\kappa\}~  \exp\Big\{\tfrac{\diffc^2}2 \kappa(\kappa+1)\overline{X}_{\eta(t)}^{2(\alpha-1)}{\dt}\Big\}\\
&\leq \ind_{\{  \oX_{\eta(t)} \leq \Dt^{-\frac{1}{\alphab-1}} \}}~ \exp\Big\{L_G\kappa + \tfrac{\diffc^2}2 \kappa(\kappa+1)\Big\}. 
\end{aligned}
\end{equation*}
For $\texp = \inf\{ s\geq 0; \oX_{\eta{(s)}} > \Dt^{-\frac{1}{\alphab-1}} \}$, we  consider the stopped process $(\overline{X}^{-\kappa}_{\cdot \wedge \texp} - b(0)\delta(\cdot\wedge \texp))$.  From  \ref{H:polygrowth} and the previous estimation we obtain
\begin{equation*}
\begin{aligned}
& d\ \EE[(\overline{X}_\cdot - b(0)\delta(\cdot))^{-\kappa}_{t \wedge \texp} ]\\
& \leq \kappa \EE\left[\ind_{\{  \oX_{\eta(t)} \leq \Dt^{-\frac{1}{\alphab-1}} \}} Z^{-\kappa}(\dt)\oX_{\eta(t)}^{-\kappa}\Big(L_G\left(\overline{X}_{\eta(t)}^{\alphab-1}\vee1\right)  + \tfrac{\diffc^2}2(\kappa +1)\oX_{\eta(t)}^{2(\alpha-1)} \Big)\right]\ dt\\
& \quad = \kappa \EE\left[\EE\left[\ind_{\{  \oX_{\eta(t)} \leq \Dt^{-\frac{1}{\alphab-1}} \}} Z^{-\kappa}(\dt)\Big|\Ff_{\eta(t)}\right]\oX_{\eta(t)}^{-\kappa}\Big(L_G\left(\overline{X}_{\eta(t)}^{\alphab-1}\vee1\right)  + \tfrac{\diffc^2}2(\kappa +1)\oX_{\eta(t)}^{2(\alpha-1)} \Big)\right]\ dt\\
& \qquad \leq C \EE\left[\ind_{\{  \oX_{\eta(t)} \leq \Dt^{-\frac{1}{\alphab-1}} \}}\big( \oX_{\eta(t)}^{-\kappa}+ \oX_{\eta(t)}^{-\kappa + \alphab-1} + \oX_{\eta(t)}^{-\kappa + 2(\alpha-1)}\big)\right]\ dt. 
\end{aligned}
\end{equation*}
Thus, for any $\kappa > \alphab-1 \geq  2 (\alpha-1)$, (otherwise we can use H\"older inequality):
\begin{equation*}
\begin{aligned}
\sup_{0\leq s\leq t}\EE[(\overline{X}_\cdot - b(0) \delta(\cdot))^{-\kappa}_{s \wedge \texp} ]
&\leq  C\ \left(x_0^{-\kappa} + \int_0^t\sup_{0\leq r\leq s} \EE\left[\overline{X}^{-\kappa}_{r \wedge \texp}\right] ds\right)\\
& \leq  C\ \left(x_0^{-\kappa} + \int_0^t \sup_{0\leq r\leq s}\EE\left[(\overline{X}_\cdot - b(0) \delta(\cdot))^{-\kappa}_{r \wedge \texp}\right] ds\right).
\end{aligned}
\end{equation*}
From Gronwall's inequality, we conclude that there exists a positive constant $C$ independent on $\Dt$, such that
\begin{equation*}
\sup_{t\in[0,T]}\EE[(\overline{X}_\cdot - b(0) \delta(\cdot))^{-\kappa}_{t \wedge \texp} ]
\leq C\ (x_0^{-\kappa} + 1).
\end{equation*}
\end{proof}

\subsection{Exponential moment for the exponential scheme}\label{sec:exponential_moments_scheme}

\begin{prop}\label{prop:schem_ExpoMoment}
Assume \ref{H:polygrowth} and  \ref{H:control}.  Then,  for  $\texp$ defined in \eqref{eq:texp}, for any $\mu < B_2$, 
\begin{equation}\label{eq:Exp_moments_scheme_stopped}
 \EE \exp\left\{ \mu \int_0^{t\wedge \texp}   \oX_{\eta(s)}^{\alphab -1} ds  \right\}\  \leq C(\mu,\alphab,T).
\end{equation}
\end{prop}
\begin{proof}
Applying the Itô formula with $\log$ function to 
\begin{equation}
	d(\overline{X}_t - b(0)\dt) =\left(\overline{X}_t-b(0){\dt}\right)\Big(\tfrac{b\left(\overline{X}_{\eta(t)}\right)-b(0)}{\overline{X}_{\eta(t)}}dt+\tfrac{\sigma(\oX_{\eta(t)})}{\oX_{\eta(t)}}dW_t\Big), 
\end{equation}
and using \ref{H:control}, we get
\begin{equation*}
\log(\oX_t - b(0){\dt}) 
\leq \log(x_0) + B_1 T -B_2 \int_0^t \overline{X}_{\eta(s)}^{\alphab-1}ds+\int_0^t\tfrac{\sigma(\oX_{\eta(s)})}{\oX_{\eta(s)}}dW_s-\tfrac{1}2\int_0^t\tfrac{\sigma^2(\oX_{\eta(s)})}{\oX_{\eta(s)}^2}ds. 
\end{equation*}
Then, for any $\mu>0$, 
\begin{equation*} 
\exp\left\{\mu \int_0^t \overline{X}_{\eta(s)}^{\alphab-1}ds\right\}\leq C (\oX_t-b(0)\dt)^{-\tfrac{\mu}{B_2}} \exp\left\{\tfrac{\mu}{B_2}\int_0^t\tfrac{\sigma(\oX_{\eta(s)})}{\oX_{\eta(s)}}dW_s-\tfrac{\mu}{B_2}\tfrac{1}2\int_0^t\tfrac{\sigma^2(\oX_{\eta(s)})}{\oX_{\eta(s)}^2}ds\right\},
\end{equation*}
where $C = x_0^{\tfrac{\mu}{B_2}}\exp\{\tfrac{\mu}{B_2}B_1 T\}$. Stopping the time at $\texp$ defined in \eqref{eq:texp}, taking expectations and applying H\"older inequality, we get, for all $0<\mu<B_2$,
\begin{equation*}
\begin{aligned}
& \EE\left[\exp\left\{\mu \int_0^{t\wedge\texp} \overline{X}_{\eta(s)}^{\alphab-1}ds\right\}\right] \\
&\leq C~\EE\left[(\oX_{t\wedge\texp} - b(0)\delta({t\wedge\texp}) )^{-\tfrac{\mu}{B_2}}\exp\left\{\tfrac{\mu}{B_2}\int_0^{t\wedge\texp}\tfrac{\sigma(\oX_{\eta(s)})}{\oX_{\eta(s)}}dW_s-\tfrac{\mu}{B_2}\tfrac{1}2\int_0^{t\wedge\texp}\tfrac{\sigma^2(\oX_{\eta(s)})}{\oX_{\eta(s)}^2}ds\right\}\right]\\
& \leq C~\EE^{\tfrac{(B_2-\mu)}{B_2}}\left[(\oX_{t\wedge\texp}-b(0)\delta({t\wedge\texp}))^{-\tfrac{\mu}{B_2-\mu}}\right].
\end{aligned}
\end{equation*}
We end the proof by applying the Proposition \ref{prop:negative_moments}. 
\end{proof}

\section{The strong rate of convergence  for the exponential scheme}\label{sec:strong_error_analysis}
We introduce the error  process  $Y_t = X_t-\oX_t$, with dynamics
\begin{equation}\label{eq:sde_Y}
\begin{aligned}
	dY_t &=  Y_t \left( \frac{b(X_t)-b(0)}{X_t} dt +   \frac{\sigma(X_{t})}{X_{t}} dW_t\right) + b(0)\dt \frac{\sigma(\oX_{\eta(t)})}{\oX_{\eta(t)}}dW_t + b(0)\dt \left[\frac{b(\oX_{\eta(t)})-b(0)}{\oX_{\eta(t)}}\right]dt\\
	& \quad + \oX_t  \left( \Big[\frac{b(X_t)-b(0)}{X_t} - \frac{b(\oX_{\eta(t)})-b(0)}{\oX_{\eta(t)}}\Big]dt 
	+ \Big[\frac{\sigma(X_t)}{X_t}-\frac{\sigma(\oX_{\eta(t)})}{\oX_{\eta(t)}}\Big] dW_t  \right).
\end{aligned}
\end{equation}
This error process involves the diffusion coefficient  $ \oX_t [\tfrac{\sigma(X_t)}{X_t}-\tfrac{\sigma(\oX_{\eta(t)})}{\oX_{\eta(t)}}]$ which requires further regularity from $\sigma$. For simplicity, we assume the differentiability of $\sigma$ with a compatible polynomial growth that reinforces  \ref{H:polygrowth}-$(\sigma)$:
\begin{hypo}{\hspace{-0.12cm}\Blue{\bf  \textrm Control on $\sigma'$. (\ref{H:diffusion_deriv}).}}
\makeatletter\def\@currentlabel{ {\bf\textrm H}$_{\mbox{\scriptsize\bf\textrm{Control\_$\sigma'$}}}$}\makeatother
\label{H:diffusion_deriv}
In addition to Hypothesis \ref{H:polygrowth}-$(\sigma)$, the diffusion function $\sigma$ is differentiable in $\RR^+$ and there exists some non-negative constant $\diffcd$ such that, for all $x\in\RR^+$, 
\begin{equation}\label{eq:sigma_prime_LocallyLips}
|\sigma'(x)| \leq \alpha \ \diffcd  \  x^{\alpha-1}.
\end{equation}
\end{hypo}
We are now in a position to state our first  convergence results. 
\begin{theorem}\label{thm:strong_rate_stopped}
Assume \ref{H:pieceloclip}, \ref{H:polygrowth}, \ref{H:control} and \ref{H:diffusion_deriv}. Then, for all $0 < \varepsilon < \frac12$, 
there exists $0 < \Delta(\varepsilon) \leq 1$,    such that for all $\Delta t \leq \Delta(\varepsilon)$,  for   $\texp$ defined in \eqref{eq:texp},   for any integer $p\geq1$  satisfying 
\begin{align}
\ind_{\{2\alpha-1\}}(\alphab) \  2(p+1 + \Dt^{2\varepsilon (p+1)}) \ (2\alpha-1) \leq \tfrac12+\tfrac{B_2}{\diffc^2}, & \label{eq:condition_param_main_thm_1_1} \\
2 \ (p+1)^2 \left( \osLip  + 2(2p-1)(6 \diffc^2 + 4(\alpha \diffcd)^2\right)   < B_2,& \quad\mbox{when }(b(0),\alpha)\notin (0,\tfrac{\diffc^2}2)\times[1,\tfrac32],\label{eq:condition_param_main_thm_1_3}
\end{align}
there exists  some positive constant $C(p)$, independent of $\Dt$ and $\varepsilon$,   such that
\begin{equation*}
\sup_{0\leq t\leq T}  \big\| X_{t\wedge\texp} - \oX_{t\wedge\texp}\big\|_{L^{2p}(\Omega)}  +  \big\| \sup_{0\leq t\leq T} |X_{t\wedge\texp} - \oX_{t\wedge\texp}|\big\|_{L^{p}(\Omega)} 
\leq C(p) \,\Dt^{\frac{1}{2}  - \varepsilon}.
\end{equation*}
Furthermore, 
when $b$ is continuous on $\RR^+$,  the convergence above takes place  for $\Dt \leq 1$, ($\varepsilon=0$), and  
\begin{equation*}
\sup_{0\leq t\leq T}  \big\| X_{t\wedge\texp} - \oX_{t\wedge\texp}\big\|_{L^{2p}(\Omega)}  +  \big\| \sup_{0\leq t\leq T} |X_{t\wedge\texp} - \oX_{t\wedge\texp}|\big\|_{L^{p}(\Omega)} \leq C(p) \, \Dt^{\frac{1}{2}}.
\end{equation*}
with conditions  \eqref{eq:condition_param_main_thm_1_1} and \eqref{eq:condition_param_main_thm_1_3} enlarged to  
\begin{align}\label{eq:condition_param_thm_main_b_continuous_1}
\ind_{\{2\alpha-1\}}(\alphab)\ 2 (p+1) \ (2\alpha -1) \leq \frac12+\frac{B_2}{\diffc^2},&\\
\label{eq:condition_param_thm_main_b_continuous_2}
2 \ p \ (p+1) \left( \osLip  + 2(2p-1)(6 \diffc^2 + 4(\alpha \diffcd)^2\right)    < B_2,&  \quad\mbox{when }(b(0),\alpha)\notin (0,\tfrac{\diffc^2}2)\times[1,\tfrac32].
\end{align}
\end{theorem}
The proof  of Theorem \ref{thm:strong_rate_stopped} is  detailed in Section \ref{sec:proof_under_H5}, after some preliminaries estimates. In the proof, we draw the following  estimation for  $\varepsilon \mapsto \Delta(\varepsilon)$, behaving mainly with an exponential decay: 
\begin{equation}\label{eq:def_final_delta_epsilon}\Delta(\varepsilon):=\exp( - \frac{1}{2\varepsilon} \log (\log(  \tfrac{1}{2 \varepsilon}))) \ \wedge \ \left\{\frac{1}{16\diffc}(\tfrac{2}{3})^\alpha \underline{\Delta \pntDisc}^{1 - \alpha} \right\}^{\tfrac{2}{1-2\varepsilon}}  \wedge \  \left\{\tfrac{8}{3} \frac{\diffc}{L_G}  (\tfrac{3}{2} \underline{\Delta \pntDisc})^{\alpha - \alphab}\right\}^{\tfrac{2}{1 + 2\varepsilon}} \wedge \ \frac{\underline{\Delta \pntDisc}}{4 b(0)},
\end{equation}
where $\underline{\Delta \pntDisc} =  \min\{(\pntDisc_{k+1} - \pntDisc_k); {k=0,\ldots, m-1} \}$. 

\begin{rem}
The extension of Theorem \ref{thm:strong_rate_stopped} to $p\geq 1/2$ with value in $\RR^+$ can be achieved without too much effort using the Itô Tanaka's formula, setting $Z_t = |Y_t| = |X_t - \oX_t|$, with dynamics of the form 
\begin{equation*}
d Z_t  =   \int _0^t \sgn(Y_s) b_Y(X_t,\oX_{\eta (s)})ds
+ \sigma  \int _0^t \sgn(Y_s) \sigma_Y(X_t,\oX_{\eta (s)})  dW_s + L^0_t(Y), 
\end{equation*} 
where $\sgn(x) := 1 - 2\,\ind_{(x\leq 0)}$, and $L^0(Y) = L^0(Z)$. Since $\int_0^t Z_s^{2p-1}  dL^0_s(Y)=0$, we have 
\begin{equation*}
d \EE[Z_t^{2p}]  =   \int _0^t 2p \EE[\sgn(Y_s) Z_t^{2p-1} b_Y(X_t,\oX_{\eta (s)})]ds
+   \int _0^t p(2p-1) \sigma_Y^2(X_t,\oX_{\eta (s)})  Y^{2p-2}_s ds
\end{equation*} 
There is no major difficulty in proceeding in this way, provided we carefully add the conditions for controlling negative moments. However, the proof, already rather long, would almost double in size due to the separation of the different cases to be discussed, which motivates us to work with $p\in\mathbb{N}$. 
\end{rem}

\subsection{The Itô formula for the error process}\label{sec:error_Ito_formula_Ti}

For simplicity, the integer  $p\geq 1$  allows us to consider equivalently, $|Y_t|^{2p} = Y_t^{2p}$, for the error process  
 $(Y_t= X_t-\oX_t, t\in[0,T])$, satisfying \eqref{eq:sde_Y}. Then, applying It\^o formula to $Y_t^{2p}$, we get
\begin{equation*} 
\begin{aligned}
d Y_t^{2p}  & =  {2p} Y_t^{2p-1}\left[ Y_t\left(\tfrac{b(X_t)-b(0)}{X_t}\right)  +  \oX_t  \left(\tfrac{b(X_t)-b(0)}{X_t} - \tfrac{b(\oX_{\eta(t)})-b(0)}{\oX_{\eta(t)}}\right)  +  b(0)\dt\left(\tfrac{b(\oX_{\eta(t)})-b(0)}{\oX_{\eta(t)}}\right)\right]dt\\
&\quad +{2p} Y_t^{2p-1} \left(Y_t \tfrac{\sigma(X_t)}{X_t}+ \oX_t \Big[\tfrac{\sigma(X_t)}{X_t}-\tfrac{\sigma(\oX_{\eta(t)})}{\oX_{\eta(t)}}\Big] +  b(0)\dt\tfrac{\sigma(\oX_{\eta(t)})}{\oX_{\eta(t)}}\right)dW_t \\
&\quad + p (2p -1)Y_t^{2p-2} \left(Y_t \tfrac{\sigma(X_t)}{X_t} + \oX_t \Big[\tfrac{\sigma(X_t)}{X_t}-\tfrac{\sigma(\oX_{\eta(t)})}{\oX_{\eta(t)}}\Big] +  b(0)\dt\tfrac{\sigma(\oX_{\eta(t)})}{\oX_{\eta(t)}}\right)^2dt.
\end{aligned}
\end{equation*}
On the last line, containing the squared sum, keeping  the same terms order and putting the smallest weigh of the second term, we apply the inequality $(a+b+c)^2\leq  2 b^2 + 4 a^2 + 4 c^2$ and obtain the upper-bound of the cross variation terms. Isolating the Brownian integrals, we enumerate the part of the drift term as follows:
\begin{equation}\label{eq:error_gamma_proto}
\begin{aligned}
d Y_t^{2p} & \leq \sum_{i=1}^5 T_i(t) \; dt + {2p} Y_t^{2p-1} \left[Y_t \frac{\sigma(X_t)}{X_t}+ \oX_t \Big[\frac{\sigma(X_t)}{X_t}-\frac{\sigma(\oX_{\eta(t)})}{\oX_{\eta(t)}}\Big] +  b(0)\dt\frac{\sigma(\oX_{\eta(t)})}{\oX_{\eta(t)}}\right]dW_t,
	\end{aligned}
\end{equation}
with
\begin{equation*}
\begin{array}{lr}
T_1(t) 
=  {2p} Y_t^{2p} \left(\tfrac{b(X_t)-b(0)}{X_t} + 2(2p -1) \tfrac{\sigma^2(X_t)}{X_t^2}\right),
&\quad 
T_4(t) 
= 2p b(0) \dt \   Y_t^{2p-1} \tfrac{b(\oX_{\eta(t)})-b(0)}{\oX_{\eta(t)}},
\\
T_2 (t) 
= 2p Y_t^{2p-1}  \oX_t   \left( \tfrac{b(X_t)-b(0)}{X_t} - \tfrac{b(\oX_{\eta(t)})-b(0)}{\oX_{\eta(t)}} \right) ,
&\quad
T_5(t) 
= 4 p (2p -1) b(0)^2 \  \dt^2   Y_t^{2p-2} \tfrac{\sigma^2(\oX_{\eta(t)})}{\oX_{\eta(t)}^2}, \\
T_3(t) 
 = 2 p (2p -1) Y_t^{2p-2}   \oX_t^{2} \Big[\tfrac{\sigma(X_t)}{X_t}-\tfrac{\sigma(\oX_{\eta(t)})}{\oX_{\eta(t)}}\Big]^2.
&
\end{array} 
\end{equation*}
Notice that the stochastic integral in \eqref{eq:error_gamma_proto} is a square-integrable martingale under the sufficient condition that the $2(2p +\alpha-1)$th-moments  for $X$ and $\oX$ are bounded. With Proposition \ref{prop:XMoments}  and Lemma \ref{lem:Schememoments}, this leads to the assumption that  $2p+\alpha-1\leq \frac12 +  \frac{B_2}{\diffc^2}$, when $\alphab = 2\alpha-1$, largely covered by \eqref{eq:condition_param_main_thm_1_1} or \eqref{eq:condition_param_thm_main_b_continuous_1}. 

\subsection{Preliminary  estimations for the treatment of the  discontinuity of $b$}
We use the short notation $\{\oX_{[s,t]}\in B(\pntDisc, \delta)\} $ for the event of the entire trajectory $(\oX_u,u\in[s,t])$ in the ball $B(\pntDisc,\delta)$.

With the help of the occupation time formula for diffusions,  we estimate some statistics of the  amount of  time the approximation process spends in a  ball of radius $\hhh$, centered at a given discontinuity point $\pntDisc$: 
\begin{lem}{[Sojourn time of the scheme in a given ball].}\label{lem:occupation_time} Assume \ref{H:polygrowth}, \ref{H:control}, \ref{H:diffusion_deriv}. 
Let $t\in[0,T]$  and $p >0$. Let $\pntDisc>0$ and $\Dt \leq \frac{\pntDisc}{4 b(0)}$.  Then, for any $0 < \hhh \leq \frac{1}{4}\pntDisc$,  for any $q\geq 1$,  there exists a constant $C$ (independent of $\Dt$ and $h$) such that 
\begin{equation}\label{eq:expo_moment_small_ball}
\begin{aligned}
 \left\|\int_0^t\ind_{\{ \oX_{[\eta(s), \eta(s) + \Dt]} \in B(\pntDisc,  2 \hhh )\}} ds\right\|_{L^q(\Omega)}
  &\leq \frac{C}{\displaystyle\inf_{y\in B(\pntDisc, \frac{\pntDisc}{2})} \sigma^2(y)}  \ \hhh,
\\ 
\text{and for all $\mu >0$, }\quad  \EE \left[ \exp\left\{\mu \int_0^t\ind_{\{ \oX_{[\eta(s), \eta(s) + \Dt]} \in B(\pntDisc, 2 \hhh )\}} ds \right\} \right]  & \leq  \exp\Big\{\mu \  \frac{C}{\displaystyle\inf_{y\in B(\pntDisc, \frac{\pntDisc}{2})} \sigma^2(y)} \  \hhh \Big\}. 
\end{aligned}
\end{equation}
\end{lem} 
\begin{proof} We focus the proof of the first statement in \eqref{eq:expo_moment_small_ball}, the second one being implied by this first.  We introduce the following interpolated dynamics of $\oX$, coinciding with \eqref{eq:ContExpscheme}: 
\begin{equation}\label{eq:Xbar_coeff}
\begin{array}{c}
  d\overline{X}_t 
 := \bexp(\oX_t, \oX_{\eta(t)}) dt + \sexp(\oX_t, \oX_{\eta(t)}) dW_t \\[0.2cm]
\bexp(x,y) = (x-b(0)\dt)\Big(\dfrac{b(y)-b(0)}{y}\Big)  + b(0), 
 \qquad\qquad\quad  \sexp(x,y) = \dfrac{\sigma(y)}{y}(x-b(0)\dt). 
\end{array}
\end{equation}
We observe that
\begin{equation*}
\int_0^t\ind_{\{ \oX_{[\eta(s), \eta(s) + \Dt]} \in B(\pntDisc, 2 \hhh )\}} ds 
= \sum_{k=0}^{\frac{\eta(t)}{\Dt}}  \ind_{\{ \oX_{[t_k,t\wedge t_{k+1}]}\in B(\pntDisc , 2\hhh)\}} \int_{t_k}^{t\wedge t_{k+1}}  \ind_{\{ \oX_s \in B(\pntDisc , 2\hhh)\}} ds.
\end{equation*}
Since $\sigma$ is continuous, then it is bounded on  $B(\pntDisc, 2\hhh)$ and the diffusion $\sexp^2$ is bounded below on the ball $B(\pntDisc, 2\hhh)^2$ by $\frac{(\pntDisc  - 2\hhh -b(0)\Dt )^{2}}{(\pntDisc + 2\hhh)^2}\ \inf_{y\in B(\pntDisc, 2\hhh)}{\sigma^2(y)} $.   With the imposed conditions  that  $\hhh\leq \pntDisc/4$ and $\Dt\leq \frac{\pntDisc}{4b(0)}$, we have 
\begin{equation*}
\int_{t_k}^{t\wedge t_{k+1}}  \ind_{\{ \oX_s \in B(\pntDisc , 2\hhh)\}} ds  \leq 
\frac{36}{\displaystyle
\inf_{y\in B(\pntDisc, \pntDisc/2)}\sigma^2(y)} \int_{t_k}^{t\wedge t_{k+1}}  \sexp(\oX_s,\oX_{\eta(s)})^2 \ind_{\{ \oX_s \in B(\pntDisc , 2\hhh)\}}ds. 
\end{equation*}
The  occupation time formula (see, e.g.~\cite[Ex.1.15,Chap.VI]{Revuz_Yor_1999}) leads to representation 
\begin{equation*}
\int_{t_k}^{t\wedge t_{k+1}}  \sexp(\oX_s,\oX_{\eta(s)})^2\ind_{\{ \oX_{s}\in B(\pntDisc , 2\hhh)\}} ds  = \int_{\RR} \ind_{[\pntDisc -2\hhh, \pntDisc + 2\hhh ]} (y) ( L^y_{t\wedge t_{k+1}}(\oX) - L^y_{t_k}(\oX) )  \ dy 
\end{equation*} 
where $L^y(\oX)$ is the (increasing) local time in $y$ of the Itô process $\oX$. Then 
\begin{equation*}
\int_{t_k}^{t\wedge t_{k+1}}  \ind_{\{ \oX_s \in B(\pntDisc , 2\hhh)\}} ds  \leq {C} \hhh 
\left( \sup_{y\in B(\pntDisc , 2\hhh)} \tfrac{1}{2}( L^y_{t\wedge t_{k+1}}(\oX) - L^y_{t_k}(\oX) )\right). 
\end{equation*}
By the Itô-Tanaka-Meyer formula, 
\begin{equation*} 
0 \leq \tfrac12 ( L^y_{t\wedge t_{k+1}}(\oX) - L^y_{t_k}(\oX) ) = 
(\oX_{t\wedge t_{k+1}}-y)_+ - (\oX_{t_k}-y)_+ -  \int_{t_k}^{t\wedge t_{k+1}} \ind_{\{\oX_{s-} > y\}}d\oX_s. 
\end{equation*} 
But for all $y$, from the Lipschitz property of the positive part function, it is easy to check that
\begin{equation*} 
(\oX_{t_{k+1}}-y)_+ - (\oX_{t_k}-y)_+  \leq
 (\oX_{t_{k+1}} - \oX_{t_k})\ 
 \ind_{\{\oX_{t_{k}} \leq \oX_{t_{k+1}} \}}
 = \ \int_{t_k}^{t_{k+1}} \  \ind_{\{ \oX_{\eta(s)} \leq \oX_{\eta(s) + \Dt}\}} d\oX_s. 
\end{equation*} 
By summing on the time intervals, we have obtained that 
\[ \int_0^t\ind_{\{ \oX_{[\eta(s), \eta(s) + \Dt]} \in B(\pntDisc, 2 \hhh )\}} ds \leq C \hhh \ H_t(\oX),\]
with 
\begin{equation*}
\begin{aligned} 
 H_t(\oX):= &  \sup_{y\in B(\pntDisc , 2\hhh)}  \sum_{k=1}^{\frac{\eta(t)}{\Dt}}  \ind_{\{ \oX_{[t_k,t\wedge t_{k+1}]}\in B(\pntDisc , 2\hhh)\}}
\frac12 ( L^y_{t\wedge t_{k+1}}(\oX) - L^y_{t_k}(\oX) ),
\end{aligned}
\end{equation*} 
and $|H_t(\oX)|$ further bounded by 
\begin{equation*}
\begin{aligned} 
|H_t(\oX)|  & \leq 
\left| \int_0^t \ind_{\{ \oX_{[\eta(s), \eta(s) + \Dt]}\in B(\pntDisc , 2\hhh)\}} \ind_{\{ \oX_{\eta(s)} \leq \oX_{\eta(s) + \Dt}\}} d\oX_s  \right|
\\
& \qquad  +\sup_{y\in B(\pntDisc , 2\hhh)}   \left|  \int_{0}^{t}   \ind_{\{ \oX_{[\eta(s), t \wedge \eta(s) + \Dt]}\in B(\pntDisc , 2\hhh)\}}  \ind_{\{\oX_{s-} > y\}}d\oX_s \right|\\
& \leq 2 \int_0^t \ind_{\{ \oX_{[\eta(s), \eta(s) + \Dt]}\in B(\pntDisc , 2\hhh)\}}  d\oX_s .
\end{aligned}
\end{equation*} 
To complete the proof, it remains to verify that for any $q\geq 1$, $\EE[H^q(\oX)]$ is bounded with a polynomial dependency in $q$ for the right-hand side constant: 
\begin{equation*}
\begin{aligned} 
\EE[H^q(\oX)] 
& \leq 2^q  \EE\left[ \left( \int_0^t \ind_{\{ \oX_{[\eta(s), \eta(s) + \Dt]}\in B(\pntDisc , 2\hhh)\}}  d\oX_s \right)^q \right]\\
& \leq 2^{2q} \EE\left[ \left( \int_0^t \ind_{\{ \oX_{[\eta(s), \eta(s) + \Dt]}\in B(\pntDisc , 2\hhh)\}}  \bexp(\oX_s, \oX_{\eta(s)}) ds \right)^q \right]\\
&\quad + 2^{2q} \EE^\frac12 \left[ \int_0^t \ind_{\{ \oX_{[\eta(s), \eta(s) + \Dt]}\in B(\pntDisc , 2\hhh)\}} \sexp(\oX_t, \oX_{\eta(t)})^{2q} ds \right]\\
& \leq 2^{2q} t^q (C \  \pntDisc^{\alphab}\vee\pntDisc + b_0)^q \  + \ 2^{2q} t^\frac{q}{2} \diffc^q (\tfrac{3}{2}\pntDisc)^{\alpha q}.
\end{aligned}
\end{equation*}
\end{proof}

In the convergence analysis, the only instance where the discontinuity of the drift $b$ needs to be addressed is within the local error term singled out on the right-hand side of \eqref{eq:disc_contrib} below, and examined in the following lemma: 
\begin{lem}\label{lem:disc_contrib}
Assume \ref{H:pieceloclip}, \ref{H:polygrowth}, \ref{H:control} and \ref{H:diffusion_deriv}.  For all $0 < \varepsilon < \frac12$, 
consider $0 < \Delta(\varepsilon) \leq 1$,  given by \eqref{eq:def_final_delta_epsilon}. Then, for all $\Delta t \leq \Delta(\varepsilon)$,  for any integer $p\geq1$  satisfying 
\begin{equation}\label{cond:lemma_discont}
\begin{aligned}
\ind_{\{2\alpha-1\}}(\alphab)\  ( \ 2p + \Dt^{2 p \varepsilon}  ) \ (2 \alpha -1)  &  \leq \tfrac12+\tfrac{B_2}{\diffc^2}, 
\end{aligned}
\end{equation}
there exits a positive constant $C$, independent of $\Dt$ and $\epsilon$, such that 
\begin{equation}\label{eq:disc_contrib}
\begin{aligned}
& \EE\left[\left| Y_t^{2p-1}  \oX_t\left( \frac{b(\oX_t)-b(0)}{\oX_t} - \frac{b(\oX_{\eta(t)})-b(0)}{\oX_{\eta(t)}} \right)  \right|\right] \\
& \leq C \EE[ Y_t^{2p} ] +  C  \Dt^{p (1 - 2\varepsilon)}  + \EE\left[|Y_t|^{2p-1}  \sum_{i=1}^m\ind_{  \left\{\oX_{[\eta(t), \eta(t) + \Dt]}\in B(\pntDisc_i , 8\diffc  (\tfrac{3}{2} \underline{\Delta \pntDisc})^{\alpha}\ \Dt^{1/2-\varepsilon} )\right\}}\right].
\end{aligned}
\end{equation}
When $b$ is a continuous function, for all  $\Delta t \leq  1$,  for any integer $p\geq1$  satisfying 
\begin{equation}\label{cond:lemma_cont}
\begin{aligned}
\ind_{\{2\alpha-1\}}(\alphab) \ \ p \ (4 \alpha - 3)   \leq \tfrac12+\tfrac{B_2}{\diffc^2}, 
\end{aligned}
\end{equation}
there exits a positive constant $C$, independent of $\Dt$, such that 
\begin{equation}\label{eq:cont_contrib}
\begin{aligned}
\EE\left[\left| Y_t^{2p-1}  \oX_t\left( \frac{b(\oX_t)-b(0)}{\oX_t} - \frac{b(\oX_{\eta(t)})-b(0)}{\oX_{\eta(t)}} \right)  \right|\right] \leq C \EE[ Y_t^{2p} ] +  C  \Dt^{p}. 
\end{aligned}
\end{equation}
\end{lem}

\begin{proof}
Without loss of generality, we simplify the discussion in this proof, by assuming that $b$ has only one discontinuity point $\pntDisc_m$. We consider a ball $B(\pntDisc_m , f(\Dt) )$, with $f(\Dt)\leq \pntDisc_m/4$ to be chosen later  small enough.   We decompose 
\begin{equation}\label{eq:aux_A}
A:= Y_t^{2p-1}  \oX_t\left( \frac{b(\oX_t)-b(0)}{\oX_t} - \frac{b(\oX_{\eta(t)})-b(0)}{\oX_{\eta(t)}} \right) 
\end{equation}
discussing the relative position of $\oX_t$ and $\oX_{\eta(t)}$ w.r.t $\pntDisc_m$. First we introduce the two  events
\begin{equation*}
A \left(\ind_{\{\oX_{\eta(t)}\in B(\pntDisc_m , f(\Dt) )\}} + \ind_{\{\oX_{\eta(t)}\notin B(\pntDisc_m , f(\Dt) )\}}\right).
\end{equation*}
These events are further decomposed.  For the event $\{\oX_{\eta(t)}\notin B(\pntDisc_m , f(\Dt) )\}$, we isolate the discontinuity point $\pntDisc_m$, while discussing the distance  $| \oX_{\cdot}-\oX_{\eta(t)}|$, compared to the threshold $\uu$: 
\begin{equation*}
\begin{aligned}
& \ind_{\{\oX_{\eta(t)}\notin B(\pntDisc_m , f(\Dt) )\}}   \leq 
\ind_{\{ (\oX_{\eta(t)} \wedge \oX_t) > \pntDisc_m  \}}  +  \ind_{\{(\oX_{\eta(t)}\vee \oX_t) < \pntDisc_m   \}}  + \ind_{\{| \oX_{t}-\oX_{\eta(t)}| \geq  f(\Dt) \}} \\
& \leq \left(\ind_{\{ (\oX_{\eta(t)} \wedge \oX_t) > \pntDisc_m  \}}  +  \ind_{\{(\oX_{\eta(t)}\vee \oX_t) < \pntDisc_m   \}} \right) \\
& \quad  + \ind_{\{| \oX_{t}-\oX_{\eta(t)}| \geq  
f(\Dt) \}}\left(\ind_{\{\sup_{s\in[\eta(t),t)}| \oX_{s}-\oX_{\eta(t)}| \geq\uu \}} + \ind_{\{ \sup_{s\in[\eta(t),t)}| \oX_{s}-\oX_{\eta(t)}| \leq\uu\} \cap \{(\oX_{\eta(t)} \wedge \oX_t) \leq \pntDisc_m \leq (\oX_{\eta(t)}\vee \oX_t) \}}\right).
\end{aligned}
\end{equation*}
Notice that  $\pntDisc_m - \uu  - f(\Dt) \geq \pntDisc_m /2$  is far from zero, and  
\begin{equation*}
\begin{aligned}
 \ind_{\{\oX_{\eta(t)}\notin B(\pntDisc_m , f(\Dt) )\}}   
& \leq  \left(\ind_{\{ (\oX_{\eta(t)} \wedge \oX_t) > \pntDisc_m  \}}  +  \ind_{\{(\oX_{\eta(t)}\vee \oX_t) < \pntDisc_m   \}} \right) \\
& \quad + 
\left(
\ind_{\{\sup_{s\in[\eta(t),\eta(t)+\Dt)}| \oX_{s}-\oX_{\eta(t)}| \geq\uu \}} + 
\ind_{\{| \oX_{t}-\oX_{\eta(t)}| \geqslant f(\Dt) \}} \ 
\ind_{\{\oX_{[\eta(t),\eta(t)+\Dt)}  \in B(\pntDisc_m,\uu)| \}}\right).
\end{aligned}
\end{equation*}
We decompose the event $\{\oX_{\eta(t)}\in B(\pntDisc_m , f(\Dt) )\}$, following the same idea: 
\begin{equation*}
\begin{aligned}
& \ind_{\{\oX_{\eta(t)}\in B(\pntDisc_m , f(\Dt) )\}}    \\
& = \ind_{\{\oX_{\eta(t)}\in B(\pntDisc_m , f(\Dt) )\}} \left( 
 \ind_{\{\sup_{s\in[\eta(t), \eta(t) + \Dt)}| \oX_{s}-\oX_{\eta(t)}| \geq \uu \}} +
  \ind_{\{\sup_{s\in[\eta(t), \eta(t) + \Dt)}| \oX_{s}-\oX_{\eta(t)}| \leq \uu \}} \right)\\
& \quad \leq \ind_{\{\sup_{s\in[\eta(t), \eta(t) + \Dt)}| \oX_{s}-\oX_{\eta(t)}| \geq \uu \}}  \\
& \quad \quad + 
 \ind_{\{\oX_{\eta(t)}\in B(\pntDisc_m , f(\Dt) )\}} \left( 
 \ind_{\{\oX_{[\eta(t), \eta(t) + \Dt]}\in B(\pntDisc_m , 2f(\Dt) )\}}  + 
  \ind_{\{\oX_{[\eta(t), \eta(t) + \Dt]}\in B(\pntDisc_m ,\uu+f(\Dt) ) -  B(\pntDisc_m , 2f(\Dt) )\}} \right) \\
& \quad \leq  \ind_{\{\sup_{s\in[\eta(t), \eta(t) + \Dt)}| \oX_{s}-\oX_{\eta(t)}| \geq \uu \}} \\
& \quad +  \ind_{\{\oX_{\eta(t)}\in B(\pntDisc_m , f(\Dt) )\}}\ind_{  \{\oX_{[\eta(t), \eta(t) + \Dt]}\in B(\pntDisc_m , 2f(\Dt) )\}}  \\
& \quad +  \ind_{\{\oX_{\eta(t)}\in B(\pntDisc_m , f(\Dt) )\}}\ind_{\{\oX_{[\eta(t), \eta(t) + \Dt]}\in B(\pntDisc_m ,\uu+f(\Dt) ) \}} \ind_{\{\sup_{s\in[\eta(t), \eta(t) + \Dt)}| \oX_{s}-\oX_{\eta(t)}| \geq f(\Dt) \}}.
\end{aligned}
\end{equation*}
After small bound in the indicators and joining common terms, we can gather the two decompositions  as follows, identifying four terms that we evaluate with different techniques:   
\begin{equation*}
\begin{aligned}
A  & \leq  
\left(\ind_{\{ (\oX_{\eta(t)} \wedge \oX_t) > \pntDisc_m  \}}  +  \ind_{\{(\oX_{\eta(t)}\vee \oX_t) < \pntDisc_m   \}} \right) \ A  \\
& \quad + 2 \ \ind_{\{\sup_{s\in[\eta(t), \eta(t) + \Dt)}| \oX_{s}-\oX_{\eta(t)}| \geq \uu \}} \  A  \\
& \quad  +  \ind_{  \{\oX_{[\eta(t), \eta(t) + \Dt]}\in B(\pntDisc_m , 2f(\Dt) )\}}  \ A  \\
& \quad  +  2 \ \ind_{\{\oX_{[\eta(t), \eta(t) + \Dt]}\in B(\pntDisc_m ,\uu+f(\Dt) ) \}} \ind_{\{\sup_{s\in[\eta(t), \eta(t) + \Dt)}| \oX_{s}-\oX_{\eta(t)}| \geq f(\Dt) \}} \ A \\
& \qquad := A_{1} + A_{2} + A_{3} + A_{4}.
\end{aligned}
\end{equation*}
In what follows, the changing generic constant $C$ remains  uniformly bounded in $\Dt$, but may depend on $\pntDisc_m$. 

\subsubsection*{Bound for $|A_{3}|$. }
The path $\oX_{[\eta(t), \eta(t) + \Dt]}$ being confined in the ball $ B(\pntDisc_m , 2f(\Dt) )\}$, this   term  is just bounded as 
\begin{equation*}
\begin{aligned}
|A_{3}| = & \left|Y_t^{2p-1}  \oX_t\left( \frac{b(\oX_t)-b(0)}{\oX_t} - \frac{b(\oX_{\eta(t)})-b(0)}{\oX_{\eta(t)}} \right) \right| \ind_{  \{\oX_{[\eta(t), \eta(t) + \Dt]}\in B(\pntDisc_m , 2f(\Dt) )\}}  \\ 
& \leq C  |Y_t|^{2p-1}  \ind_{  \{\oX_{[\eta(t), \eta(t) + \Dt]}\in B(\pntDisc_m , 2f(\Dt) )\}}.
\end{aligned}
\end{equation*}

\subsubsection*{Bound for $\EE|A_1|$,  far from the discontinuity. }
The indicators locate the arguments of $b$ in its intervals of continuity, as $\pntDisc_m \notin [\oX_t\wedge\oX_{\eta(t)},\oX_t\vee\oX_{\eta(t)}]$, where $b$  is locally Lipschitz. 
We observe that
\begin{equation*}
 \oX_t^{2p}\left( \frac{b(\oX_t)-b(0)}{\oX_t} - \frac{b(\oX_{\eta(t)})-b(0)}{\oX_{\eta(t)}} \right)^{2p} =  (\oX_t - \oX_{\eta(t)})^{2p}\left\{\frac{\oX_t}{\oX_t - \oX_{\eta(t)}} \left(\frac{b(\oX_t)-b(0)}{\oX_t} - \frac{b(\oX_{\eta(t)})-b(0)}{\oX_{\eta(t)}} \right)\right\}^{2p}.
\end{equation*}
From \ref{H:pieceloclip} in the  continuity interval $[\oX_t\wedge\oX_{\eta(t)},\oX_t\vee\oX_{\eta(t)}]$, and \ref{H:polygrowth}
\begin{equation*}
\begin{aligned}
& \left|\frac{\oX_t}{\oX_t - \oX_{\eta(t)}}\left( \frac{b(\oX_t)-b(0)}{\oX_t} - \frac{b(\oX_{\eta(t)})-b(0)}{\oX_{\eta(t)}} \right)\right| = \left| \frac{b(\oX_t)-b(\oX_{\eta(t)})}{{\oX_t - \oX_{\eta(t)}}}- \frac{b(\oX_{\eta(t)})-b(0)}{\oX_{\eta(t)}}\right| \\
& \quad\leq \Lips ( 1+\oX_t^{\alphab-1}+\oX_{\eta(t)}^{\alphab-1}) + \Lg (\oX_{\eta(t)}^{\alphab-1} \vee 1).
\end{aligned}
\end{equation*}
Therefore, for any $r_2>1$ (to be chosen later), for another constant $C$
\begin{equation*}
\EE\left[ \oX_t^{2p}\left( \frac{b(\oX_t)-b(0)}{\oX_t} - \frac{b(\oX_{\eta(t)})-b(0)}{\oX_{\eta(t)}} \right)^{2p}\right] 
\leq C \left(1+\sup_{t\in[0,T]}\EE^{\frac{r_2-1}{r_2}}\Big[\oX_t^{2p(\alphab-1)\frac{r_2}{r_2-1}}\Big]\right)\EE^{\frac{1}{r_2}}\Big[|\oX_t - \oX_{\eta(t)}|^{2p r_2}\Big],
\end{equation*}
and we get, applying a Young inequality in $A_1$ first,  (and removing the indicators) 
\begin{equation*}
\EE[|A_1|] \leq C \EE[ Y_t^{2p} ] +  C  \left(1+\sup_{t\in[0,T]}\EE^{\frac{r_2-1}{r_2}}\Big[\oX_t^{2p(\alphab-1)\frac{r_2}{r_2-1}}\Big]\right)
\|\oX_t  - \oX_{\eta(t)}\|_{L^{2p r_2}(\Omega)} 
^{2p}.
\end{equation*}
We apply Lemma \ref{lem:Schememoments} and Lemma \ref{lem:local_error_prev}
to control $\|\oX_t- \oX_{\eta(t)}\|_{L^{2 p r_2}} \leq C \Dt^{\frac12}$ and $\sup_{s\in[0,T]}\|\oX_{s}\|_{L^{2p (\alphab -1)  \frac{r_2}{r_2-1}}}\leq C$, implying to satisfy the combined constraint that 
\[\ind_{\{2\alpha-1\}}(\alphab)\, p  \ \Big(r_2 \alphab \vee [\tfrac{r_2}{r_2-1}  (\alphab-1)]\Big) \leq \frac{1}{2} +\frac{B_2}{\diffc^2}, 
\]
leading to the well balanced  choice of $r_2 = 1 + \frac{\alphab-1}{\alphab}$, and the constraint  $\ind_{\{2\alpha-1\}}(\alphab)\, ( 4  \alpha -3) p \leq \frac12+\frac{B_2}{\diffc^2}$ in \eqref{cond:lemma_cont} covered by the condition \eqref{cond:lemma_discont}. Then 
\begin{equation*}
\EE[|A_1|] \leq C \EE[ Y_t^{2p} ] +  C  \Dt^p.
\end{equation*}
When $b$ is continuous,  $A = A_1$ and the latter estimation gives \eqref{eq:cont_contrib} under the condition \eqref{cond:lemma_cont}. 

\subsubsection*{Bernstein inequality for the bound of $\EE|A_{4}|$. } 
From \ref{H:polygrowth},  we have the following bound  
\begin{equation*}
\begin{aligned}
& \oX_t\left( \frac{b(\oX_t)-b(0)}{\oX_t} - \frac{b(\oX_{\eta(t)})-b(0)}{\oX_{\eta(t)}} \right)
\ind_{\{\oX_{[\eta(t), \eta(t) + \Dt]}\in B(\pntDisc_m ,\uu+f(\Dt) ) \}} \ind_{\{\sup_{s\in[\eta(t), \eta(t) + \Dt)}| \oX_{s}-\oX_{\eta(t)}| \geq f(\Dt) \}}\nonumber\\
& \qquad \leq C  
\ind_{\left\{\oX_{[\eta(t), \eta(t) + \Dt]}\in B(\pntDisc_m ,\uu+f(\Dt) ) \right\}}  \ind_{\left\{\sup_{s\in[\eta(t), \eta(t) + \Dt)}| \oX_{s}-\oX_{\eta(t)}| \geq f(\Dt)\right \}}, 
\end{aligned}
\end{equation*}
where $C$ above is a polynomial on $\pntDisc_m^{\alphab}$.  So 
\begin{equation}\label{eq:G11_G122}
    |A_{4}| \leq C  |Y_t|^{2p-1} 
\ind_{\{\oX_{[\eta(t), \eta(t) + \Dt]}\in B(\pntDisc_m ,\uu+f(\Dt) ) \}}  \ind_{\{\sup_{s\in[\eta(t), \eta(t) + \Dt)}| \oX_{s}-\oX_{\eta(t)}| \geq f(\Dt) \}}.
\end{equation}
Using the definition of $\sexp$ in \eqref{eq:Xbar_coeff}, we denote by $M$ the martingale
\[ M_t = \int_0^t \ind_{\{\oX_{[\eta(s), s)}\in B(\pntDisc_m ,\uu+f(\Dt) ) \}} \sexp(\oX_s, \oX_{\eta(s)}) dW_s,\]
which from \ref{H:diffusion_deriv} satisfies, for all $t\leq T$ 
\begin{equation*}
\av{M,M}_t \leq \diffc^2 \int_0^t \EE[ \ind_{\{\oX_{[\eta(s), s)}\in B(\pntDisc_m ,\uu+f(\Dt) ) \}} \ \oX_{\eta(s)}^{2(\alpha-1)} \oX_s^{2}] \ ds \leq  c^2 \ t,
\end{equation*}
with $c= \diffc  (\tfrac{3}{2} \pntDisc_m)^{\alpha} $.
Then, we have
\begin{equation*}
\begin{aligned}
& \left\{\sup_{s\in[\eta(t), \eta(t) + \Dt)}| \oX_{s}-\oX_{\eta(t)}| \geq f(\Dt) \right\} \  \cap \  \left\{\oX_{[\eta(t), \eta(t) + \Dt]}\in B(\pntDisc_m ,\uu+f(\Dt) ) \right\} \\
& \subset \left\{\sup_{s\in[\eta(t), \eta(t) + \Dt)} |\int_{\eta(t)}^s \ind_{\{\oX_{[\eta(t), \theta)}\in B(\pntDisc_m ,\uu+f(\Dt) ) \}} \bexp(\oX_\theta,\oX_{\eta(t)})\  d\theta + M_s - M_{\eta(s)}| \geq  f(\Dt)\right \} \\
& \quad \subset  \left \{\sup_{s\in[\eta(t), \eta(t) + \Dt)} |M_s - M_{\eta(t)}| \geq f(\Dt)  - \Dt \sup_{B(\pntDisc_m ,\uu+f(\Dt) )}|\bexp|\right \} \\
& \quad  \subset  \left \{ \sup_{s\in[\eta(t), \eta(t) + \Dt)} |M_s - M_{\eta(s)}| \geq \frac{1}{2} f(\Dt) \right\},
\end{aligned}
\end{equation*}
provided that $\sup_{x,y\in B(\pntDisc_m,\uu+f(\Dt))} |b^{exp}(x,y)| \leq \frac12 f(\Dt)/ \Dt$, in addition to the previous condition $f(\Dt)\leq\uu$. 
We now use the Bernstein inequality (see, e.g., Revuz and Yor~\cite[Ex.3.16,Chap.IV, p. 153]{Revuz_Yor_1999}):
\begin{equation*}
\PP\left( \sup_{t_k \leq s \leq t_{k+1}} |M_s - M_{\eta(s)}| \geq \frac{f(\Dt) }{2\Dt} \Dt \right)  \leq \exp\left(-\dfrac{1}{c^2}\dfrac{f^2(\Dt)}{8 \Dt}\right).
\end{equation*}
Thus, taking expectation and applying Young inequality in \eqref{eq:G11_G122}, we can use the estimation above as follows:
\begin{equation*}
\begin{aligned}
\EE[|A_{4}|] 
&\leq  C \ \EE[Y_t^{2p-1} 
  \ind_{\{\sup_{s\in[\eta(t), \eta(t) + \Dt)}| M_{s}-M_{\eta(t)}| \geq \frac{f(\Dt)}{2}\}} ] \\
&\leq C\ \EE[Y_t^{2p}] + C\ \left(\PP\left( \sup_{t_k \leq s \leq t_{k+1}} |M_s - M_{\eta(s)}| \geq \frac{f(\Dt) }{2\Dt} \Dt \right) \right)^{\frac1{2p}}\\
&\leq C\  \EE[Y_t^{2p}] +  C \ \exp\left(-\dfrac{1}{c^2}\dfrac{f^2(\Dt)}{16p \Dt}\right).
\end{aligned}
\end{equation*}
Considering $f(\Dt) : = 4c \ \Dt^{1/2-\epsilon_1}$, for some $\epsilon_1\in(0,\frac12)$, we get:
\begin{equation*}
\EE[|A_{4}|] \leq C\  \EE[Y_t^{2p}] +  C \ \exp\left(-\dfrac{1}{p}\Dt^{-2\epsilon_1}\right).
\end{equation*}
Notice that, the map $y \mapsto \exp( \frac{1}{p^2}\frac{2\epsilon_1}{1 - 2\epsilon_1} y) - y$ hits zero, its minimum value, at $y^* = \log (p^2 \frac{(1 -2 \epsilon_1)}{2 \epsilon_1})$. So, for a given $p$ and $\epsilon_1$, if $\Dt$ is chosen smaller than $\exp( - \frac{1}{2\epsilon_1} \log (\log( p^2 \frac{(1 -2 \epsilon_1)}{2 \epsilon_1})))$, then $\Dt^{-2\epsilon_1}>y^*$ and the map is increasing, leading to the inequality
\begin{equation*}
\exp\left(- \frac{1}{ p }\Dt^{-2\epsilon_1}\right) \leq  \Dt^{p-2\epsilon_1 p }.
\end{equation*}

Putting all the constraint on $\Dt$ together, a sufficient condition on $\Dt$ is that $\Dt\leq \Delta(\epsilon_1)$, with 
\begin{equation}\label{eq:def_delta_epsilon}\Delta(\epsilon_1):=\exp( - \frac{1}{2\epsilon_1} \log (\log(  \tfrac{1}{2 \epsilon_1}))) \ \wedge \ \left\{\frac{1}{16\diffc}(\tfrac{2}{3})^\alpha \pntDisc_m^{1 - \alpha} \right\}^{\tfrac{2}{1-2\epsilon_1}}  \wedge \  \left\{\tfrac{8}{3} \frac{\diffc}{L_G}  (\tfrac{3}{2}\pntDisc_m)^{\alpha - \alphab}\right\}^{\tfrac{2}{1 + 2\epsilon_1}} \wedge \ \frac{\pntDisc_m}{4 b(0)},
\end{equation}
using $f(\Dt)\leq \pntDisc_m/4$ to impose    
$\sup_{B(\pntDisc_m,\uu+f(\Dt))}|b^{\exp}|\leq 3 L_G (\tfrac{3}{2}\pntDisc_m^\alphab \vee 1)\leq \tfrac{1}{2}{f(\Dt)}/{\Dt}$, as required. In this case we obtain the estimation:
\begin{equation*}
\EE[|A_{4}|] \leq C\  \EE[Y_t^{2p}] +  C \ \Dt^{p - 2\epsilon_1 p }.
\end{equation*}

\subsubsection*{Markov inequality argument for the bound of $\EE|A_{2}|$. } 
Applying Young inequality with \ref{H:polygrowth}:
\begin{equation*}
\begin{aligned}
A_2 = & 2 Y_t^{2p-1}  \oX_t\left( \frac{b(\oX_t)-b(0)}{\oX_t} - \frac{b(\oX_{\eta(t)})-b(0)}{\oX_{\eta(t)}} \right)  \ind_{\{\sup_{s\in[\eta(t), \eta(t) + \Dt)}| \oX_{s}-\oX_{\eta(t)}|  \geq  \uu \}} \\
& \leq C  |Y_t|^{2p-1} \left(1+ \oX_{t}^{\alphab} + \oX_{\eta(t)}^{\alphab}  \right) \ind_{\{\sup_{s\in[\eta(t), \eta(t) + \Dt)}| \oX_{s}-\oX_{\eta(t)}| \geq  \uu\} }\\
& \leq C  |Y_t|^{2p} + C \left(1+ \oX_{t}^{2p \alphab} + \oX_{\eta(t)}^{2p \alphab}  \right) \ind_{\{\sup_{s\in[\eta(t), \eta(t) + \Dt)}| \oX_{s}-\oX_{\eta(t)}| \geq  \uu \}}. 
\end{aligned}
\end{equation*}
For some constant exponent $r_3>1$, 
\begin{equation*}
\begin{aligned}
&\EE\Big[\left(1+\oX_{t}^{2p \alphab} + \oX_{\eta(t)}^{2p \alphab}  \right) \ind_{\{\sup_{s\in[\eta(t), \eta(t) + \Dt)}| \oX_{s}-\oX_{\eta(t)}| \geq  \uu \}}\Big] \\
&\qquad \leq 
\Big(1+\sup_{s\leq T} \EE^{\frac{r_3-1}{r_3}}[\oX_{s}^{2p\alphab \frac{r_3}{r_3-1}}]\Big) \PP^\frac{1}{r_3}\left(\sup_{s\in[\eta(t), \eta(t) + \Dt)}| \oX_{s}-\oX_{\eta(t)}| \geq  \uu \right). 
\end{aligned}
\end{equation*}
By the Markov inequality combined with the local error in Lemma \ref{lem:local_error_prev}, for $\gamma=1$, $\epsilon_2\in(0,\tfrac12)$ and $p$ substitute by $2pr_3$:
\begin{equation*}
\PP^\frac{1}{r_3}\left(\sup_{s\in[\eta(t), \eta(t) + \Dt)}| \oX_{s}-\oX_{\eta(t)}| \geq  \uu \right)  
\leq \EE^\frac{1}{r_3}\left[\sup_{s\in[\eta(t), \eta(t) + \Dt)}| \oX_{s}-\oX_{\eta(t)}|^{2p r_3} \right]  \leq C \ \Dt^{p - 2\epsilon_2 p }. 
\end{equation*}
Notice that, in order to get $\|\sup_{\eta(s) \leq \theta \leq s}  |\oX_\theta- \oX_{\eta(s)}|\|_{L^{2 p r_3}(\Omega)}\leq C \Dt^{{\frac12}-\epsilon_2}$ from Lemma \ref{lem:local_error_prev}  and the upperbound of $\sup_{s\leq T}\|\oX_{s}\|_{L^{2p \alphab \frac{r_3}{r_3-1}}(\Omega)}$ from Lemma \ref{lem:Schememoments},   we need to assume the combined constraint that 
\begin{equation*}
\ind_{\{2\alpha-1\}}(\alphab)\, p \alphab \left(r_3 \frac{p+\Dt^{2p\epsilon_2}}{p} \vee \frac{r_3}{r_3-1}\right)  \leq \frac{1}{2} +\frac{B_2}{\diffc^2},
\end{equation*}
that coincides with the condition \eqref{cond:lemma_discont} with   $r_3 = 1 + \frac{p}{p+\Dt^{2p\epsilon_2}}$  and $\varepsilon=\epsilon_2$. 
Then, the contribution of this terms is of the form
\begin{equation*}
\EE| A_2| \leq C \EE[ Y_t^{2p}] + C \Dt^{p-2\epsilon_2 p }.
\end{equation*}
The proof is completed by joining the four bounds for $\EE|A_i|$ together and choosing $\varepsilon = \epsilon_1 = \epsilon_2$ and $f(\Dt) = 4 \diffc  (\tfrac{3}{2} \pntDisc_m)^{\alpha} \Dt^{1/2 - \varepsilon}$.
\end{proof}

\subsection{Proof of Theorem \ref{thm:strong_rate_stopped}} \label{sec:proof_under_H5}
In this proof and as before,  the positive constant $C$ may  change from line to line. It depends on $T$, $p$ and all the parameters in the hypotheses, but not on $\Dt$ and $\epsilon$.

To simplify the presentation, we consider in \ref{H:pieceloclip} only one point of discontinuity $\pntDisc_m$ for $b$, the other cases being a sum of similar contributions. 

From the five drift contributions $T_i$ in \eqref{eq:error_gamma_proto}, the two lasts $T_4$ and $T_5$ are residual. They can be bounded in a similar manner, with the help of the minoration \eqref{eq:scheme_minoration}, Young and Cauchy Schawrz inequalities. Note also that they  vanish with $b(0)$.   More precisely using \ref{H:polygrowth}, 
\begin{equation*}
\begin{aligned}
T_5 &  \leq  \diffc^2\frac{p-1}{p}\Dt\, Y_t^{2p} + C \  \Dt^{p} \ \oX_{\eta(t)}^{2p(\alpha-1)},\\
T_4 & \leq 2p b(0)\ \dt \   |Y_t|^{2p-1}  \ \Lg (\oX_{\eta(t)}^{\alphab-1} \vee 1) \leq \frac{2p-1}{2p}\sqrt{\Dt}\  Y_t^{2p} + C \  \Dt^{p}\ (1+\oX_{\eta(t)}^{2p(\alphab-1)}),
\end{aligned}
\end{equation*}
for some constant $C$ depending on $b(0)$, $\sigma$ and $p$.  In the above inequalities,  $\sup_{t\in[0,T]}(\EE[\oX_{t}^{2p(\alpha-1)}]+ \EE[\oX_{t}^{2p(\alphab-1)}])$ is  finite under \eqref{eq:condition_param_main_thm_1_1} or \eqref{eq:condition_param_thm_main_b_continuous_1}.  Therefore,
\begin{equation}\label{eq:estimation_T45}
\EE[|T_4| +\EE[|T_5|] \leq  \left\{ \left(\diffc^2\tfrac{p-1}{p} + \tfrac{2p-1}{2p}\right)\sqrt{\Dt} \ \EE[Y_t^{2p} ]   + 	C\, \Dt^{p }  \right\}\ind_{\{{b(0)} >  0\}}.
\end{equation}

The main idea of this proof is to make use of the control of some exponential moments of $\int_0^t X_s^{\alphab -1} ds$ and $\int_0^{t}\oX_{\texp\wedge\eta(s)}^{\alphab -1} ds$ given in  Propositions \ref{prop:XMoments} and  \ref{prop:schem_ExpoMoment} by the mean of a change in time formula in the dynamics of the error process $Y$. 

For the stopped process $\oX_{\cdot \wedge \texp}$, this exponential control is unconstrained on $\beta$, whereas it induces an additional constraint on $B_2$  for  the exact process $X$. Thus, a part of the game  consists in tracking the terms containing the factor $Y_t^{2p} X_t^{\alphab -1}$, and upper-bound whenever is possible, in favour of terms containing $Y_t^{2p} \oX_{\eta(t)}^{\alphab -1}$.

We first state the intermediate bound for the  sum $T_1+ T_2 + T_3$ in \eqref{eq:error_gamma_proto}.  
We consider first the term $T_3=2 p (2p -1) Y_t^{2p-2}   \oX_t^{2} \left[\frac{\sigma(X_t)}{X_t}-\frac{\sigma(\oX_{\eta(t)})}{\oX_{\eta(t)}}\right]^2$, exchanging $\oX_t$ with $X_t$, 
\begin{equation*}
T_3 \leq 4p(2p -1) \diffc^2 Y_t^{2p}  [X_t^{\alpha-1}\vee\oX_{\eta(t)}^{\alpha-1}]^2 + 4 p (2p -1) Y_t^{2p-2} (X_t-\oX_{\eta(t)})^2\left[\tfrac{X_t}{X_t - \oX_{\eta(t)}}\left(\tfrac{\sigma(X_t)}{X_t} - \tfrac{\sigma(\oX_{\eta(t)})}{\oX_{\eta(t)}} \right)\right]^2.
\end{equation*}
From Lemma \ref{lem:estimation_from_Hquatre}-\text{($\sigma$)},  $ \left[\frac{X_t}{X_t - \oX_{\eta(t)}}\left(\frac{\sigma(X_t)}{X_t} - \frac{\sigma(\oX_{\eta(t)})}{\oX_{\eta(t)}} \right)\right]^2 \leq 2((\alpha\diffcd)^2 + \diffc^2)(X_t\vee \oX_{\eta(t)})^{2(\alpha-1)}$. So, with Young and Cauchy-Schwarz inequalities, 
\begin{equation*}
\begin{aligned}
\EE[T_3 \ind_{\{t\leq \texp\}}] &\leq 4p(2p -1) (5\diffc^2 + 4(\alpha\diffcd)^2) \EE[Y_t^{2p}  (X_t\vee\oX_{\eta(t)})^{2(\alpha-1)}\ind_{\{t\leq \texp\}}]\\
&\qquad + 16 p (2p -1)((\alpha\diffcd)^2 + \diffc^2)\EE\left[Y_t^{2p-2}(\oX_t-\oX_{\eta(t)})^{2} (X_t\vee \oX_{\eta(t)})^{2(\alpha-1)}\ind_{\{t\leq \texp\}}\right]\\
&\leq C\ \EE[Y_{t\wedge\texp}^{2p}]+ 4p(2p -1) (5\diffc^2 + 4(\alpha\diffcd)^2) \EE[Y_{t\wedge\texp}^{2p}  (X_{t\wedge\texp}\vee\oX_{\eta({t\wedge\texp})})^{2(\alpha-1)}] \\
&\qquad + C\ \EE^{1/2}\left[(\oX_{t\wedge\texp}-\oX_{\eta({t\wedge\texp})})^{4p} \right]\EE^{1/2}\left[(X_{t\wedge\texp}\vee \oX_{\eta({t\wedge\texp})})^{4(\alpha-1)p}\right].
\end{aligned}
\end{equation*}

With the assumption that $\ind_{\{2\alpha-1\}}(\alphab)\, 2p (2\alpha-1) \leq \tfrac12 + \tfrac{B_2}{\diffc^2}$, which is covered by  \eqref{eq:condition_param_main_thm_1_1}, 
the control of the $L^{4p}$-local error is obtained from   Proposition \ref{prop:local_error}. This hypothesis also allowed to apply Proposition \ref{prop:XMoments}  and Lemma \ref{lem:Schememoments} for the control of the   $4(\alpha-1)p$-moments. We  obtain the estimation:
\begin{equation*}
\EE[T_3\ind_{\{t\leq \texp\}}] \leq    C \EE[Y_{t\wedge\texp}^{2p}]  +  4p(2p -1) (5\diffc^2 + 4(\alpha\diffcd)^2) \EE[Y_{t\wedge\texp}^{2p}  (X_{t\wedge\texp}\vee\oX_{\eta({t\wedge\texp})})^{2(\alpha-1)}] + C \Dt^p. 
\end{equation*}

Next, we consider $T_2 (t) 
= 2p Y_t^{2p-1}  \oX_t   \left( \frac{b(X_t)-b(0)}{X_t} - \frac{b(\oX_{\eta(t)})-b(0)}{\oX_{\eta(t)}} \right)$.   In view of Lemma \ref{lem:disc_contrib}, we decompose $T_2$ as follows, using the notation $A$ in \eqref{eq:aux_A}, 
\begin{equation*}
T_2 (t) = - 2p Y_t^{2p}   \frac{b(X_t)-b(0)}{X_t} + \  A  \ +  2p Y_t^{2p-1}  \left( b(X_t) - b(\oX_t) \right).  
\end{equation*}
So, with $T_1 = 2pY_t^{2p} ( \frac{b(X_t)-b(0)}{X_t} + 2(2p-1)\frac{\sigma^2(X_t)}{X_t^2})$, we have 
\begin{equation*}
T_1 (t) + T_2 (t) 
=  4 p (2p-1) Y_t^{2p}\frac{\sigma^2(X_t)}{X_t^2}  + \  A  \ +  2p Y_t^{2p}  \left( \frac{ b(X_t) - b(\oX_t) }{X_t - \oX_t}\right). 
\end{equation*}
Applying \ref{H:polygrowth} and \ref{H:control}, we obtain 
\begin{equation*}
T_1 (t) + T_2 (t) 
\leq  4 p (2p-1) Y_t^{2p} \diffc^2 X_t^{2(\alpha -1)}   + \  A  \ +  2p Y_t^{2p} \ \osLip  ( 1 + X_t^{\alphab -1} \vee \oX_t^{\alphab -1}). 
\end{equation*}
It remains to apply Lemma \ref{lem:disc_contrib} for the term $A$,  
accompanied with the constraint that $\Dt\leq \Delta(\varepsilon)$, for $\Delta(\varepsilon)$ in \eqref{eq:def_final_delta_epsilon},  and $\varepsilon\in(0,\tfrac12)$, while the condition \eqref{cond:lemma_discont} is covered by   \eqref{eq:condition_param_main_thm_1_1}  (and \eqref{cond:lemma_cont} by  \eqref{eq:condition_param_thm_main_b_continuous_1} when $b$ is continuous).

Adding the estimation of $T_3$ above and  \eqref{eq:estimation_T45} for $T_4 + T_5$, we can summarize the bound  we get as 
\begin{equation}
\begin{aligned}
\EE[\sum_{i=1}^5 T_i(t)] 
& \leq \EE[\ \KK_p(X_t,\oX_{\eta(t)}) \ Y^{2p}_t] +  C \Dt^{p-2\varepsilon p } +  C \EE\left[ |Y_t |^{2p-1}   
\ind_{  \left\{\oX_{[\eta(t), \eta(t) + \Dt]}\in\BB(\Dt^{1/2 - \varepsilon}) \right\}} \right], 
\end{aligned}
\end{equation}
with, from the application of Lemma \ref{lem:disc_contrib},  $\BB(\Dt^{1/2 - \varepsilon}) =  B(\pntDisc_m, 8\diffc  (\frac{3}{2} \pntDisc_m)^{\alpha}\ \Dt^{1/2-\varepsilon} )$,  and 
\begin{equation*}
\begin{aligned}
\KK_p(x,y) & = 
C 
+ 2p \KK_\star (y\vee x)^{\alphab-1}, \quad \text{with ~} \mathcal{K}_\star  : = \osLip +  2(2p -1) (6\diffc^2 + 4(\alpha\diffcd)^2).
\end{aligned}
\end{equation*}
We bound $ |Y_t|^{2p-1} \ind_{  \left\{\oX_{[\eta(t), \eta(t) + \Dt]}\in \BB(\Dt^{1/2 - \varepsilon})\right\}}$ as follows, discussing whenever $|Y_t| > \Dt^{\frac{1}{2} -\varepsilon}$ or not: 
\begin{equation*}
\begin{aligned}
&|Y_t|^{2p-1} \ind_{\{ \oX_{[\eta(t), \eta(t) + \Dt]} \in \BB(\Dt^{1/2 - \varepsilon})\}}  \\
& \leq \Dt^{(\frac{1}{2} -\varepsilon)(2p-1)} \ \ind_{\{ \oX_{[\eta(t), \eta(t) + \Dt]} \in \BB(\Dt^{1/2 - \varepsilon})\}}  
+ |Y_t|^{2p} \Dt^{-(\frac{1}{2} -\varepsilon)} \ind_{\{ \oX_{[\eta(t), \eta(t) + \Dt]} \in \BB(\Dt^{1/2 - \varepsilon})\}}.
\end{aligned}
\end{equation*}
From Lemma \ref{lem:occupation_time} applied  with the norm $L^1(\Omega)$ to the first term of the above inequality and the estimation of $\EE[\sum_{i=1}^5 T_i(t)] $,  we obtain 
\begin{equation*}
\EE[Y_{t\wedge\texp}^{2p}]
 \leq \EE\Big[\int_0^{t\wedge\texp}\widetilde{\KK}_p(X_s, \oX_{[\eta(s), \eta(s) + \Dt]}) \  Y_s^{2p}ds\Big] + C  \ \Dt^{p-2\varepsilon p },
\end{equation*}
with 
\begin{equation*}
\widetilde{\KK}_p(X_t, \oX_{[\eta(t), \eta(t) + \Dt]}) = \KK_p(X_t,\oX_t) + C \  \Dt^{-(\frac{1}{2} -\varepsilon)} \ind_{\{ \oX_{[\eta(t), \eta(t) + \Dt]} \in \BB( \Dt^{1/2 - \varepsilon} )\}}.
\end{equation*}
In order to apply the  Gronwall Lemma in presence of the  stochastic coefficient $\widetilde{\KK}_p(X_s, \oX_{[\eta(s), \eta(s) + \Dt]})$, we define now the stopping time 
$$\tau_\lambda = \inf\{t\geq0:~ {\bf\textrm H}(t)  \geq \lambda\}\wedge T \wedge \texp, $$
such that  ${\bf\textrm H}(\tau_\lambda) = \lambda$, when ${\bf\textrm H}^{-1}(\lambda) \leq T\wedge\texp$,  with $t\mapsto {\bf\textrm H}(t)$ defined as 
\begin{equation}\label{eq:def_h_general}
{\bf\textrm H}(t) =   \int_0^t  \widetilde{\KK}_p(X_s, \oX_{[\eta(s), \eta(s) + \Dt]})  \, ds. 
\end{equation}
For any $\lambda>0$, through the change of time $u = {\bf\textrm H}(s)$,  $\EE\left[\int_0^{\tau_\lambda}   Y_{s}^{2p} d {\bf\textrm H}(s) \right] =\EE\left[\int_0^{\lambda}  Y_{\tau_u}^{2p}  du\right]$, and we obtain 
$$ \EE[ Y_{\tau_{\lambda}}^{2p} ] \leq \EE\left[\int_0^{\lambda}
  Y_{\tau_u}^{2p}  du\right]  +  C  \Dt^{p-2\varepsilon p }. 
 $$
From Gronwall inequality we derive the first estimation
\begin{equation}\label{eq:estimatio_Y_tau}
\begin{aligned}
\EE[Y_{\tau_\lambda}^{2p}]&\leq   C \  \Dt^{p-2\varepsilon p}   \  \exp\{ \lambda\}.
\end{aligned}
\end{equation}
Now, choosing $\lambda = {\bf\textrm H}({T\wedge \texp})$,
\begin{equation*}
\begin{aligned}
\EE[Y_{T\wedge \texp}^{2p}] \leq 
\EE\left[\int_0^{{\bf\textrm H}({T\wedge \texp})}
  Y_{\tau_u}^{2p}  du\right]  +   C  \Dt^{p-2\varepsilon p }. 
\end{aligned}
\end{equation*}
It remains to bound the first term in the right-hand side above. Applying H\"older inequality with exponent $\frac{p+1}{p}$, 
\begin{equation*}
\EE\left[\int_0^{{\bf\textrm H}({T\wedge \texp})}
  Y_{\tau_u}^{2p}  du\right] \leq 
\int_0^{+\infty}
\left(\PP({\bf\textrm H}({T\wedge \texp})\geq u)\right)^{\tfrac{1}{p+1}}~\EE^{\tfrac{p}{p+1}}\left[Y_{\tau_u}^{2(p+1)}\right] du.
\end{equation*}
By using the preliminary upper bound \eqref{eq:estimatio_Y_tau}  (increasing the sufficient condition  in Lemma \ref{lem:disc_contrib} to get this upper bound for $\EE[Y_{\tau_u}^{2(p+1)}]$, with $\ind_{\{2\alpha-1\}}(\alphab)\, 2(p+1 +\Dt^{2 (p+1) \varepsilon} )\ (2\alpha-1) \leq \tfrac12 + \tfrac{B_2}{\diffc^2}$), we get 
\begin{equation*}
\EE^{\tfrac{p}{p+1}}\left[Y_{\tau_u}^{2(p+1)}\right] \leq 
C \ \Dt^{p -2\varepsilon p }  \ \exp\Big\{ u \frac{p}{p+1}\Big\}.
\end{equation*}
Next, applying Markov's inequality, we observe that, for any $\mu>0$, 
$$\PP({\bf\textrm H}({T\wedge \texp})\geq u) \leq \EE[\exp\{\mu \ {\bf\textrm H\, }({T\wedge \texp})\}] \ \exp(-\mu u).$$ 
Thus, considering $\mu > p$, we obtain
\begin{equation*}
\begin{aligned}
\EE\left[\int_0^{{\bf\textrm H}({T\wedge \texp})}
  Y_{\tau_u}^{2p}  du\right] 
&\leq C \ \Dt^{p -2\varepsilon p } \  
\EE^{\frac{1}{p+1}}[\exp\{\mu \ {\bf\textrm H\, }({T\wedge \texp})\}]  
\int_0^{+\infty}
\exp\left\{\frac{u \ p}{p+1} - \mu \frac{u}{p+1}\right\}   du\\
&\quad =  C \ \Dt^{p -2\varepsilon p} \   \left(\frac{p+1}{\mu -p}\right)  \ 
\EE^{\frac{1}{p+1}}[\exp\{\mu \ {\bf\textrm H\, }({T\wedge \texp})\}], 
\end{aligned}
\end{equation*}
and therefore, for an updated constant $C$, for all $t\in[0,T]$, 
\begin{equation}\label{eq:strng_error_bound}
\EE[Y_{t\wedge\texp}^{2p}]
\leq C \ \Dt^{p -2\varepsilon p } \left( 1 +  \left(\frac{p+1}{\mu -p}\right)  \ 
\EE^{\frac{1}{p+1}}\Big[\exp\Big\{\mu \ {\bf\textrm H\, }({T\wedge \texp})\Big\}\Big]\right). 
\end{equation}
Applying the H\"older inequality with exponent $\epsilon_0>1$, 
\begin{equation*}
\begin{aligned}
& \EE[\exp\{\mu {\bf\textrm H\, }({T\wedge \texp})\}] \\
& \quad \leq C \ \EE\left[\exp\left\{\mu\, 2p \mathcal{K}_\star  \int_0^{T\wedge \texp} (\oX_{\eta(s)}\vee X_s)^{\alphab-1} ds  + \mu\int_0^{T\wedge \texp} \Dt^{-(\frac{1}{2} -\varepsilon)}\ind_{\{ \oX_{[\eta(t), \eta(t) + \Dt]} \in \BB(\Dt^{1/2 - \varepsilon} )\}} ds 
\right\}\right] \\
& \quad \leq C\, \EE^{\tfrac1{\epsilon_0}}\left[\exp\left\{\mu\epsilon_0 \ 2p \mathcal{K}_\star \int_0^{T\wedge \texp}(\oX_{\eta(s)}\vee X_s)^{\alphab-1}ds\right\}\right] \\
 &\qquad\qquad \times \  \EE^{\frac{\epsilon_0-1}{\epsilon_0}}\left[\exp\left\{ \mu\tfrac{\epsilon_0}{\epsilon_0-1} \int_0^{T\wedge \texp} \Dt^{-(\frac{1}{2} -\varepsilon)} \ind_{\{ \oX_{[\eta(t), \eta(t) + \Dt]} \in \BB(\Dt^{1/2 - \varepsilon} )\}} ds 
\right\}\right]. 
\end{aligned}
\end{equation*}
From Lemma \ref{lem:occupation_time},  we derive the following bound for the second exponential term: 
\begin{equation*}
\EE \left[ \exp\left\{\mu\tfrac{\epsilon_0}{\epsilon_0-1} \Dt^{-(\frac{1}{2} -\varepsilon)} \int_0^t\ind_{\{ \oX_{[\eta(s), \eta(s) + \Dt]} \in \BB(\Dt^{1/2 - \varepsilon})\}} ds \right\} \right]  \leq \exp\left\{C\,\mu\tfrac{\epsilon_0}{\epsilon_0-1} \right\}.
\end{equation*}
It remains to apply Proposition \ref{prop:schem_ExpoMoment} for the first term, provided that $\mu>p$ and $\epsilon_0>1$.  
Taking $\epsilon_0 = 1+\tfrac1p$ and $\mu = p+1$,  imposing 
$$2(p+1)^2  \mathcal{K}_\star < B_2,\quad\mbox{when }(b(0),\alpha)\notin (0,\tfrac{\diffc^2}2)\times[1,\tfrac32],$$
we obtain 
\begin{equation*}
\sup_{t\in[0,T]}\EE[Y_{t\wedge\texp}^{2p}]\leq C~\Dt^{p - 2\varepsilon p },
\end{equation*}
for any $\varepsilon\in(0,1/2)$ and $\Dt\leq \Delta(\varepsilon)$. 

Applying the same reasoning as in Remark \ref{rem:exponentialPrototype}-(ii), with the help of the Lenglart's inequality, we obtain 
\[\EE\left[\sup_{0\leq t\leq T}Y_{t\wedge\texp}^{p}\right]\leq C\,\Dt^{\frac{p}{2} (1 - 2\varepsilon)}. \]
Here we choose  $\frac{p}{\epsilon} = \frac{1}{2}$, in order to avoid further reinforcing the conditions of convergence. 

When $b$ is continuous, the proof is simpler. Repeating the steps, to apply Proposition \ref{prop:local_error} and Lemma \ref{lem:disc_contrib}, we need only that $\ind_{\{2\alpha-1\}}(\alphab) 2(p+1) (2\alpha -1)    \leq \tfrac12+\tfrac{B_2}{\diffc^2}$  to get 
\begin{equation*}
\EE[Y_{t\wedge\texp}^{2p}]
\leq \EE\Big[\int_0^{t\wedge\texp} \left( C 
+ 2p  \ \KK_\star  (X_s \vee  \oX_{\eta(s)})^{\alphab-1}\right) Y_s^{2p} ds \Big]  +  C \Dt^{p}.   
\end{equation*}
Proceeding as before with the change of time argument, we end with the expression \eqref{eq:strng_error_bound} when $\varepsilon=0$,  and 
\begin{equation*}
\begin{aligned}
\EE[\exp\{\mu {\bf\textrm H\, }({T\wedge \texp})\}]
\leq C \ \EE\left[\exp\left\{\mu\, 2p \  \mathcal{K}_\star  \int_0^{T\wedge \texp} (\oX_{\eta(s)}\vee X_s)^{\alphab-1} ds 
\right\}\right].
\end{aligned}
\end{equation*}
Taking  $\mu = p+1$,  imposing 
$2 p \ (p+1) \ \mathcal{K}_\star < B_2$, when $(b(0),\alpha)\notin (0,\tfrac{\diffc^2}2)\times[1,\tfrac32]$, we obtain 
\[\sup_{t\in[0,T]}\EE[Y_{t\wedge\texp}^{2p}]\leq C~\Dt^{p}.\]

As before, from Lenglart's inequality with $r = \frac{1}{2}$ and  under the same conditions, we get    
\[\EE\left[\sup_{0\leq t\leq T}Y_{t\wedge\texp}^{p}\right]\leq C\,\Dt^{p/2}.\]

\subsection{Strong convergence rate of the unstopped scheme} 

In exchange for more control over the drift behaviour, we extend the convergence result to the unstopped scheme. We also extend the parameter domain of  Theorem \ref{thm:strong_rate_stopped}  for which we prove the convergence rate.  
Similarly to the weak error analysis in \cite{BoJaMa2021}, we  introduce some smoothness on $b$ such that, at least for large values of $x$,   $b'(x)$ exists and inherits from the inward effect assumed in \ref{H:control}. In other words,  $b'(x)$ is bounded by a polynomial function of order $\alphab -1$, with negative main coefficient of large enough amplitude:
\begin{hypo}{\hspace{-0.15cm}\Blue{\bf  \textrm {Control on $b'$. (\ref{H:sch_control_b_deriv}).}}}
\makeatletter\def\@currentlabel{ {\bf\textrm H}$_{\mbox{\scriptsize\bf\textrm{Control\_$b'$}}}$}\makeatother
\label{H:sch_control_b_deriv} 
For the same  $\Lg$ constant in \ref{H:polygrowth}, there exist positive constants ${B_1^\prime}$ and a point $\pntCompact> \pntDisc_m$,   such that, for any $x\geq \pntCompact$ (so outside any discontinuity interval of $b$), 
\begin{equation}\label{cond:control_b'}
	b'(x) \leq {B_1^\prime} - \alphab \Lg\  x^{\alphab -1}. 
\end{equation}
\end{hypo}
A typical case where \ref{H:control} and \ref{H:sch_control_b_deriv} are together is when $b$ is a polynomial function, and so \ref{H:sch_control_b_deriv} is just the derivative of \ref{H:control}. However, it is easy to get out of this type of case, for example when the polynomial function  has oscillating coefficients (see Section \ref{sec:num}). 
Under this additional hypothesis, we prove the following convergence result, the proof  of  which is  detailed in Appendix \ref{sec:proof thm regular}. 

\begin{theorem}\label{thm:strong_rate_non_stopped}
Assume \ref{H:pieceloclip}, \ref{H:polygrowth}, \ref{H:control}, \ref{H:sch_control_b_deriv} and \ref{H:diffusion_deriv}. 
For all $0 < \varepsilon < \frac12$, 
there exists $0 < \Delta(\varepsilon) \leq 1$,  estimated with \eqref{eq:def_final_delta_epsilon},  such that for all $\Delta t \leq \Delta(\varepsilon)$,   for $p\geq1$ any integer satisfying 
\begin{align}
\ind_{\{2\alpha-1\}}(\alphab)\,\left(  \{( \ 2(p+1) + \Dt^{2 (p+1) \varepsilon}  ) \ (2 \alpha -1) \} \vee \{  2\ (2(p+1) -  \tfrac34) \} \ \vee \ \{2\alpha(p+1)\}\right)   \leq \frac12+\frac{B_2}{\diffc^2}  \label{eq:condition_thm_unstopped_1} \\
\quad\mbox{and}\quad
\ind_{\{2\alpha-1\}}(\alphab)\  4(2(p+1) -1) (\diffc^2 +  (\alpha\diffcd)^2)   \leq L_G \ ((\alphab -1) \wedge 1) .
\label{eq:condition_thm_unstopped_3}
\end{align}
there exists  some positive constant $C(p)$, independent of $\Dt$ and $\varepsilon$,   such that
\begin{equation*}
\sup_{t\in[0,T]} \left\|X_t-\oX_t \right\|_{L^{2p}(\Omega)} + \Big\|\sup_{t\in[0,T]} |X_t-\oX_t| \Big\|_{L^{p}(\Omega)} \leq C(p) \, \Dt^{\frac{1}{2} - \varepsilon}. 
\end{equation*}
Furthermore, when $b$ is continuous on $\RR^+$, 
replacing  \eqref{eq:condition_thm_unstopped_3} and  \eqref{eq:condition_thm_unstopped_1}   by the condition
\begin{equation}\label{eq:condition_thm_unstopped_continuous_1}
\ind_{\{2\alpha-1\}}(\alphab)\,\left(  \{ p \ (4 \alpha -3) \}\vee\{2\alpha p \} \vee \{  2\ (2p-  \tfrac34) \} \right)  \leq \frac12+\frac{B_2}{\diffc^2}, 
\end{equation}
the following  convergence takes place for  $\Dt \leq 1$: 
\begin{equation*}
\sup_{t\in[0,T]} \left\|X_t-\oX_t \right\|_{L^{2p
}(\Omega)} + \Big\|\sup_{t\in[0,T]} |X_t-\oX_t| \Big\|_{L^{p}(\Omega)} \leq C(p) \, \Dt^{\frac{1}{2}}. 
\end{equation*} 
\end{theorem}

\begin{rem}\label{rem:strong_versus_weak_condition}
The weak error analysis in \cite{BoJaMa2021} gives a convergence (with a rate one) whenever  $b$ in $C^4$, $\frac{B_2}{\diffc^2} + \frac{1}{2} > 5 \alpha + C$, where the positive constant $C$ depends on the polynomial decays of the successive derivatives of $b$. Moreover, some of the parameters were forbidden. Typically, when $b(0)>0$, the  theoretical weak convergence was proven  under the condition that $\alpha > 3/2$. 
With Theorem \ref{thm:strong_rate_non_stopped}, the quadratic error is going to zero (with rate 1/2) for a larger set of parameters, allowing all the values of $b(0)\geq 0$ for all the values of $\alpha>1$. 
\end{rem}

The following lemma, the proof of which is postponed in Appendix  \ref{sec:addproofs}, summarizes how we use the hypotheses \ref{H:sch_control_b_deriv} and \ref{H:diffusion_deriv} in order to derive the strong error.  Hypothesis \ref{H:sch_control_b_deriv} strengthens  \ref{H:control},  putting even more weight on the inward behaviour of the drift $b$. The factor $\alphab$ in front of $\Lg$ dramatically simplifies  the proof of this technical Lemma, by using the conservation of inequality by integration. On the other hand, hypothesis \ref{H:diffusion_deriv} maintains the polynomial growth of the map $x\mapsto \tfrac{\sigma(x)}x$ while keeping the driven power in terms of the parameter $\alpha$. 
\begin{lem}\label{lem:estimation_from_Hquatre}
\hspace{2em}
\begin{itemize}
\item[$(b)$.]
Assume \ref{H:pieceloclip}, \ref{H:polygrowth} and  \ref{H:sch_control_b_deriv}. 
For  all $x,y\in (0,+\infty)$, we have 
\end{itemize}
\begin{equation}\label{eq:estim_lemma_cases}
\begin{aligned}
& \frac{y}{x-y}\left(\frac{b(x)-b(0)}{x} - \frac{b(y)-b(0)}{y} \right) \leq 
\Cczero - \ \Lg ((\alphab -1) \wedge 1)  \ y \ (x\vee y ) ^{\alphab -2}, 
\end{aligned}
\end{equation}
and, with  $\pntCompact$ given in \ref{H:sch_control_b_deriv}, such that $\pntCompact>\pntDisc_m$, we  estimate the non-negative constant $\Cczero$ as
\begin{equation}\label{eq:constants_Ccercled}
\begin{aligned}
\Cczero & = 
\left( {B_1^\prime}\vee [\osLip(\pntCompact^{\alphab-1}+1)] + \Lg (\alphab \pntCompact^{\alphab-1}+1) \right)
\ \vee \ 
\big(\osLip +\Lg\big) ((\pntCompact^{\alphab-1}\vee 1)+1). 
\end{aligned}
\end{equation}
When  \eqref{cond:control_b'}-\ref{H:sch_control_b_deriv} is satisfied in the whole space $\RR^+$, then $b$ is locally Lipschitz on $\RR^+$, and   \eqref{eq:estim_lemma_cases} is  valid with the constant $\Cczero  = 
\left({{B_1^\prime}} + \Lg \right)$.
\begin{itemize}
\item[$(\sigma)$.] 
Assume \ref{H:diffusion_deriv}. Then for all $x,y$ in $\RR^+$
\begin{equation}\label{eq:estim_lemma_cases_sigma}
\frac{y}{(x-y)^2}\left(\frac{\sigma(x)}x - \frac{\sigma(y)}y\right)^2
\leq  2 ( \diffc^2 +  \alpha^2 \diffcd^2 ) 
\  (x\vee y)^{2\alpha -3}. 
\end{equation}
\end{itemize}
\end{lem}

\subsection{Asymptotic behaviour and stability of the \expSch scheme on a prototypical case}\label{sec:stability}

Understanding the asymptotic behaviour of a stochastic system is crucial in various applications  as it enables the identification of suitable actions to be taken in different situations.  
 For instance, it is relevant in modelling asset returns or volatility~\cite{Ait96,Spr}, estimating the proportion of individuals infected during an epidemic~\cite{gray2011stochastic}, and studying the equilibrium regime in turbulent kinetic energy modeling~\cite{bossy2022instantaneous}. Equally important is the guarantee that an approximation scheme will behave like the model. 

We restrict our analysis to the prototypical case where the SDE \eqref{eq:IntroSDE} has  drift $b$ and diffusion $\sigma$ of the form 
\begin{equation}\label{stab:prototipical_case}
b(x) = b(0) + B_1 x - B_2\  x^{2\alpha-1}\quad\mbox{and }\quad \sigma(x) = \diffc x^{\alpha}.
\end{equation}
The stability results presented in this section  can be extended straightforwardly  to the case $\alphab>2\alpha-1$. 
The following proposition states that the solution $X$ of   \eqref{eq:IntroSDE},  with \eqref{stab:prototipical_case}, approaches a steady state $\xi^*$ an infinite number of times in the asymptotic regime. The same behaviour is expected to be reached by the approximation $\oX$,  as the time-step $\Dt$ tends to zero.
\begin{prop}\label{prop:asympt_sol}
Let $B_1,b(0)$ some non-negative constants and $B_2>0$, with $2(\alpha-1)\leq 1 +  2\tfrac{B_2}{\diffc^2}$. Then the solution of \eqref{eq:IntroSDE} with the coefficients $b$ and $\sigma$ defined in  \eqref{stab:prototipical_case} satisfies
\begin{equation*}
\liminf_{t\rightarrow+\infty} X_t \leq \xi^*\qquad\mbox{ and }\qquad \limsup_{t\rightarrow+\infty} X_t \geq \xi^*,\qquad \PP\mbox{-a.s.}
\end{equation*}
where $\xi^*$ is the unique positive solution of $\{x \in \RR^+ ;
b(x)/x - \tfrac{1}{2}(\sigma(x)/x))^2= 0\}$.
\end{prop}
\begin{proof}
For the sake of completeness, we describe below the main steps of the proof; the details are obtained by following the ideas proposed in  \cite{gray2011stochastic}. 

The It\^o formula applied to the map $\log(X_t)$ in $[0,t]$ leads to
\begin{equation}\label{eq:Ito_stability}
\frac{\log(X_t)}t = \frac{\log (x)}t + \frac1t\int_0^t \varphi(X_s)ds + \frac{\diffc}t\int_0^t X_{s}^{\alpha-1}dW_s,
\end{equation}
where $\varphi(x) =  \frac{b(0)}{x}+B_1 - \left(B_2+\frac{\diffc^2}2\right)x^{2(\alpha-1)}$.  Notice that $\varphi'(x)<0$ for all $x>0$, implying that 
\[\lim_{x\rightarrow+\infty}\varphi(x) < 0 < \lim_{x\rightarrow0^+}\varphi(x).\]
Therefore, there exists a unique positive $\xi^*$ solution to $\{x;\varphi(x) = 0\}$. Moreover $\varphi$ changes signs in the neighbourhood of $\xi^*$.
\medskip
 
Let us assume that $\limsup_{t\rightarrow+\infty} X_t <\xi^*$. Then we consider $\epsilon>0$ small enough such that, for $\omega_0\in \{\omega:~\limsup_{t\rightarrow+\infty} X_t(\omega)\leq \xi^*-2\epsilon\}$ there exists $T(\omega_0)$ such that,
\begin{equation*}
	X_t(\omega_0)\leq \xi^*-\epsilon, \mbox{ for all }t\geq T(\omega_0).
\end{equation*}
Since $x\mapsto\varphi(x)$ is non-increasing, we obtain
\begin{equation}\label{eq:drift_bound_solution}
	\varphi(X_t(\omega_0))\geq \varphi(\xi^*-\epsilon), \mbox{ for all }t\geq T(\omega_0).
\end{equation}
From the Law of Large Numbers (LLN) for continuous martingales, we also have
\begin{equation*} 
\omega_0\in \{\omega:~\limsup_{t\rightarrow+\infty}X_{t}(\omega)\leq \xi^*-2\epsilon,\quad \mbox{ and }\quad\lim_{t\rightarrow+\infty}\tfrac1t\int_0^tX_{s}^{\alpha-1}(\omega)dW_s(\omega)=0\}.
\end{equation*}
Then, from Equation \eqref{eq:Ito_stability} and \eqref{eq:drift_bound_solution}, we obtain
\begin{equation*}
\begin{aligned}
\liminf_{t\rightarrow+\infty}\frac{\log(X_t(\omega_0))}t \geq \liminf_{t\rightarrow+\infty}\frac1t\int_0^{T(\omega_0)}\varphi(X_s(\omega_0))ds +\liminf_{t\rightarrow+\infty}\frac{1}t\int_{T(\omega_0)}^{t}\varphi(\xi^*-\epsilon)ds =\varphi(\xi^*-\epsilon)>0.
\end{aligned}
\end{equation*}
The above implies that $\limsup_{t\rightarrow+\infty}X_t(\omega_0)\geq\liminf_{t\rightarrow+\infty}X_t(\omega_0) = +\infty,$ which is a contradiction since we assumed $\limsup_{t\rightarrow+\infty} X_t <\xi^*<+\infty$. Hence, it must be the case that, $\PP$-a.s.,  $\limsup_{t\rightarrow+\infty} X_t \geq \xi^*$.

Following a similar argument, it can be shown that $\liminf_{t\rightarrow+\infty} X_{t}\leq\xi^*$. 
\end{proof}

\begin{rem} For a more general drift  $b$ satisfying \ref{H:polygrowth}, it is still  possible to show that the function $\psi(x) = \frac{b(x)}x - \tfrac{1}{2}(\sigma(x)/x)^2$ satisfies $\lim_{x\rightarrow+\infty}\psi(x)<0$ and $\lim_{x\rightarrow0^+}\psi(x)>0$, i.e. $\psi(x) = 0$ has at least one solution $\xi^*$. 
However, in order to extend the results from Proposition \ref{prop:asympt_sol}, we need $\xi^*$ to be unique. For this we could consider some additional hypotheses on the growth of $\tfrac{d}{dx}\left(b(x)/x\right)$ or impose the uniqueness for the root of  $\psi$.
\end{rem}

\subsubsection{Stability of the \expSch scheme}

For the prototypical  case  \eqref{stab:prototipical_case},  given an  homogeneous $N$-partition of the time interval $[0,T]$   with  time-step $\Delta t=t_{n+1}-t_n$, the \expSch scheme simplifies as
\begin{equation}\label{eq:proto_scheme_continuous}
	d\oX_t = \left(\overline{X}_t-b(0){\dt}\right)\Big[(B_1-B_2\overline{X}_{\eta(t)}^{2(\alpha-1)})\ dt+\diffc\ \overline{X}_{\eta(t)}^{\alpha-1}dW_t\Big]+b(0) dt.
\end{equation}

\begin{prop}\label{prop:asympt_disc}
	Let $B_1,b(0)$ some non-negative constants, $B_2>0$ such that $2(\alpha-1)\leq1+2\tfrac{B_2}{\diffc^2}$ and $\Dt\leq\frac1{B_1}$. Then, the scheme $\oX$ solution to \eqref{eq:proto_scheme_continuous} satisfies
	\[\liminf_{t\rightarrow+\infty} \oX_{t} \leq \xi_{\Dt},\qquad \mbox{ and }\qquad\limsup_{t\rightarrow+\infty} \oX_{t} \geq \xi^{\Dt},\qquad \PP\mbox{-a.s.}\] 
	where $\xi^{\Dt}$ is the unique positive solution of  $\{x \in \RR^+ ;\varphi^{\Dt}(x) := 
\frac{b(x)}{x} - \tfrac{1}{2}(\frac{\sigma(x)}{x}))^2 - \frac{b(0)}{x}B_1\Dt= 0\}$,  
and $\xi_{\Dt}$ is the unique positive solution of 
$\{x \in \RR^+ ;\varphi_{\Dt}(x) := 
\frac{b(x)}{x} - \tfrac{1}{2}(\frac{\sigma(x)}{x}))^2 + b(0)(\diffc^2+B_2)x^{2\alpha-3} \Dt= 0\}$. 
\end{prop}
\begin{proof}
Define the map $\varphi(t,x) = B_1 - \left(1-\frac{b(0)\dt}x\right)\left(\frac{\diffc^2}2+B_2-\frac{\diffc^2}2\frac{b(0)\dt}x\right)x^{2(\alpha-1)}+\tfrac{b(0)(1-B_1\Dt)}x$. 
Since, for all $t>0$, $\oX_t> \oX_t - b(0)\dt\geq0$, we have 
\begin{equation}\label{eq:first_funct_bound}
\varphi(t,\oX_t)\geq\varphi(0,\oX_t) = \varphi^{\Dt}(\oX_t).
\end{equation}
Applying It\^o formula and reorganizing some terms, we get
\begin{equation*}
\begin{aligned}
\log(\oX_t) &= \log (x) + \int_0^t \frac{\left(\overline{X}_s-b(0){\delta(s)}\right)}{\oX_s}\left(B_1 - B_2\oX_{\eta(s)}^{2(\alpha-1)}\right)ds +\int_0^t \frac{b(0)}{\oX_s}ds + \int_0^t \frac{\diffc\left(\overline{X}_s-b(0){\delta(s)}\right)}{\oX_s}\overline{X}_{\eta(s)}^{\alpha-1}dW_s\\
&\qquad\qquad-\frac{\diffc^2}2\int_0^t\frac{\left(\overline{X}_t-b(0){\delta(s)}\right)^2}{\oX_s^2}\oX_{\eta(s)}^{2(\alpha-1)}ds\\
&= \log (x) + \int_0^t\left[B_1- \frac{\left(\overline{X}_s-b(0){\delta(s)}\right)}{\oX_s}\left(\frac{\diffc^2}2\frac{\left(\overline{X}_s-b(0){\delta(s)}\right)}{\oX_s}+B_2\right)\oX_{\eta(s)}^{2(\alpha-1)}\right]ds +\int_0^t \frac{b(0)(1-B_1\delta(s))}{\oX_s}ds \\
&\qquad\qquad+ \int_0^t \frac{\diffc\left(\overline{X}_s-b(0){\dt}\right)}{\oX_s}\overline{X}_{\eta(s)}^{\alpha-1}dW_s.
\end{aligned}
\end{equation*}
If we assume $\Dt$ is such that $\Dt\leq\frac1{B_1}$, then, for all $t\geq0$, $1-B_1\dt\geq1-B_1\Dt\geq0$ and 
{\small{
\begin{equation*}
\begin{aligned}
\log(\oX_t) &\geq\log (x) + \int_0^t\left[B_1- \frac{\left({\oX_s\vee\oX_{\eta(s)}}-b^+(0){\delta(s)}\right)}{{\oX_s\vee\oX_{\eta(s)}}}\left( \frac{\diffc^2}2\frac{\left({\oX_s\vee\oX_{\eta(s)}}-b^+(0){\delta(s)}\right)}{\oX_s\vee\oX_{\eta(s)}}+B_2\right)({\oX_s\vee\oX_{\eta(s)}})^{2(\alpha-1)}\right]ds \\
				&\qquad+\int_0^t\frac{b^+(0)(1-B_1\Dt)}{\oX_s\vee\oX_{\eta(s)}}ds + \int_0^t \frac{\diffc\left(\overline{X}_s-b^+(0){\dt}\right)}{\oX_s}\overline{X}_{\eta(s)}^{\alpha-1}dW_s\\
				&\quad=\log (x) + \int_0^t\varphi(s,\oX_s\vee\oX_{\eta(s)})ds+ \int_0^t \frac{\diffc\left(\overline{X}_s-b^+(0){\dt}\right)}{\oX_s}\overline{X}_{\eta(s)}^{\alpha-1}dW_s\\
				&\qquad \geq \log (x) + \int_0^t\varphi^{\Dt}(\oX_s\vee\oX_{\eta(s)})ds+ \int_0^t \frac{\diffc\left(\overline{X}_s-b^+(0){\dt}\right)}{\oX_s}\overline{X}_{\eta(s)}^{\alpha-1}dW_s,
\end{aligned}
\end{equation*}
}}
where the last inequality is due to \eqref{eq:first_funct_bound}.
	
It is easy to verify that the map $x\mapsto \varphi^{\Dt}(x)$ is strictly decreasing in $(0,+\infty)$ and satisfies
\[\lim_{x\rightarrow+\infty}\varphi^{\Dt}(x) < 0 < \lim_{x\rightarrow0^+}\varphi^{\Dt}(x).\]
Therefore, there exists a unique (positive) solution $\xi^{\Dt}$ to the equation $\varphi^{\Dt}(\xi^{\Dt}) = 0,$ and $\varphi^{\Dt}(\xi^{\Dt}+\epsilon)<0<\varphi^{\Dt}(\xi^{\Dt}-\epsilon)$ for any $\epsilon>0.$

\medskip
We assume now that $\limsup_{t\rightarrow+\infty} \oX_{t} <\xi^{\Dt}$, and consider $\epsilon>0$ such that
\[\PP(\{\omega:~\limsup_{t\rightarrow+\infty} \oX_{\eta(t)}(\omega) = \limsup_{t\rightarrow+\infty}\oX_t(\omega)\leq \xi^{\Dt}-2\epsilon\})>\epsilon.\]
Take the pair $(\omega_0, T(\omega_0))$ and $\epsilon$ small enough such that, for all $t\geq T(\omega_0),$ $\oX_{t}(\omega_0)\vee\oX_{\eta(t)}(\omega_0)\leq \xi^{\Dt}-\epsilon$. From the monotonicity of $\varphi^{\Dt}(\cdot)$, we obtain
\begin{equation}\label{eq:drift_bound_1}
\varphi^{\Dt}(\oX_{t}(\omega_0)\vee\oX_{\eta(t)}(\omega_0))\geq \varphi^{\Dt}(\xi^{\Dt}-\epsilon)>0,\qquad \mbox{for all }t\geq T(\omega_0).
\end{equation}
Thus, from LLN for continuous martingales,  we can consider $\omega_0$ such that, 
\begin{equation*}
\liminf_{t\rightarrow+\infty}\frac{\log(\oX_t(\omega_0))}t \geq\liminf_{t\rightarrow+\infty}\frac{1}t\int_{T(\omega_0)}^{t}\varphi^{\Dt}(\xi^{\Dt}-\epsilon)ds=\varphi(0,\xi^{\Dt}-\epsilon)>0,
\end{equation*}
implying $\limsup_{t\rightarrow+\infty}\oX_{t}(\omega_0) = +\infty,$ and contradicting the inequality $\limsup_{t\rightarrow+\infty} \oX_{t} <\xi^{\Dt}$.
	
In order to prove $\liminf_{t\rightarrow+\infty} \oX_{t} \leq \xi_{\Dt}$ we first bound $\log(\oX_t)$ from above:
\small{
\begin{equation*}
\begin{aligned}
\log(\oX_t) &= \log (x) + \int_0^t\left[B_1- \left(1-\tfrac{b(0){\delta(s)}}{\oX_s}\right)\left(\frac{\diffc^2}2+B_2-\frac{\diffc^2}2\frac{b(0){\delta(s)}}{\oX_s}\right)\oX_{\eta(s)}^{2(\alpha-1)}\right]ds +\int_0^t \frac{b(0)(1-B_1\delta(s))}{\oX_s}ds \\
&\qquad\qquad+ \int_0^t \frac{\diffc\left(\overline{X}_s-b(0){\dt}\right)}{\oX_s}\overline{X}_{\eta(s)}^{\alpha-1}dW_s\\
&\leq \log (x) + \int_0^t\left[B_1- \left(1-\tfrac{b(0){\delta(s)}}{\oX_s}\right)\left(\frac{\diffc^2}2+B_2-\frac{\diffc^2}2\frac{b(0){\delta(s)}}{\oX_s}\right)(\oX_{s}\wedge\oX_{\eta(s)})^{2(\alpha-1)}\right]ds +\int_0^t \frac{b(0)}{\oX_s\wedge\oX_{\eta(s)}}ds \\
&\qquad\qquad+ \int_0^t \frac{\diffc\left(\overline{X}_s-b(0){\dt}\right)}{\oX_s}\overline{X}_{\eta(s)}^{\alpha-1}dW_s. 
\end{aligned}
\end{equation*}
}
Then, using that $\oX_t\geq b(0)\dt$, we get
\begin{equation*}
\begin{aligned}
\log(\oX_t) 
&\leq \log (x) + \int_0^t\left[B_1- \left(\frac{\diffc^2}2+B_2\right)(\oX_{\eta(s)}\wedge\oX_{s})^{2(\alpha-1)} +
\frac{b(0){\delta(s)}}{\oX_s}\left(\diffc^2+B_2\right)(\oX_{\eta(s)}\wedge\oX_{s})^{2(\alpha-1)}\right]ds \\
&\qquad\qquad+\int_0^t \frac{b(0)}{\oX_s\wedge\oX_{\eta(s)}}ds+ \int_0^t \frac{\diffc\left(\overline{X}_s-b(0){\dt}\right)}{\oX_s}\overline{X}_{\eta(s)}^{\alpha-1}dW_s\\
& = \log (x) + \int_0^t\varphi_{\Dt}(\oX_{s}\wedge\oX_{\eta(s)})ds + \int_0^t \frac{\sigma\left(\overline{X}_s-b(0){\dt}\right)}{\oX_s}\overline{X}_{\eta(s)}^{\alpha-1}dW_s.
\end{aligned}
\end{equation*}

Notice that $\lim_{x\rightarrow0^+}\varphi_{\Dt}(x) =+\infty$ and $\lim_{x\rightarrow+\infty}\varphi_{\Dt}(x) =-\infty$, form which we deduce that there exist at least one positive solution to the equation $\varphi_{\Dt}(\xi_{\Dt}) = 0$, although  $\varphi_{\Dt}$ is no more non-increasing (as show in Figure~\ref{fig:psi_function_stability}). Let us consider the case $\varphi_{\Dt}$ not monotone and take $x_0$ a critical point of the map. Then, using  $\varphi_{\Dt}'(x_0) = 0$, we deduce that
\begin{equation*}
\varphi_{\Dt}(x_0)
= B_1+\frac{(2\alpha-1)b(0)}{2(\alpha-1)x_0}+\frac{b(0){\Dt}\left(\diffc^2+B_2\right)}{2(\alpha-1)}x_0^{2\alpha-3} + \frac{x_0}{2(\alpha-1)}\varphi_{\Dt}'(x_0)>0,
\end{equation*}
and so either $\varphi_{\Dt}$ is non-increasing or its minimum is positive (see Figure \ref{fig:psi_function_stability}). As a consequence, there exists a unique (positive) solution $\xi_{\Dt}$ to the equation $\varphi_{\Dt}(\xi_{\Dt}) = 0,$ and 
\begin{equation}\label{eq:above_root}
\varphi_{\Dt}(\xi_{\Dt}+\epsilon)<0<\varphi_{\Dt}(\xi_{\Dt}-\epsilon),\quad \mbox{ for any }\epsilon>0.
\end{equation}
Assuming that $\liminf_{t\rightarrow+\infty} \oX_{t} >\xi_{\Dt}$, we can consider $\epsilon>0$ such that
\[\PP(\{\omega:~\liminf_{t\rightarrow+\infty} \oX_{\eta(t)}(\omega) = \liminf_{t\rightarrow+\infty}\oX_t(\omega)\geq \xi_{\Dt}+2\epsilon\})>\epsilon.\]
Take the pair $(\omega_1, T(\omega_1))$ and $\epsilon$ small enough such that, 
\begin{equation*}
\oX_{t}(\omega_1)\wedge\oX_{\eta(t)}(\omega_1)\geq \xi_{\Dt}+\epsilon,\qquad\mbox{ for all }t\geq T(\omega_1).
\end{equation*}
\begin{figure}[ht!]
\centering
\begin{tikzpicture}[scale=0.7]
\begin{axis}[
	    axis lines = center,
	    ymin=-4, ymax=11,
	    xmin=-1, xmax=4.5,
	]
	\addplot [
	    domain=0:3, 
	    samples=100, 
	    color=red,
	    ultra thick
	]
	{1+(1/x)-((x^4)/2)+((5*x^3)/4)};
	\draw[blue, very thick] (axis cs:3.2,-0.3) -- (axis cs:3.2,0.3);
	\fill[color = blue!60, opacity=0.2] (axis cs:3.2,-0.3) rectangle (axis cs:4.5,0.3);
	\node[label={45:{$\xi_{\Delta t}$}},circle,fill,inner sep=1.5pt] at (axis cs:2.64,0) {};
	\node[blue,label={45:{$\xi_{\Delta t}+\epsilon$}},circle,fill,inner sep=1.5pt] at (axis cs:3.2,0) {};
	\node[label={265:{$x$}}] at (axis cs:4.5,-0.3) {};
	\node[label={0:{$\psi(x)$}}] at (axis cs:0.2,6) {};
	\end{axis}
	\end{tikzpicture}
	\caption{Behaviour of function $\varphi_\Dt(x) = B_1- (\frac{\diffc^2}2+B_2)x^{2(\alpha-1)}+\frac{b(0)}{x}+b(0)\left(\diffc^2+B_2\right)x^{2\alpha-3} \Dt$.\label{fig:psi_function_stability}}
\end{figure}
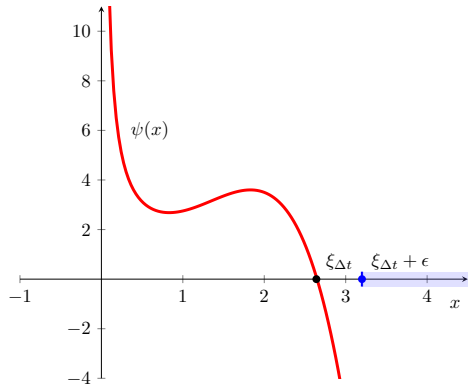
From the behaviour of $x\mapsto\varphi_{\Dt}(x)$ in the interval $(\xi_{\Dt}+\epsilon,+\infty)$ (see Figure \ref{fig:psi_function_stability}), we obtain
\begin{equation}\label{eq:drift_bound_2}
\varphi_{\Dt}(\oX_{t}(\omega_1)\wedge\oX_{\eta(t)}(\omega_1))\leq \varphi_{\Dt}(\xi_{\Dt}+\epsilon)<0,\quad \mbox{for all }t\geq T(\omega_1).
\end{equation}
Thus, from LLN for martingales and the inequality \eqref{eq:drift_bound_2}, we can consider $\omega_1$ such that, 
\begin{equation*}
\limsup_{t\rightarrow+\infty}\frac{\log(\oX_t(\omega_1))}t \leq\varphi_{\Dt}(\xi_{\Dt}+\epsilon)<0,
\end{equation*}
implying $\liminf_{t\rightarrow+\infty}\oX_{t}(\omega_0) = 0,$ and contradicting the inequality $\liminf_{t\rightarrow+\infty} \oX_{t} >\xi_{\Dt}$.
\end{proof}

From Propositions \ref{prop:asympt_sol} and \ref{prop:asympt_disc} we have $\xi^{\Dt}\leq \xi^*\leq \xi_{\Dt}$. Remarkably, as $\Delta t$ tends to 0, the \expSch process attains the same stable state than the solution $(X_t; t\geq 0)$ of the SDE \eqref{eq:IntroSDE} since 
\[|\varphi^{\Dt}(\xi^*)| = \frac{B_1b(0)}{\xi^*}\Dt, \quad \mbox{and} \quad |\varphi_{\Dt}(\xi^*)| = b(0)(\diffc^2+B_2)(\xi^*)^{2\alpha-3}\Dt.\]
Moreover, for $b(0) = 0$ we have $\xi^{\Dt}= \xi^*= \xi_{\Dt}$.  When $b(0)>0$, is it possible to control the asymptotic bias between the scheme and the true stationary state  $\xi^*$  having 
\[|\varphi^{\Dt}(\xi^*)| \vee |\varphi_{\Dt}(\xi^*)|  \leq \epsilon_0,\]
with a certain threshold $\epsilon_0>0$, we can choose $\Delta t$ such that
$\Dt\leq \epsilon_0 \ {\xi^*}  [ b(0) \ \{ (\diffc^2+B_2) (\xi^*)^{2(\alpha-1)} \vee B_1\}]^{-1}$. 

\section{Numerical experiments}\label{sec:num}

In this section, we investigate the numerical rate of convergence of the \expSch scheme and the optimality of the parameter conditions, as expounded in Theorem~\ref{thm:strong_rate_non_stopped} and Theorem~\ref{thm:strong_rate_stopped}.   
Most of the cases under consideration delve into the prototypical scenario of decreasing polynomial coefficients, encompassing both continuous and discontinuous cases satisfying assumptions  of Theorem~\ref{thm:strong_rate_non_stopped}. A distinctive case has been incorporated, where the driving term of the drift is a function bounded above by a decreasing polynomial outside a compact set (according to \ref{H:control}) but the derivative of this function is unbounded (contradicting \ref{H:sch_control_b_deriv}) , thereby illustrating the applicability of Theorem~\ref{thm:strong_rate_stopped}. 

The following class of model allows to restrict the set of parameters involved in the condition of convergence to the only parameter $B_2>0$: 
\begin{equation}\label{eq:proto_case}
dX_t = (1 + X_t -B_2\;X_t^{2\alpha-1}) dt+ X_t^2  dW_t, \;X_0=x>0.
\end{equation}
We vary the value of parameter $B_2$ in the range $[1.0, 6.5]$, such that at least, according to Lemma  \ref{lem:Schememoments},  the second moment of the solution is finite ($B_2\geq 1/2$). 
Note that expression \eqref{eq:test_b_control_explosion}	in the Feller test applied in the Appendix, allows to extract the condition $2 B_2 \geq -(2\alpha-1) \Sigma^2$ for $X$ to not explode at infinity in finite time, here $B_2 \geq - 1/2$.  

The  model \eqref{eq:proto_case} serves as a first test bench for the  numerical performance of the \expSch  scheme and its behaviour in terms of the theoretical condition on the parameters, allowing us to obtain a more nuanced understanding of its convergence characteristics. In addition, this model class  enables us to compare  the range of cases encompassed in the strong error convergence with the convergence  obtained in \cite{BoJaMa2021} for the weak error of the \expSch scheme. The theoretical condition \eqref{eq:condition_thm_unstopped_continuous_1} for the strong convergence with order 1/2  in Theorem~\ref{thm:strong_rate_non_stopped} is  expressed as $\mkappa{2p}\geq0$, while the conditions in Corollary 4.1 of \cite{BoJaMa2021} are presented as $\wkappa\geq0$ with, in term of $(\alpha, B_2)$, 
\begin{equation}\label{cond:kappa condition}
\begin{aligned}
\mkappa{2p}& := 2B_2+1 -2p[(4\alpha - 3)\vee(2\alpha)]  \vee \{  4\ (2p-  \tfrac34) \}, \\
\wkappa& := 2B_2-6 \alpha - \alpha^2 \vee [(12\alpha - 19)\vee(8\alpha - 10)\vee \tfrac{5\alpha^2}{2\alpha-1}]. 
\end{aligned}
\end{equation}
To simplify our comparison, $\alpha$ is fixed to 2 in the model \eqref{eq:proto_case}. We also  limit the  experiments to the $L^2(\Omega)$-error, choosing  $p=1$ and  we expect to observe numerically that
\begin{equation}\label{eq:rateconv}
\sqrt{\EE\left[\sup_{0\leq t\leq T}|X_t - \overline{X}_t|^2\right]} \leq C\;\sqrt{\Delta t},
\end{equation}
where now the conditions in \eqref{cond:kappa condition} simplify as $\mkappa{2} = 2B_2-9 \geq 0$ and the  $\wkappa = 2B_2 - {56}/{3}\geq 0$. Note then that $\mkappa{2}$ reduces to the condition \eqref{cond:lemma_cont} that allows to bound the local error term in Lemma \ref{lem:disc_contrib}.

\paragraph{Numerical parameters.  }  Unless expressly stated otherwise, we consider the initial condition $x_0=1$ and the terminal time $T=1$. Regarding the time-step, in the following we consider $\Delta t=2^{-q}$, for $q=1,\ldots,20$. In addition, the expectation in the right-hand side of  \eqref{eq:rateconv} is estimated by a Monte Carlo approximation, involving $N=10^6$ independent trajectories. 
For the numerical computation of the error, the reference values of $(X_{t_n};n=0,\ldots, \lfloor\tfrac{T}{\Delta t}\rfloor)$ are computed based on a refinement of the numerical scheme with time-step $\Delta t_{\text{ref}} =2^{-21}$. 

\paragraph{Numerical convergence of the \expSch scheme for continuous drift SDE \eqref{eq:proto_case}. } We consider the following cases, determined by the value of $B_2$
\begin{center}
\begin{tabular}{llll}
{\bf Case 1}\quad
&  $dX_t = (1+X_t-{13}/2X_t^{3}) dt+ X_t^{2} dW_t$
&  $\mkappa{2}=4$, &$\wkappa=  -{17}/{3}$, \\
{\bf Case 2}\quad %
&  $dX_t = (1+X_t-{11}/2X_t^{3}) dt+ X_t^{2} dW_t$
&  $\mkappa{2}=2$,& $\wkappa=  -{23}/{3}$, \\
{\bf Case 3}\quad %
&  $dX_t = (1+X_t-\;{9}/{2}\;X_t^{3}) dt+ X_t^{2} dW_t$
&  $\mkappa{2}=0$,&  $\wkappa=  -{29}/{3}$, \\
{\bf Case 4}\quad %
&  $dX_t = (1+X_t-\;7/2\;X_t^{3}) dt+ X_t^{2} dW_t$
& $\mkappa{2}=-2$,& $\wkappa=  -{35}/{3}$, \\
{\bf  Case 5}\quad
&  $dX_t = (1+X_t-\;5/2\;X_t^{3}) dt+ X_t^{2} dW_t$
& $\mkappa{2}=-4$,& $\wkappa=  -{41}/{3}$.
\end{tabular}
\end{center}
Notice that the {\bf Cases 4 \& 5} do not satisfy the assumptions $\mkappa{2}\geq0$, meaning that strong convergence has not been proved for these cases. Weak convergence  has not been proved for all cases. 

Case by case, the $L^2(\Omega)$-errors are shown in Table \ref{tab:experiment1}, where it is clearly shown  that the convergence rate is of the order of $1/2$ for all. These examples show that the criteria set out in Theorem \ref{thm:strong_rate_non_stopped} are sufficient conditions to obtain a strong convergence rate of $1/2$, far to be necessary. This behaviour is also illustrated in Figure \ref{im:experiment1}, plotting the obtained error estimates in a log-log scale.  In particular, we can observe (in the left plot) that the sufficient condition for the $L^2(\Omega)$-$\sup$ error to converge with rate 1/2 is significantly improvable.  Similar results (not shown here for brevity) were obtained for the error process stopped at $ \texp$ (see Equation \eqref{eq:texp}), exhibiting differences, with respect to the unstopped error in the variance level only, which increases as the value of $B_2$ decreases. 
\begin{table}[H]
\centering
{\footnotesize{
\begin{tabular}{| l l l l l l l l l l l}
\toprule
\multicolumn{11}{c}{Strong error with $\Delta t={2^{-q}}$, for $q=10,\ldots,20$}\\
\hline
 \tabhead{$q=10$} & \tabhead{$q=11$} & \tabhead{$q=12$} &\tabhead{$q=13$} & \tabhead{$q=14$} & \tabhead{$q=15$} &\tabhead{$q=16$} & \tabhead{$q=17$} & \tabhead{$q=18$}&\tabhead{$q=19$} & \tabhead{$q=20$} \\
\midrule
\multicolumn{11}{c}{\tabhead{{\bf Case 1}: $\kappa_{\texttt{strong}}>0$.}} \\ \hline
7.04e-03&5.02e-03& 3.58e-03& 2.55e-03& 1.81e-03 & 1.28e-03 & 9.02e-04 & 6.29e-04 & 4.30e-04 & 2.82e-04 &1.63e-04\\
\hline
\multicolumn{11}{c}{\tabhead{{\bf Case 2}: $\kappa_{\texttt{strong}}>0$.}} \\ \hline
8.65e-03 & 6.19e-03 & 4.42e-03 & 3.15e-03 & 2.24e-03 & 1.59e-03 & 1.12e-03 & 7.79e-04 & 5.32e-04 & 3.50e-04 & 2.02e-04\\
\hline
\multicolumn{11}{c}{\tabhead{{\bf Case 3}: $\kappa_{\texttt{strong}}=0$.}} \\ \hline
1.14e-02 & 8.17e-03 & 5.84e-03 & 4.17e-03 & 2.97e-03 & 2.10e-03 & 1.48e-03 & 1.03e-03 & 7.07e-04 & 4.64e-04 & 2.68e-04\\
\hline
\multicolumn{11}{c}{\cellcolor{red!25}\tabhead{{\bf Case 4}: $\kappa_{\texttt{strong}}<0$.}}\\
\hline
1.67e-02 & 1.20e-02 & 8.66e-03 & 6.20e-03 & 4.43e-03 & 
3.13e-03 & 2.21e-03 & 1.55e-03 & 1.06e-03 & 6.94e-04 & 4.01e-04\\
\hline
\multicolumn{11}{c}{\cellcolor{red!25}\tabhead{{\bf Case 5}: $\kappa_{\texttt{strong}}<0$.}}\\
\hline
3.03e-02 & 2.21e-02 & 1.60e-02 & 1.15e-02 & 8.22e-03 & 5.85e-03 & 4.14e-03 & 2.91e-03 & 1.99e-03 & 1.31e-03 & 7.54e-04\\
\bottomrule
\end{tabular}
}}
\caption{Observed numerical strong $L^2(\Omega)$ - error for a time interval $[0,1]$ and dynamics $dX_t = (1 + X_t - B_2 X_t^3)dt + X_t^2 dW_t$ with deterministic initial condition $x_0=1$.\label{tab:experiment1}}
\end{table}

\begin{figure}[H]
\centering
{\includegraphics[width=0.49\textwidth, height=0.34\textwidth]{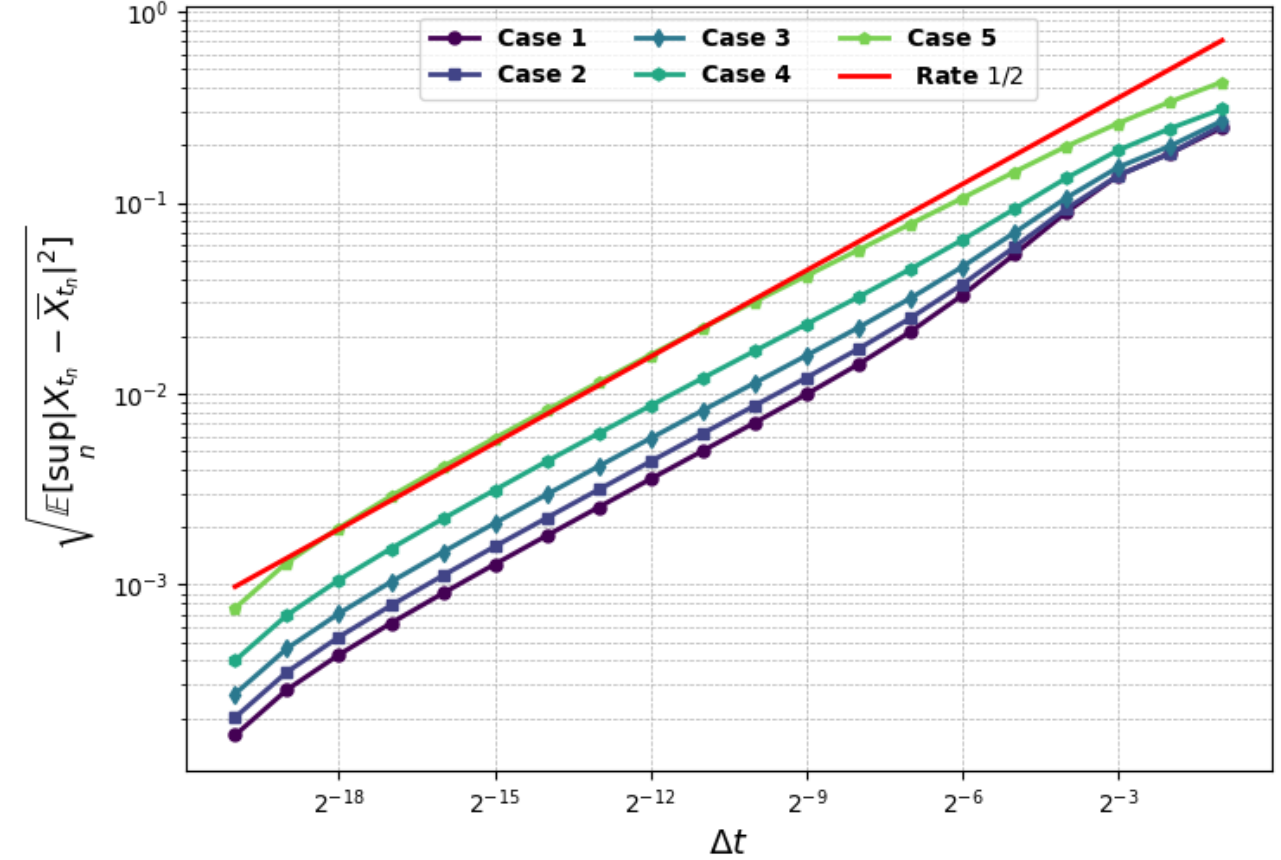}}
{\includegraphics[width=0.49\textwidth,height=0.34\textwidth]{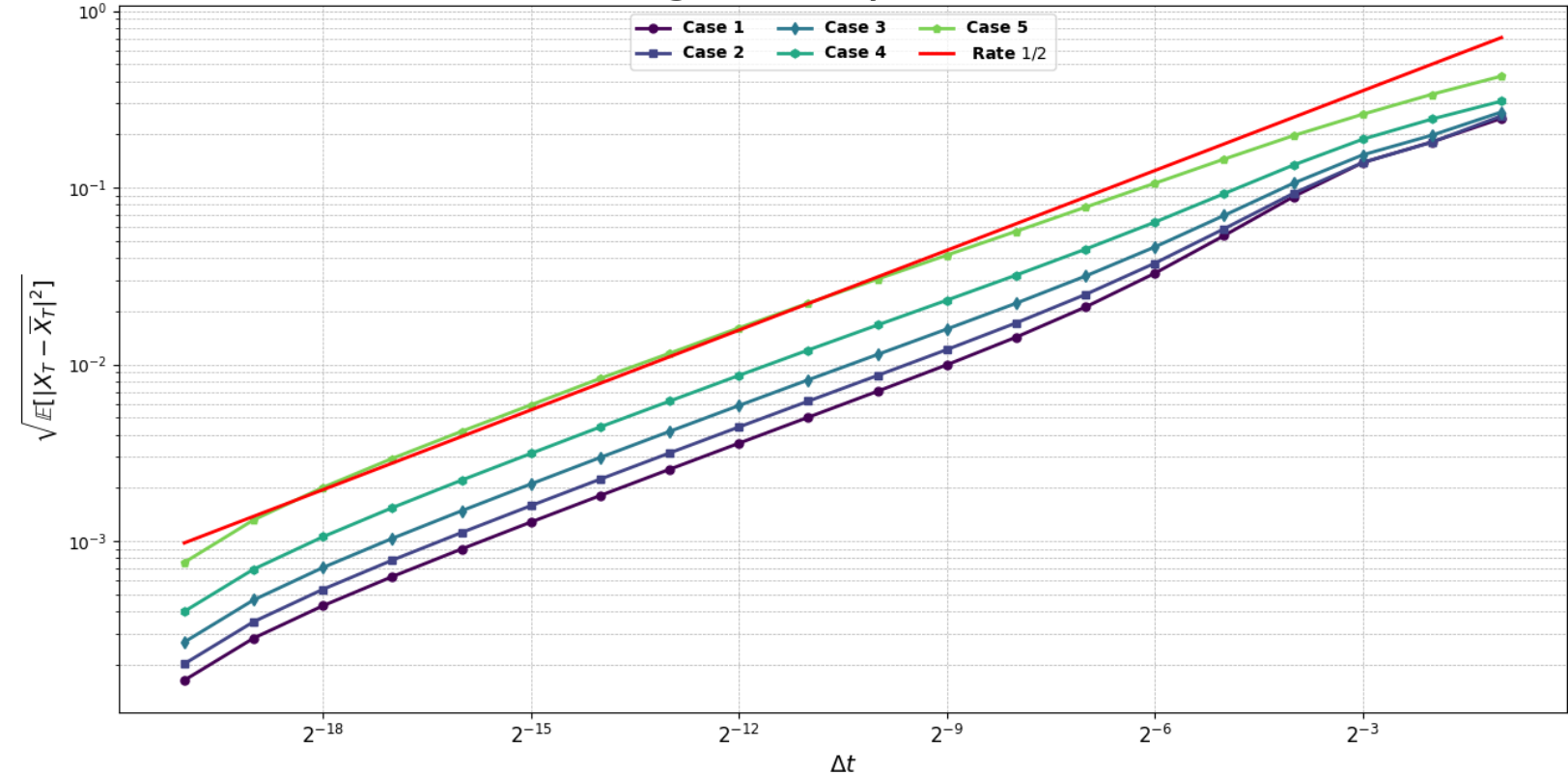}}
\caption{ Strong approximation error for the \expSch  scheme applied to  \eqref{eq:proto_case}, with  {\bf Cases 1  to 5} (in log-log scale), $\EE^{1/2}\left[\sup_{0\leq t\leq T}|X_t - \overline{X}_t|^2\right] $ (left) and $\EE^{1/2}\left[|X_T - \overline{X}_T|^2\right]$ (right). The strong error is compared with the reference slope of order 1/2 (in red).\label{im:experiment1}   }
\end{figure}

\paragraph{Numerical convergence of the \expSch scheme in more complex situations. }  The second set of cases  is composed with:   ({\bf 6})  a  prototypical case \eqref{eq:proto_case} with $B_2=1$, for which $\kappa_{\texttt{strong}}=-7$, allowing the \expSch scheme to have  a moment of order 3 only;   ({\bf 7}) a case with polynomial drift, similar to the previous case, with bounded and differentiable coefficients but with an unbounded derivative; ({\bf 8}) a case with discontinuous  drift, with $B_2$ changing  sign and becoming explosive ($B_2 < -1/2$), but differentiable outside a compact;  and ({\bf 9}) a Lotka-Volterra type of SDE with polynomial drift and explicit solution. In all cases the finiteness of -at least- moments of order 3 are guaranteed.
More precisely, we consider:
\begin{center}
\begin{tabular}{lll}
{\bf  Case 6}\quad &  $dX_t = (1+X_t-X_t^{3}) dt+ X_t^{2} dW_t$,  & with $x_0 = 1$,\\
{\bf Case 7}\quad
&  $dX_t = (1+X_t-(\cos(X_t)+2)^2 X_t^{3}) dt+ X_t^{2} dW_t$, & with $x_0 = 1$,\\
{\bf Case 8}\quad
&  $dX_t = (1+X_t-(6~{\bf 1}_{\{X_t>6\}\cup \{X_t<1.5\}} - 0.6~{\bf 1}_{\{1.5\leq X_t\leq6\}})X_t^{3}) dt+ X_t^{2} dW_t$, & with $x_0 = 3$,\\
{\bf Case 9}\quad
&  $dX_t = X_t(B_1-B_2X_t) dt+ \diffc X_t dW_t$, \qquad with $B_1=\diffc=1$, $B_2 =2$ & ~and $x_0 = 1$.
\end{tabular}
\end{center}
As mentioned, the SDE in {\bf Case 9} has explicit solution given by (see Example 5.1 in \cite{mao2021positivity})
\[X_t = \frac{\exp\{(B_1-\tfrac{\diffc^2}2)t\}+\diffc W_t}{1+B_2\int_0^t\exp\{(B_1-\tfrac{\diffc^2}2)s+\diffc W_s\}ds},\]
where the integral will be approximated using trapezoid rule. It is noteworthy that  {\bf Case 9} resembles a geometric Brownian motion. Thus, a rate of convergence of the \expSch scheme faster than $1/2$ is expected.

Following Theorem \ref{thm:strong_rate_stopped}, we consider also the stopped error at time  $\texp$ introduced in \eqref{eq:texp}: 
$$\sqrt{\EE\left[\sup_{0\leq t\leq T}|X_{t\wedge \texp} - \overline{X}_{t\wedge \texp}|^2\right]}, \quad\text{with}\quad\texp = \inf\{ s\geq 0; \oX_{\eta(s)} > \Dt^{-\frac{1}{ 2(\alpha - 1)}} \}.$$ 
This error is expected to perform better in {\bf Cases 6, 7} and {\bf 8} as the threshold in stopping time is intended to stop the process as soon as the numerical approximation is very high. We limit the comparison  between the stopped and unstopped convergence to the values $\Dt = 2^{-q}$, with $q = 10, 6,\ldots, 20$, since  when $\Dt$ is too large the process reaches the threshold irrelevantly.  Results are shown in Figure \ref{im:stopped_error}. 
\begin{figure}[H]
\centering
{\includegraphics[height=0.33\textwidth,width=0.49\textwidth]{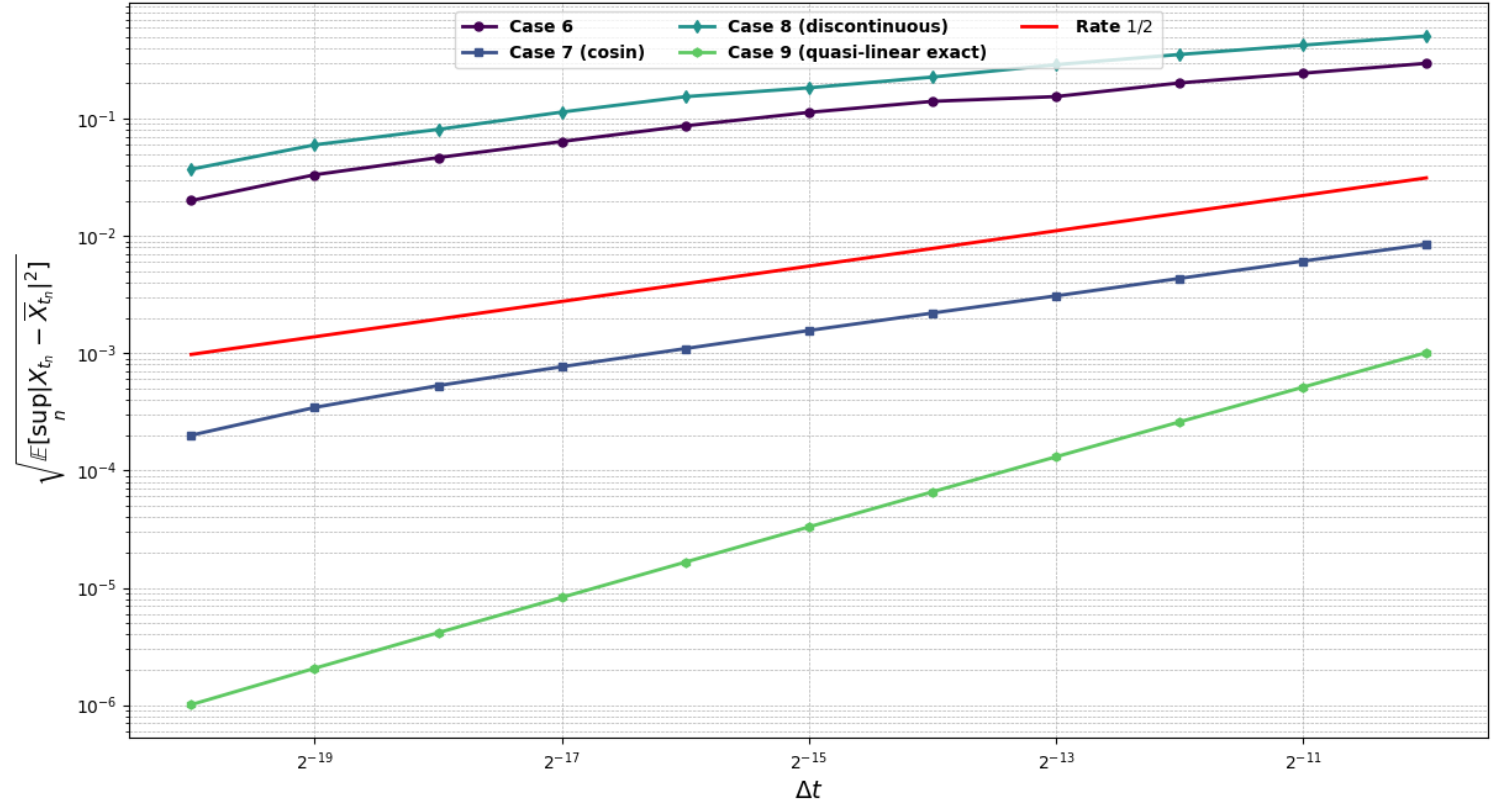}}
{\includegraphics[height=0.33\textwidth,width=0.49\textwidth]{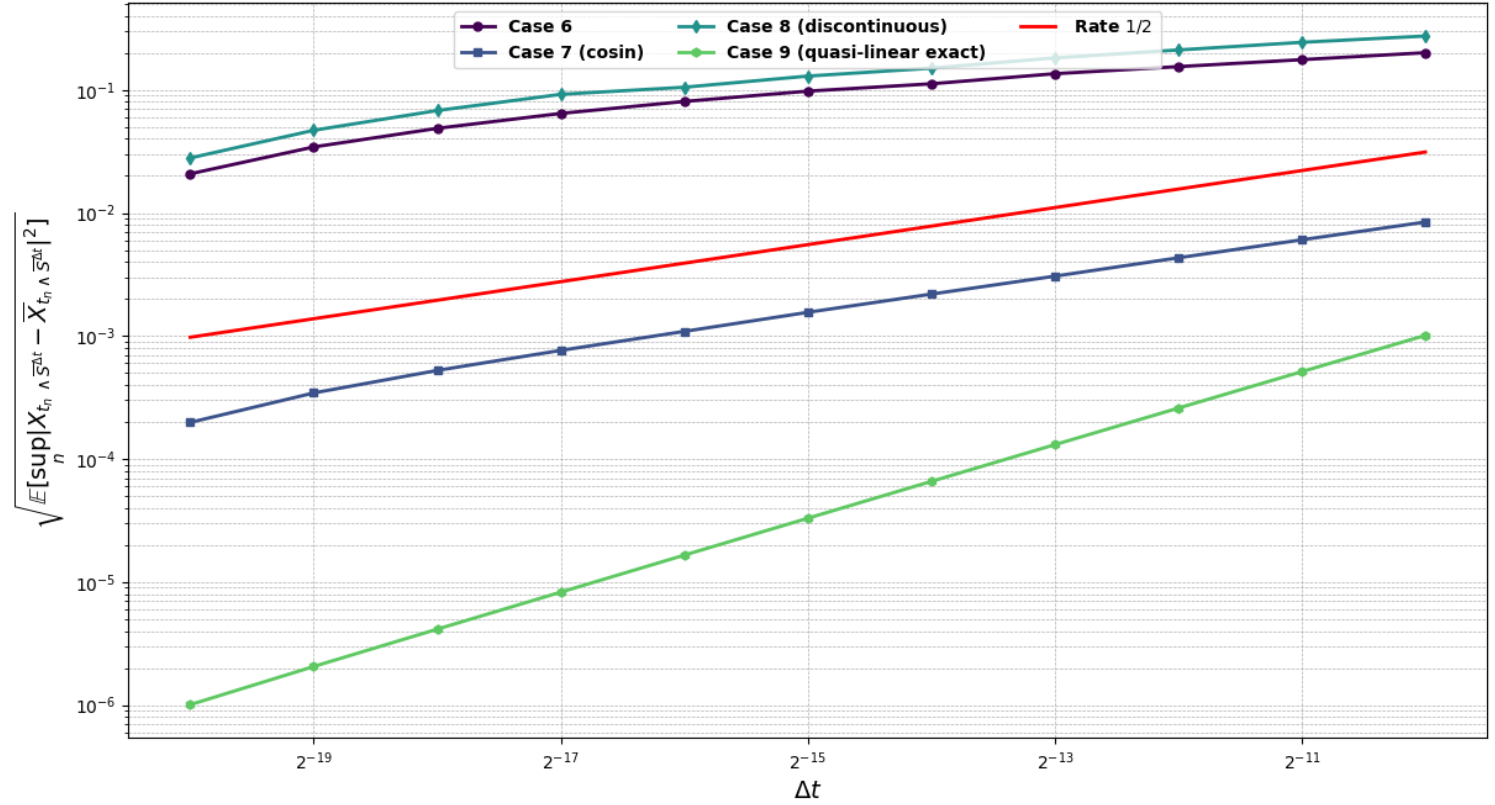}}
\caption{ Strong approximation error (left) and stopped strong approximation error (right) for the \expSch scheme applied to  \eqref{eq:proto_case}, with  {\bf Cases 6  to 9} (in log-log scale), the strong error is compared with the reference slope of order 1/2 (in red).\label{im:stopped_error}}
\end{figure}

Figure \ref{im:stopped_error} shows a convergence of order $1/2$ for the {\bf Cases 6,7}, as stated in Theorem \ref{thm:strong_rate_stopped}, even for the non-stopped error, possibly due to the control given by the cosine function. A convergence of order $1/2$ for the discontinuous {\bf Case 8} is observed, provided that $\Dt$ is small enough, otherwise the convergence can be slowed down somewhat, as stated in Theorem \ref{thm:strong_rate_non_stopped}. Regarding the {\bf Case 9}, Figure \ref{im:stopped_error} exhibits a faster convergence, as expected.

At first glance, the difference between the convergence of the strong and stopped strong errors is not noticeable. a deeper analysis of the results shows that by considering the stopping time $\texp$, the error's variance of the stopped  error stabilizes instead of increasing when $\Dt$ is large and increasing. It also decreases and coincides with the unstopped error when $\Dt$ is small enough and  decreasing, as the role of the stopping time fades. This is illustrated in the most exploding  {\bf Cases 6 \& 8} in Figure \ref{im:variance_error}.
\begin{figure}[ht!]
\centering
{\includegraphics[width=0.55\textwidth,trim={1.6cm 1.0cm 1.6cm 0.4cm},clip]{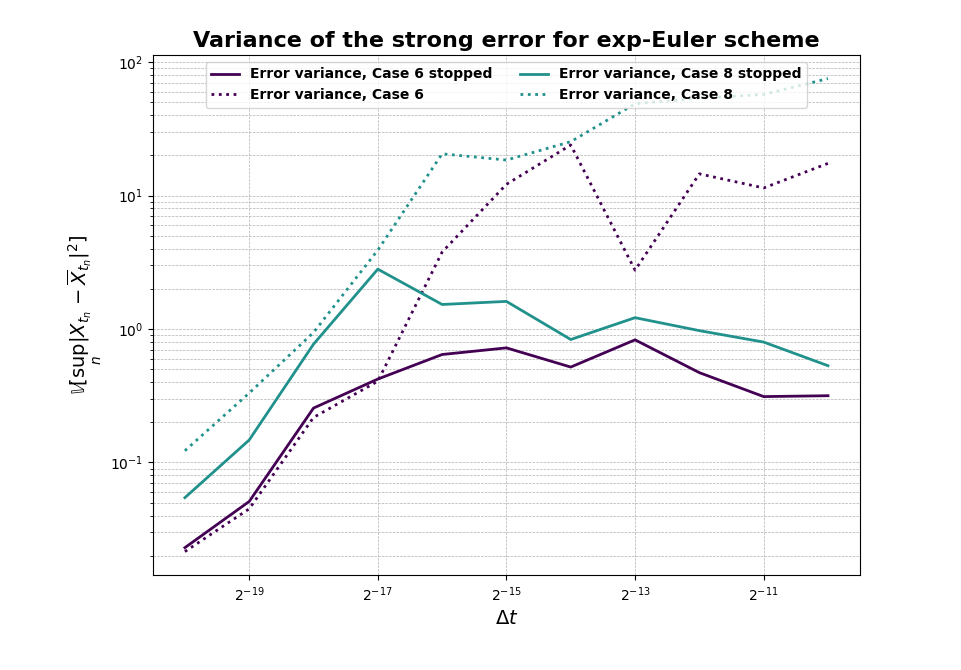}}
\caption{ Variance of the approximation error (stopped and non-stopped) for the \expSch scheme applied in {\bf Cases 6  and 8} (in log-log scale).\label{im:variance_error}}
\end{figure}

\paragraph*{Stability. } 
From the above experiments, we can notice that the \expSch scheme performs very well, even when the sufficient convergence conditions are not satisfied. We now want to illustrate the ability of the numerical scheme to stabilize the approximations when the number of iterations for time integration is large. To this aim, we consider for simplicity the prototypical SDE \eqref{eq:proto_case} with time-step $b(0)=0$ and the rest of parameters as in {\bf Case 1}, meaning: 
\[dX_t = (X_t - 6 X_t^3)dt + X_t^2 dW_t,\quad X_0 = 1.\]

As shown in Proposition \ref{prop:asympt_sol}, the exact process will cross infinitely many times the threshold $\xi$ given by the solution of the equation $\frac{b(0)}\xi +B_1-(B_2+\tfrac{\diffc^2}2)\xi^{2(\alpha-1)} = 0$; in this case, $\xi^2 = \tfrac2{13}$ which coincides with the limit $\Dt=0$ of the scheme thresholds $\xi_{\Delta t}$ and $\xi^{\Delta t}$ given in Proposition \ref{prop:asympt_disc}. A trajectory generated by the \expSch scheme with a time-step of $\Delta t = 0.001$ and terminal time $T=50$ is depicted in Figure \ref{im:stability}. It is clear from the figure that, for large times, the trajectories consistently remain in proximity to $\xi$. The stability exhibited by numerical approximations ensures that any potential accumulation of numerical errors throughout the iterations of the scheme does not disrupt the accurate estimation of statistical quantities.

\begin{figure}[H]
\centering
{\includegraphics[width=0.65\textwidth]{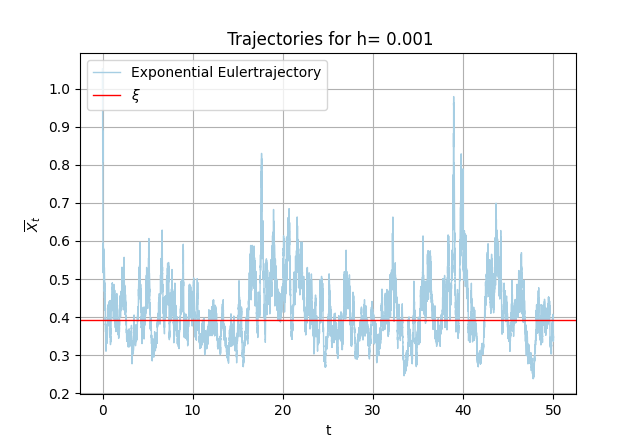}}
\caption{\expSch approximation of the trajectories of the solution to $dX_t = (X_t - 6 X_t^3)dt + X_t^2 dW_t$, $X_0 = 1$, crossing infinite times the threshold $\xi^2=\tfrac{2}{13}$ (horizontal line in red) as shown in propositions \ref{prop:asympt_sol} and \ref{prop:asympt_disc}.\label{im:stability}}
\end{figure}

\bigskip
\paragraph*{Acknowledgement}
The second author acknowledges the support of ANID FONDECYT/POSTDOCTORADO N$^{\circ}$ 321011.


\appendix
\section{Appendix}

\subsection{Proof of Theorem \ref{thm:strong_rate_non_stopped}}\label{sec:proof thm regular}
Throughout the proof, the positive constant $C$ may  change from line to line. It depends on $T$, $p$ and all the parameters in the hypotheses, but not on $\Dt$ and $\epsilon$. To simplify the presentation, we consider in \ref{H:pieceloclip} only one point of discontinuity $\pntDisc_m$, the other cases being just a sum of  contributions similar to that one. From \eqref{eq:error_gamma_proto}, we start the proof with 
\begin{equation*}
d \EE[Y_t^{2p}] \leq \sum_{i=1}^5 \EE[T_i(t)] \; dt. 
\end{equation*}
The contribution $T_1+T_2+T_3$ is treated jointly  while $\EE[|T_4| + |T_5|]$ is controlled using \eqref{eq:estimation_T45}, assuming  \eqref{eq:condition_thm_unstopped_1} or \eqref{eq:condition_thm_unstopped_continuous_1}.

From  \ref{H:control} and \ref{H:diffusion_deriv} we can bound the term $T_1$ as:
\begin{equation*}
T_1(t)
 \leq   {2p} Y_t^{2p} \left\{B_1  +  2(2p-1) \diffc^2 X_t^{2(\alpha-1)} - B_2 X_t^{\alphab-1} \right\}.
\end{equation*}
First, we isolate the local error terms by adding the needed pivots  in $T_2$ and $T_3$, obtaining the following decomposition:  
\begin{align}\label{eq:proof_rate1_aux0}
\begin{aligned}
(T_1 + T_2 + T_3)(t) \leq & {2p} Y_t^{2p} \Big\{B_1  + 2(2p-1) \diffc^2 X_t^{2(\alpha-1)} - B_2 X_t^{\alphab-1}  \\
& \quad\left. + \ 2(2p -1) \frac{\oX_t^{2}}{Y^2_t}\left(\tfrac{\sigma(X_t)}{X_t}-\tfrac{\sigma(\oX_{t})}{\oX_t}\right)^2
+  \frac{\oX_t}{Y_t}   \left(\tfrac{b(X_t)-b(0)}{X_t} - \tfrac{b(\oX_{t})-b(0)}{\oX_{t}}\right)\right\} \\
&  +  2p Y_t^{2p-1}  \oX_t   \left\{ \tfrac{b(\oX_t)-b(0)}{\oX_t} - \tfrac{b(\oX_{\eta(t)})-b(0)}{\oX_{\eta(t)}} \right\} \\ & + 4 p (2p -1) Y_t^{2p-2}   \oX_t^{2} \left(\tfrac{\sigma(\oX_t)}{\oX_t}-\frac{\sigma(\oX_{\eta(t)})}{\oX_{\eta(t)}}\right)^2\\
& \quad := G_0 + G_1 + G_2. 
\end{aligned}
\end{align} 
The only place where the discontinuity on $b$ has to be discussed is in the local error  isolated in $G_1$ and already bounded in 
Lemma \ref{lem:disc_contrib} with 
\begin{equation*}
\begin{aligned}
\EE\left[G_1\right] 
& \leq C \EE[ Y_t^{2p} ] +  C  \Dt^{p (1 - 2\varepsilon)}  + \EE\left[|Y_t|^{2p-1}  \ind_{  \left\{\oX_{[\eta(t), \eta(t) + \Dt]}\in B(\pntDisc_m , 8\diffc  (\tfrac{3}{2}  \pntDisc_m)^{\alpha}\ \Dt^{1/2-\varepsilon} )\right\}}\right],
\end{aligned}
\end{equation*}
under the choice of $\Delta(\varepsilon)$ and the condition \eqref{eq:condition_thm_unstopped_1}, or \eqref{eq:condition_thm_unstopped_continuous_1} when $b$ is continuous.

\paragraph{ We consider first $G_0$ in \eqref{eq:proof_rate1_aux0}. } Rewriting it as 
\begin{align}\label{eq:proof_rate1_aux}
G_0= {2p} Y_t^{2p} \left\{B_1  + 2(2p-1)\diffc^2 X_t^{2(\alpha-1)} - B_2 X_t^{\alphab-1} + \oX_t E_t \right \} 
 \end{align}
where
\[\oX_t E_t = 2(2p -1)  \frac{\oX_t^{2}}{(X_t - \oX_t)^2} \left\{\frac{\sigma(X_t)}{X_t} - \frac{\sigma(\oX_{t})}{\oX_{t}}\right\}^2+ \frac{\oX_t}{X_t - \oX_t}\left\{\frac{b(X_t)-b(0)}{X_t} - \frac{b(\oX_{t})-b(0)}{\oX_{t}}\right\}.\]
We apply  Lemma \ref{lem:estimation_from_Hquatre}: 
\begin{equation*}
\oX_t E_t\leq \oX_t\  \left\{ 4(2p -1) ( \diffc^2 +  (\alpha\diffcd)^2 ) \  (X_t\vee\oX_t)^{2\alpha-3}  - \ \Lg ((\alphab -1) \wedge 1) \ (X_t\vee\oX_t)^{\alphab -2} \right\} + \Cczero, 
\end{equation*}
for $\Cczero$ defined in \eqref{eq:constants_Ccercled}. 
When  $\alphab=2\alpha-1$, by imposing $\Lg  ((\alphab-1) \wedge 1) > 4(2p-1) (\diffc^2 +  (\alpha\diffcd)^2)$ in condition \eqref{eq:condition_thm_unstopped_3} and $B_2 > 2(2p-1) \diffc^2$ in \eqref{eq:condition_thm_unstopped_1},  the first term above  and  also the first term in $G_0$ are  bounded by zero. Then, the bound for  \eqref{eq:proof_rate1_aux} can be written as:
\begin{align}\label{eq:proof_aux1}
G_0\leq  {2p} Y_t^{2p} (B_1  +  \Cczero).
\end{align}
When $\alphab>2\alpha-1$, we can identify a constant to bound such term,  observing that,  for any $a>0$, the map $z\mapsto \psi_a(z):= a \ z^{2(\alpha-1)}- \Lg ((\alphab -1) \wedge 1)  z^{\alphab-1}$ satisfies, for $x_{a}:=\left(\tfrac{a \ 2(\alpha-1)}{\Lg ((\alphab -1) \wedge 1) \ (\alphab-1)}\right)^{1/(\alphab-2\alpha+1)}$, 
\[\mbox{$\psi_a(z)\leq \psi_a(x_{a}) = a^{(\alphab-1)/(\alphab-2\alpha+1)}~\psi_1(x_1) :=\widetilde{\psi}$}.\]
From this, we get a similar bound as in  \eqref{eq:proof_aux1}, with $\Cczero$ replaced by $\Cczero + \widetilde{\psi}$.
\paragraph{We consider $\EE[G_2]$ in \eqref{eq:proof_rate1_aux0}. } 
Taking expectation,   and using Young inequality we have
\begin{equation*}
\begin{aligned}
\EE[G_2] = &  4 p (2p -1) \EE\left[Y_t^{2p-2} \oX_t^{2} \  \left(\tfrac{\sigma(\oX_t)}{\oX_t}-\tfrac{\sigma(\oX_{\eta(t)})}{\oX_{\eta(t)}}\right)^2\right]\\
&\leq  4 (2p -1)  \left( (p-1) \EE[Y_t^{2p}] + 
\EE\left[(\oX_t - \oX_{\eta(t)})^{2p}\left\{\frac{\oX_t}{\oX_t - \oX_{\eta(t)}}\left(\tfrac{\sigma(\oX_t)}{\oX_t}-\tfrac{\sigma(\oX_{\eta(t)})}{\oX_{\eta(t)}}\right)\right\}^{2p}\right]\right).
\end{aligned}
\end{equation*}
Then, from Lemma \ref{lem:estimation_from_Hquatre}-\text{($\sigma$)} and H\"older inequality in the second term  above,  we obtain
\begin{equation}\label{eq:proof_th32_aux1}
\begin{aligned}
\EE[G_2] \leq &  4 (2p -1) \diffc^2 \left((p-1) \EE[Y_t^{2p}] +
2^p(\diffc^2 +(\alpha \diffcd)^2 )^p\ \EE\left[(\oX_t - \oX_{\eta(t)})^{2p}(\oX_t\vee\oX_{\eta(t)})^{2p}\right]\right)\\
&\leq  4 (2p -1) \diffc^2 \left((p-1) \EE[Y_t^{2p}] +
2^p(\diffc^2 +(\alpha \diffcd)^2 )^p \|\oX_t - \oX_{\eta(t)}\|_{L^{2p r_1}(\Omega)}^{2p} \sup_{t\in[0,T]}\EE^{\frac{r_1-1}{r_1}}[\oX_t^{2p\frac{r_1}{r_1-1}}] \right),
\end{aligned}
\end{equation}
for some arbitrary exponent $r_1>1$ to be chosen according to the conditions of Lemmas \ref{lem:local_error_prev} and \ref{lem:Schememoments} to control the $L^{2pr_1}$-norm of the local error and the required moments, imposing 
\[ \ind_{\{2\alpha-1\}}(\alphab)\, p  \ \big(r_1  (2\alpha -1) \vee \frac{r_1}{r_1-1} \big) \leq \frac{1}{2} +\frac{B_2}{\diffc^2}, \]
that can be balanced by the choice $r_1 = 1 + \frac{1}{2\alpha-1}$, becoming $\ind_{\{2\alpha-1\}}(\alphab) \ 2 \alpha  p   \leq \frac{1}{2} +\frac{B_2}{\diffc^2}$, covered by the condition \eqref{eq:condition_thm_unstopped_1}, 
 leading to 
\begin{equation*}
\EE[G_2] \leq C(p) \EE[Y^{2p}_t]  +  C(p)  \Dt^p.
\end{equation*}
 
\paragraph{Putting all the $G_i$ together and adding $\EE[|T_4| + |T_5|]$,  }
for some  positive constant $C$, for $\Dt \leq \Delta(\epsilon)$, 
\begin{equation}\label{eq:bound_g1_discontinuous} 
\EE[Y_t^{2p}]
\leq C \left\{\int_0^t  \EE[Y_s^{2p}] ds +  \Dt^{p-2\epsilon p}  + 
\EE\left[\int_0^t  |Y_s|^{2p-1}  \ind_{\{ \oX_{[\eta(s), \eta(s) + \Dt]} \in B(\pntDisc_m , 8\diffc  (\tfrac{3}{2}  \pntDisc_m)^{\alpha}\ \Dt^{1/2-\varepsilon} )\}}  ds \right] \right\}. 
\end{equation}
When $b$ is continuous, \eqref{eq:bound_g1_discontinuous} reduces to
\begin{equation}\label{eq:bound_g1_continuous}
\EE[Y_t^{2p}]
\leq C \left\{\int_0^t  \EE[Y_s^{2p}] ds +  \Dt^{p} \right\}, 
\end{equation}
with $\Dt\leq 1$ and under the condition \eqref{eq:condition_thm_unstopped_continuous_1}. 
We end the proof applying the change of time technique of the proof of Theorem \ref{thm:strong_rate_stopped}. But, now  the map $t\mapsto {\bf\textrm H}(t)$ in \eqref{eq:def_h_general}, defining the change of time,  is simply reduced to 
\begin{equation*}
{\bf\textrm H}(t) =   C \ \int_0^t  \Dt^{-(\frac{1}{2} -\varepsilon)} \ind_{\{ \oX_{[\eta(s), \eta(s) + \Dt]} \in B(\pntDisc_m , 8\diffc  (\tfrac{3}{2}  \pntDisc_m)^{\alpha}\ \Dt^{1/2-\varepsilon} )\}}  \ ds
\end{equation*}
and $\EE[\exp\{ (p+1) {\bf\textrm H}(T)\} ]$ is bounded according to Lemma \ref{lem:occupation_time}.

\subsection{Technical lemmas and proofs}\label{sec:addproofs} 

Throughout the article we make use of Lenglart's inequality, which we reproduce here for the reader in its sharp version from  \cite{geiss_scheutzow_2023}. 
\begin{lem}{\cite[Corollary II]{Lenglart_77}}\label{lem:Lenglart_sharp}
Let $X$ be a non-negative right-continuous ${\mathcal {F}}_{t}$-adapted process and let G be a non-negative right-continuous non-decreasing predictable process such that ${\displaystyle \mathbb {E} [X(\tau )\mid {\mathcal {F}}_{0}]\leq \mathbb {E} [G(\tau )\mid {\mathcal {F}}_{0}]<\infty }$ for any bounded stopping time $\tau$. Then
\begin{equation*}
{\displaystyle \forall \ r\in (0,1),\mathbb {E} \left[\left(\sup _{t\geq 0}X(t)\right)^{r}{\Big \vert }{\mathcal {F}}_{0}\right]\leq c_{r}\ \mathbb {E} \left[\left(\sup _{t\geq 0}G(t)\right)^{r}{\Big \vert }{\mathcal {F}}_{0}\right],\quad{\text{where }}c_{r}:={\frac {r^{-r}}{1-r}}.}
\end{equation*}
\end{lem}

\begin{lem}\label{lem:equotient_Y_estimation}
When $\zeta\geq \theta>0$, for any $x,y>0$, we have
\begin{equation}\label{eq:quotient_estimation}
\{\lfloor\tfrac\zeta\theta\rfloor\, (x\wedge y)^{\zeta-\theta}\}\vee\{(x\vee y)^{\zeta-\theta}\}\leq \frac{x^{\zeta}-y^{\zeta}}{x^\theta-y^\theta}\leq \tfrac\zeta\theta (x\vee y)^{\zeta-\theta}.
\end{equation}  
When $\zeta>0$ and $\theta>0$, for all $x,y>0$, we have
\begin{equation}\label{eq:quotient_estimation_bis}
\left|\frac{x^{\zeta}-y^{\zeta}}{x^\theta-y^\theta}\right| 
\leq  \left\{ (x\wedge y)^{(\zeta-\theta)} \right\}\vee \left\{\tfrac\zeta\theta (x\vee y) ^{(\zeta-\theta)}\right\}.
\end{equation} 
\end{lem}
\begin{proof}
Let $\xi\geq1$ and $\lfloor\xi \rfloor = \max\{n\in \mathbb{N}:\ n\leq \xi\}$. Then, assuming for example that $x>y$,  we have $(\tfrac{y}{x})^{\xi-\lfloor\xi \rfloor}\leq 1$, and 
\begin{equation}\label{eq:factorization}
x^\xi - y^\xi = x^{\xi-\lfloor\xi \rfloor}\left(x^{\lfloor\xi \rfloor}-y^{\lfloor\xi \rfloor}(\tfrac{y}{x})^{\xi-\lfloor\xi \rfloor}\right)\geq x^{\xi-\lfloor \xi \rfloor}(x-y)\left(x^{\lfloor\xi \rfloor - 1}+x^{\lfloor\xi \rfloor - 2}y+\ldots+y^{\lfloor\xi \rfloor-1}\right).
\end{equation}	
Consider the case with exponents $\zeta\geq\theta>0$, and set $\widetilde{x} = x^\theta$ and $\widetilde{y} = y^\theta$. Then, with for example $x > y\geq 0$, from \eqref{eq:factorization} we obtain
\begin{equation*}
\begin{aligned}
\frac{x^{\zeta}-y^{\zeta}}{x^\theta-y^\theta} = \frac{\widetilde{x}^{\zeta/\theta}-\widetilde{y}^{\zeta/\theta}}{\widetilde{x}-\widetilde{y}} 
& \geq \widetilde{x}^{\tfrac{\zeta}{\theta} - \lfloor \tfrac{\zeta}{\theta}\rfloor}\left(\widetilde{x}^{\lfloor\tfrac\zeta\theta\rfloor -1} + \widetilde{x}^{\lfloor\tfrac\zeta\theta\rfloor - 2}\widetilde{y} + \ldots+ \widetilde{y}^{\lfloor\tfrac\zeta\theta\rfloor - 1} \right)\\
&\qquad  \geq \widetilde{x}^{\tfrac\zeta\theta - 1}\vee \left\{\lfloor\tfrac\zeta\theta\, \rfloor\widetilde{y}^{\tfrac\zeta\theta - 1}\right\} = {x}^{\zeta -\theta}\vee \left\{\lfloor\tfrac\zeta\theta\, \rfloor {y}^{\zeta -\theta }\right\} .
\end{aligned}
\end{equation*}
On the other hand, since $z\mapsto z^{{\zeta}/{\theta}-1}$ is increasing:
\begin{equation*}
\frac{x^{\zeta}-y^{\zeta}}{x^\theta-y^\theta} = \frac1{\widetilde{x}-\widetilde{y}}{\int_{\widetilde{y}}^{\widetilde{x}} \frac\zeta\theta z^{\zeta/\theta - 1} dz} \leq  \frac\zeta\theta \ \widetilde{x}^{\zeta/\theta - 1} =  \frac\zeta\theta \ {x}^{\zeta -\theta}.
\end{equation*}
Combining the two last inequalities leads to \eqref{eq:quotient_estimation}. 

Consider now the exponents $0< \zeta\leq\theta$. Using the previous lower-bound, we also get
\begin{equation*}
	\frac{x^\zeta-y^\zeta}{x^\theta-y^\theta}\leq  \frac1{\{\lfloor\tfrac\theta\zeta\rfloor\, (y\wedge x)^{\theta-\zeta}\}\vee\{(x\vee y)^{\theta-\zeta}\}} \leq   (x\wedge y)^{\zeta-\theta}.
\end{equation*}
\end{proof}
\begin{proof}[Proof of Lemma~ {\ref{lem:estimation_from_Hquatre}}] We start by proving \eqref{eq:estim_lemma_cases} for the drift term. 

${\footnotesize\bullet}$  First, we consider the case $x\wedge y> \pntCompact$, and without loss of generality, we assume $x\geq y$. Then, $b'$ exists in the interval $(y,x)$ and
\begin{equation*}
\tfrac{b(x)-b(0)}{x} - \tfrac{b(y)-b(0)}{y}  = \tfrac{b(x)-b(y)}{x} + (\frac{1}{x} - \frac{1}{y}) (b(y)-b(0))  = \frac{1}{x} \int_y^x b'(z) dz + (\frac{1}{x} - \frac{1}{y}) (b(y)-b(0)).
\end{equation*}
Observe that $(\frac{1}{x} - \frac{1}{y})<0$, and 
\begin{equation*}
- \Lg  (y^{\alphab}\vee y) \leq  b(y)-b(0) \leq  (\Lg y^{\alphab}\vee y).
\end{equation*}
So 
\begin{equation*}
\tfrac{b(x)-b(0)}{x} - \tfrac{b(y)-b(0)}{y} \leq \tfrac{1}{x} \int_y^x \{ {B_1^\prime} - \alphab \Lg z^{\alphab-1} \}dz  - \Lg (\tfrac{1}{x} - \tfrac{1}{y}) (y^{\alphab}\vee y)\leq  ({B_1^\prime} + L_G)  \tfrac{x-y}{x} - {\Lg}(x^{\alphab -1} - y^{\alphab-1} ).
\end{equation*}
Thus,  we get for all $x>y\geq\pntCompact$,
\begin{equation*}
\begin{aligned}
\left(\frac{b(x)-b(0)}{x} - \frac{b(y)-b(0)}{y} \right)\frac{y}{x-y} 
& \leq ({B_1^\prime}+\Lg) \frac{y}{x} - \Lg y \ \frac{x^{\alphab-1}  -  y^{\alphab -1}}{x-y}. 
\end{aligned}
\end{equation*}
The same results is obtained when $y>x$. 
Now, if $\beta>2$, using \eqref{eq:quotient_estimation}, 
\begin{equation*}
- \Lg y \ \frac{x^{\alphab-1}-y^{\alphab -1}}{x-y}  \leq - \Lg y  \  (x\vee y ) ^{\alphab -2}.
\end{equation*}
Otherwise, we still have
\begin{equation*}
- \Lg y \ \frac{x^{\alphab-1}-y^{\alphab-1}}{x-y}\leq - \Lg y (\alphab-1) \  (x\vee y)^{\alphab -2}, 
\end{equation*}
leading to, for all $\alphab >1$,  
\begin{equation*}
\left(\frac{b(x)-b(0)}{x} - \frac{b(y)-b(0)}{y} \right)\frac{y}{x-y} 
 \leq {B_1^\prime}+\Lg - \Lg y  \   ((\alphab-1) \wedge 1)  \ (x\vee y ) ^{\alphab -2}.
\end{equation*}

${\footnotesize\bullet}$  Now we consider the case  $x\vee y >\pntCompact > x\wedge y$. Let us assume that $x >\pntCompact > y$.  Thus we write
\begin{equation*}
\frac{b(x)-b(0)}{x} - \frac{b(y)-b(0)}{y}
 = \frac{b(x)-b(\pntCompact)}{x} + \frac{b(\pntCompact)-b(y)}{x} + (b(y)-b(0))\left(\tfrac1x   - \tfrac1{y}\right). 
\end{equation*}
We next multiply the right-hand side by $\frac{y}{x-y}$ and need to show that the result is bounded.  Before,  on the first term we apply  \ref{H:sch_control_b_deriv},  on the second \ref{H:control}--\eqref{eq:monotony_b},   and on the third  \ref{H:polygrowth},  to obtain
\begin{equation*}
\begin{aligned}
\tfrac{b(x)-b(0)}{x} - \tfrac{b(y)-b(0)}{y}
& \leq {B_1^\prime}\tfrac{x-\pntCompact}x - \Lg(x^{\alphab-1} - \tfrac{\pntCompact^\alphab}x)+ \osLip (\pntCompact^{\alphab-1}+1) \tfrac{(\pntCompact-y)}{x} +\tfrac{x-y}x \  \Lg \ y^{\alphab-1}\vee 1\\
&\leq (\Lg+ {B_1^\prime}\vee [\osLip(\pntCompact^{\alphab-1}+1)])  \tfrac{x-y}x -\Lg(x^{\alphab-1} - y^{\alphab-1}) + \tfrac{\Lg}x\left(\pntCompact^\alphab - y^\alphab\right).
\end{aligned}
\end{equation*}
Notice that, since $x>\pntCompact>y$,  and $\alphab>1$
\[\tfrac{\pntCompact^\alphab - y^\alphab}{x - y} = \tfrac{\pntCompact^\alphab - y^\alphab}{(x-\pntCompact)+(\pntCompact - y)}< \tfrac{\pntCompact^\alphab - y^\alphab}{\pntCompact - y}\leq \alphab\pntCompact^{\alphab-1}.\] 
In this case, we have for all  $x>\pntCompact>y$, 
\begin{equation*}
\left(\tfrac{b(x)-b(0)}{x} - \tfrac{b(y)-b(0)}{y} \right)\tfrac{y}{x-y}
\leq  (\Lg+ {B_1^\prime}\vee [\osLip(\pntCompact^{\alphab-1}+1)]) + {\Lg}\alphab \pntCompact^{\alphab-1} -\Lg y\frac{x^{\alphab-1} - y^{\alphab-1}}{x - y}.
\end{equation*}
In the symmetric  situation  $x < \pntCompact < y$, we just  interchange the roles between $y$ and $x$, writing first
\begin{equation*}
\frac{b(x)-b(0)}{x} - \frac{b(y)-b(0)}{y}
 = \frac{b(x)-b(\pntCompact)}{y} + \frac{b(\pntCompact)-b(y)}{y} + (b(x)-b(0))\left(\tfrac1x   - \tfrac1{y}\right)
\end{equation*}
and following next the same argument. Finally using again \eqref{eq:quotient_estimation}, we get
\begin{equation*}
\left(\tfrac{b(x)-b(0)}{x} - \tfrac{b(y)-b(0)}{y} \right)\tfrac{y}{x-y}
\leq  (\Lg+ {B_1^\prime}\vee [\osLip(\pntCompact^{\alphab-1}+1)]) + {\Lg}\alphab \pntCompact^{\alphab-1} - \Lg y  \   ((\alphab-1) \wedge 1)  \ (x\vee y ) ^{\alphab -2}.
\end{equation*}

${\footnotesize\bullet}$ When $x\vee y <\pntCompact$, from \ref{H:pieceloclip} and \ref{H:polygrowth} we have  (assuming, for instance, $x>y$):
\begin{equation*}
\frac{b(x)-b(0)}{x} - \frac{b(y)-b(0)}{y}  = \frac{(x-y)}x \frac{b(x)-b(y)}{x-y}  - \frac{(x-y)}x \frac{(b(y)-b(0))}y  \leq \frac{(x-y)}x\big(\osLip(\pntCompact^{\alphab-1} +1) +\Lg\pntCompact^{\alphab-1}\vee 1\big),
\end{equation*}
and therefore, 
\begin{equation*}
\left(\frac{b(x)-b(0)}{x} - \frac{b(y)-b(0)}{y} \right)\frac{y}{x-y}
\leq \big(\osLip +\Lg\big) ((\pntCompact^{\alphab-1}\vee 1) +1),
\end{equation*}
or $\big(\osLip +\Lg\big) (\pntCompact^{\alphab-1} +1)$ 
as desired.   We summarize  the   estimation with the three cases as 
\begin{equation*}
\begin{aligned}
& \frac{y}{x-y}\left(\frac{b(x)-b(0)}{x} - \tfrac{b(y)-b(0)}{y} \right) \\
& \quad \leq \left(
{{B_1^\prime}} + \Lg - \Lg y  \   ((\alphab-1 )\wedge 1)  \ (x\vee y ) ^{\alphab -2} \right)\  \ind_{\{x\wedge y \geq \pntCompact\}}\\
& \qquad  + \left( {B_1^\prime}\vee [\osLip(\pntCompact^{\alphab-1}+1)] + \Lg (\alphab \pntCompact^{\alphab-1}+1) - \Lg y  \   ((\alphab-1) \wedge 1)  \ (x\vee y ) ^{\alphab -2}\right) \ \ind_{\{ x\wedge y < \pntCompact < x\vee y\}} \\
& \qquad  + 
\big(\osLip +\Lg\big) ((\pntCompact^{\alphab-1}\vee 1)+1) \  \ind_{\{x\vee y \leq \pntCompact\}} 
\end{aligned}
\end{equation*}
from which we deduce the constant $\Cczero$.
\medskip

We now prove \eqref{eq:estim_lemma_cases_sigma} for the diffusion term.  
Assuming $x> y$, with the help of \ref{H:diffusion_deriv}, 
\begin{equation*}
\begin{aligned}
\frac{y}{(x-y)^2}\left(\frac{\sigma(x)}x - \frac{\sigma(y)}y\right)^2 
& = y \left( \frac{1}{x-y} ( \frac{1}{x} -  \frac{1}{y} )  \sigma(y)   +  \frac{1}{x(x-y)}  \int_y^x \sigma'(z)dz \right)^2 \\
&\quad \leq 2  y \left(\frac{1}{x^2 y^2 } \diffc^2 y^{2\alpha}    +  \frac{1}{x^2} \diffcd^2 \alpha^2 x^{2\alpha -2}  \right) \leq  2 ( \diffc^2 +  (\alpha\diffcd)^2 ) (x\vee y)^{2\alpha -3}, 
\end{aligned}
\end{equation*}
whatever the sign of $2\alpha -3$. 
The case $x < y$ is similar, changing the pivot in the above equality.
\end{proof}

\subsection{Proof of Proposition \ref{prop:XMoments}}\label{sec:appendix_wellposedness} 
  
Under locally Lipschitz assumption for the drift and diffusion coefficients $b$ and $\sigma$ in the whole domain $\RR^+$, existence and pathwise uniqueness holds for the solution of \eqref{eq:IntroSDE}  (see e.g. Theorem 3.1 in  
\cite{IkedaWatanabe} and Definition 5.1 in 
\cite{KarShr-88}) up to an explosion time. 
Without loss of generality, we  simplify the presentation of  the proof under the assumption \ref{H:pieceloclip} considering only one point of discontinuity for $b$, namely at point $\pntDisc$, the extension to several points being straightforward.
  
Following the proof of Theorem 2.6 in \cite{leobacher_2017}, we define the $C^2$ bump function $\phi$ as
\begin{equation*}
\phi(u) = \left\{ \begin{array}{lr}
(1 + u)^3(1 - u)^3, &\text{ if } |u| \leq 1,\\
0, & \text{otherwise,}
\end{array}
\right.
\end{equation*}
which satisfies that $\phi(0) = 1$, $\phi'(0) = 0$, $\phi''(0) = -6$, and for all $|u| \geq 1$,
$\phi(u) = \phi'(u) = \phi''(u) = 0$. Define now the transform $G_{\theta,c} : \RR^+ \rightarrow \RR^+$ by
\begin{equation*}
G_{\theta,c}(x) = x +  \theta c^2  + \theta\phi\left( \frac{x - \pntDisc}{c}\right)
(x - \pntDisc)|x - \pntDisc|, \quad x \geq 0,  
\end{equation*}
where $\theta \neq 0$ and $0 < c < \frac{1}{6 |\theta|}$ are some constants. 
Then, $G_{\theta,c}(x)$ is strictly  positive for all $x>0$, and $G_{\theta,c}'(x) > 0$ for all $x\geq 0$ and is bounded. Therefore, $G_{\theta,c}$ has a
global inverse $G_{\theta,c}^{-1}$  (see \cite{leobacher_2017}, Lemma 2.2). 
Abbreviating $\bar\phi(x) := \phi(\frac{x - \pntDisc}{c})(x -\pntDisc)|x-\pntDisc|$, we compute the transformation $Z = G_{\theta,c}({X}).$ Then, from It\^o formula:
\begin{equation}\label{eq:SDE_transform}
dZ_t = \widehat{b} (Z_t)dt  + \widehat{\sigma}(Z_t) dW_t, 
\end{equation}
where
\begin{equation*}
\begin{aligned}
\widehat{b}(z) &= b(G_{\theta,c}^{-1}(z)) + \theta \ \bar\phi'(G_{\theta,c}^{-1}(z)) \  b(G_{\theta,c}^{-1}(z))  + \tfrac{\theta}{2}  \ \bar\phi''(G_{\theta,c}^{-1}(z)) \  \sigma^2(G_{\theta,c}^{-1}(z)),\\
\widehat{\sigma}(z) & = \sigma(G_{\theta,c}^{-1}(z)) + \theta \  \bar\phi'(G_{\theta,c}^{-1}(z)) \  \sigma(G_{\theta,c}^{-1}(z)). 
\end{aligned}
\end{equation*}
It is clear that $\widehat{\sigma}$ is continuous. Now, observing $\lim_{x \rightarrow \chi}   \bar\phi'(x) = 0$ and $\lim_{x \rightarrow \chi^+}   \bar\phi''(x) = 2 = -\lim_{x \rightarrow \chi^-}   \bar\phi''(x)$,  a short computation shows that with the choice
\begin{equation*}
\theta  =  \frac{ b(\pntDisc^-) - b(\pntDisc^+) }{ 2 \sigma^2(\pntDisc) }, 
\end{equation*}
we get  $\widehat{b}(G_{\theta,c}(\pntDisc^+)) = \widehat{b}(G_{\theta,c}(\pntDisc^-))$. 
With the continuity of the coefficients $\widehat{b}$, $\widehat{\sigma}$,  and their locally Lipschitz properties conserved by the composition with the locally Lipschitz functions $G_{\theta,c}^{-1}$ and $\bar\phi''$, we can claim that {\it{(i)}} pathwise uniqueness holds for the solution of \eqref{eq:SDE_transform}, and then for the solution of \eqref{eq:IntroSDE}, {\it{(ii)}} a weak solution $Z$ exists, up to an explosion time. 
So a weak solution  ${X}$  to \eqref{eq:IntroSDE} exits  up to an explosion time. 

For simplicity in the coming computations of Feller test,  we consider now the transformation $\widetilde{X}_t = X_t^{-2(\alpha-1)}$ satisfying the one-dimensional SDE
\begin{equation}\label{SDEAux1}
d\widetilde{X}_t=\widetilde{b}(\widetilde{X}_t)dt +\widetilde\sigma(\widetilde{X}_t)~d{W}_t,~~\widetilde{X}_0=x^{-2(\alpha-1)},
\end{equation}
where  the drift function $\widetilde{b}$ and diffusion function $\widetilde\sigma$ are   defined as 
\begin{equation*}
\begin{aligned}
\widetilde{b}(x) &=(\alpha-1)\left((2\alpha-1)x^{\tfrac{2\alpha}{2(\alpha-1)}} \ \sigma^2\big(x^{-\tfrac1{2(\alpha-1)}}\big) -
2{{x}^{\frac{2\alpha-1}{2(\alpha -1)}}\ b({x}^{- \frac{1}{2(\alpha-1)}})}\right)\\
\widetilde{\sigma}(x) &= -2(\alpha-1)x^{\tfrac{2\alpha-1}{2(\alpha-1)}} \  \sigma\Big(x^{-\tfrac1{2(\alpha-1)}}\Big).
\end{aligned} 
\end{equation*}
Defining the explosion time as
\begin{equation*}
S = \lim_{n\rightarrow+\infty} S_n,\quad\text{with } \  S_{n}:=\inf\big\{0\leq t: \widetilde{X}_t \notin (\tfrac{1}{n}, n)\big\},
\end{equation*}
we use a Feller test to show that $\PP(S=+\infty) = 1$, which implies that the explosion time for $X$ is also $+\infty$.

\paragraph{Feller test for non explosion. } 
We add now the hypotheses \ref{H:polygrowth} and \ref{H:control} in the discussion.  This part of the proof follows the computation done in \cite{BoJaMa2021}[Supplementary, also available on arxiv]. For the seek of completeness, and as we extend in this paper the considered drift and diffusion class,  we reproduce the main steps, but remove some details of the computation that can be found in the previous reference. 

Showing that $\PP(S=+\infty) = 1$ is equivalent to show that $v(0^+):=\lim_{x\rightarrow 0^+}v(x) = + \infty$ and $ v(+\infty):=\lim_{x\rightarrow +\infty}v(x) = +\infty$, where $x\mapsto v(x)$ is defined from the scale function $x \mapsto p(x)$ as  
\begin{equation*}
v(x) = \int_c^x p'(y)\big(\int_c^y\tfrac{2~dz}{p'(z)~\widetilde\sigma^2(z)}\big)dy,\quad\quad p(x)=\int_c^x\exp\{-\int_c^z \tfrac{2{\widetilde b}(y)}{\widetilde\sigma^2(y)} dy\} \ dz,  
\end{equation*}
with a fixed $c>1$  (e.g. \cite[Theorem 5.29]{KarShr-88}). From \ref{H:control}, we get the lower bounds for  all $x\in (0,c)$,
\begin{equation}\label{eq:test_b_control_explosion}
\begin{aligned}
\frac{2\widetilde{b}(x)}{\widetilde\sigma^2(x)}  = \frac{2\alpha-1 }{2(\alpha-1)x}-x^{-\tfrac{2\alpha-1}{2(\alpha-1)}}\frac{b\Big(x^{-\tfrac1{2(\alpha-1)}}\Big)}{(\alpha-1)~\sigma^2\Big(x^{-\tfrac1{2(\alpha-1)}}\Big)}\\
\geq \frac{2\alpha-1 }{2(\alpha-1)x}-\frac{b(0)x^{\tfrac{1}{2(\alpha-1)}} + B_1}{(\alpha-1)~x^{\tfrac{2\alpha}{2(\alpha-1)}}~\sigma^2\Big(x^{-\tfrac1{2(\alpha-1)}}\Big)}, 
\end{aligned}
\end{equation}
from which we derive the following estimates for $p(0^+)$:
\begin{equation*}
\begin{aligned}
p(0^+)= -\int_0^c\exp\Big\{\int_z ^c\tfrac{2 \widetilde{b}(y)}{\widetilde\sigma^2(y)} dy\Big\}\ dz \leq -  \int_0^c\exp\Big\{\tfrac{(2\alpha-1)}{2(\alpha-1) } \ \log(\tfrac{c}{z})- \tfrac{b(0)c^{\tfrac{1}{2(\alpha-1)}} + B_1}{\alpha-1}\int_z^c \varphi(y)dy \Big\}\ dz,
\end{aligned}
\end{equation*} 
with $\varphi(y)$ defined as  $\varphi(y) := y^{-\tfrac{2\alpha}{2(\alpha-1)}}\sigma^{-2}\big(y^{-\tfrac1{2(\alpha-1)}}\big) \geq \Sigma^{-2} $.  From Hypothesis \ref{H:polygrowth}, the map $y\mapsto \varphi(y)$ is continuous and bounded, and thus $\int_0^c \varphi(y)dy$ is finite. Furthermore, since $\int_0^c z^{-\tfrac{2\alpha-1}{2(\alpha-1)}}dz$ goes to $+\infty$, we obtain $p(0^+) = -\infty$, which implies that $v(0^+)=+\infty$ (see e.g \cite{KarShr-88}, Problem 5.27).

Rewriting explicitly the function $v$ as:
\begin{equation*}
v(x) = \int_c^x \int_c^y \frac{2}{\widetilde\sigma^2(z)} \ \exp\{-\int_z^y \tfrac{2\widetilde{b}(u)}{\widetilde\sigma^2(u)} du\}  \  dz \ dy, 
\end{equation*}
we check  now that $v(+\infty)=+\infty$. From  \ref{H:polygrowth},  we have, for all  $x>0$,
\begin{equation*}
\frac{2\widetilde{b}(x)}{\widetilde\sigma^2(x)}  \leq 
\tfrac{2\alpha-1 }{2(\alpha-1)} \frac1x  +\tfrac{L_G}{\alpha-1}x^{-\tfrac{2\alpha-1}{2(\alpha-1)}}\frac{ (x^{-\tfrac1{2(\alpha-1)}}\vee x^{-\tfrac{\alphab}{2(\alpha-1)}})}{\sigma^2\Big(x^{-\tfrac1{2(\alpha-1)}}\Big)}
=  \tfrac{2\alpha-1 }{2(\alpha-1)} \frac1x  +\tfrac{L_G}{\alpha-1}\varphi(x)~x^{\tfrac{1}{2(\alpha-1)}} (x^{-\tfrac1{2(\alpha-1)}}\vee x^{-\tfrac{\alphab}{2(\alpha-1)}}), 
\end{equation*}
and for $x>c>1$, 
\begin{equation*}
\frac{2\widetilde{b}(x)}{\widetilde\sigma^2(x)}  \leq 
\frac{2\alpha-1 }{2(\alpha-1)x}+\tfrac{L_G}{\alpha-1}\varphi(x). 
\end{equation*}
\begin{equation*}
\begin{aligned}
v(+\infty)&=  \int_c^{+\infty} \big(\int_c^y  \frac{2}{\widetilde{\sigma}^2(z)} \ \exp\Big\{-\int_z^y \tfrac{2\widetilde{b}(u)}{\widetilde\sigma^2(u)} du\Big\} dz\big)dy \\
&\quad \geq \int_c^{+\infty}\int_c^y  \frac{2}{\widetilde{\sigma}^2(z)} \exp\Big\{  \tfrac{(2\alpha-1)}{2(\alpha-1) } \ \log(\tfrac{z}{y}) - \tfrac{2L_G}{\alpha-1}\int_z^y \varphi(u)du\Big\} \ dz dy \\
&\qquad = \frac{1}{2(\alpha -1)^2} \int_c^{+\infty} y^{-\tfrac{(2\alpha-1)}{2(\alpha-1)}}
\int_c^y   z^{\tfrac{1}{2(\alpha-1)}}  \varphi(z)  \exp\Big\{-\tfrac{2L_G}{\alpha-1}\int_z^y \varphi(u)du\Big\} \ dz \  dy. 
\end{aligned}
\end{equation*}
With the help of an integration by part,
\begin{equation*}
\begin{aligned}
& \int_c^y z^{\tfrac{(2\alpha-1)}{2(\alpha-1)}-1}  \varphi(z)  \exp\Big\{-\tfrac{2L_G}{\alpha-1}\int_z^y \varphi(u)du\Big\} \ dz \\
&=\tfrac{\alpha-1}{2 L_G}  \left(y^{\tfrac{(2\alpha-1)}{2(\alpha-1)}-1} -  c^{\tfrac{(2\alpha-1)}{2(\alpha-1)}-1} \exp\Big\{-\tfrac{2L_G}{\alpha-1}\int_c^y \varphi(u)du\Big\}
- \tfrac{1}{2(\alpha-1)}  \int_c^yz^{\tfrac{(2\alpha-1)}{2(\alpha-1)}-2}  \exp\Big\{-\tfrac{2L_G}{\alpha-1}\int_z^y \varphi(u)du\Big\} dz \right). 
\end{aligned}
\end{equation*}
Thus (by forgetting the multiplicative constant), $v(+\infty) \gtrsim  \int_c^{+\infty}\frac{dy}y  - I_1 - I_2$, with
\begin{equation*}
I_1:= \int_c^{+\infty} y^{-\tfrac{(2\alpha-1)}{2(\alpha-1)}} \exp\Big\{-\tfrac{2L_G}{\alpha-1}\int_c^y \varphi(u)du\Big\}dy,\quad I_2:= \int_c^{+\infty} \int_c^y (y/z)^{-\tfrac{(2\alpha-1)}{2(\alpha-1)}} z^{-2}  \exp\Big\{-\tfrac{2L_G}{\alpha-1}\int_z^y \varphi(u)du\Big\} dz dy. 
\end{equation*}
We claim that $I_1$ and $I_2$ are convergent. Indeed, since $\varphi$ is positive and continuous in $(c,y)$, for all $y>c$, then
\[I_1 \leq \int_c^{+\infty} y^{-\tfrac{(2\alpha-1)}{2(\alpha-1)}}dy<+\infty.\]
Regarding $I_2$, notice that $(y/z)^{-\tfrac{(2\alpha-1)}{2(\alpha-1)}}\leq 1$ for all $z\leq y$, and $\varphi(u)\geq \Sigma^{-2}>0$ for all $u\in\RR^+$. Then,
\begin{equation*}
\begin{aligned}
&I_2 \leq \int_c^{+\infty} \int_z^{+\infty} z^{-2}  \exp\Big\{-\tfrac{2L_G}{\alpha-1}\int_z^y \varphi(u)du\Big\} dy dz \\
& \qquad \leq  \int_c^{+\infty} \int_z^{+\infty} z^{-2}  \exp\Big\{-\tfrac{2L_G}{\Sigma^2(\alpha-1)}(y-z)\Big\} dy dz \  \lesssim \int_c^{+\infty} z^{-2} dz <+\infty.
\end{aligned}
\end{equation*}

Coming back to the estimation of $v(+\infty)$ we can conclude that $v(+\infty)=+\infty$ and thus there exists a unique strictly positive strong solution to \eqref{SDEAux1} for all $t\in[0,T]$.  Using reversely the Lamperti transformation, this immediately implies  that  $X_t={\widetilde{X}_t}^{-\frac1{2(1-\alpha)}}$ satisfies the SDE \eqref{eq:IntroSDE} on $(0,T]$ and the pathwise uniqueness of strictly positive solution is also granted. 

\paragraph{Positive moment bounds. } 
Applying It\^{o}'s formula to the stopped process $({X}_{t\wedge\tau_M};0\leq t\leq T)$ with $x_0 \in(\tfrac{1}{M},M)$ and  $\tau_M=\inf\{t\in[0,T]: {X}_t \notin (\tfrac{1}{M},M)\}$, then following the proof of Lemma \ref{lem:Schememoments}, there exists a constant $C$ such that, for all $M>0$, for all $p>0$ satisfying $p\ind_{2\alpha-1}(\alphab)\leq \tfrac12 + \tfrac{B_2}{\diffc^2}$, we have
\begin{equation*}
\EE[X_{t\wedge \tau_M}^{2p}]\leq x_0^{2p}+C\int_0^t\EE[X_{s\wedge \tau_M}^{2p}]ds.
\end{equation*}
From Gronwall's inequality, we conclude on the $2p$\,th-moment control of $X$. 

\paragraph{Negative moment bounds. } 
For $q>0$, applying It\^{o}'s formula to  $({X}_{t}^{-q}, 0\leq t\leq T)$ (omitting stopping time argument for simplicity),  and using \ref{H:polygrowth} we get
\begin{equation}\label{eq:Ito_negative_moments}
\begin{aligned}
\EE[X_{t}^{-q}] \leq & x_0^{-q} - b(0)q\int_0^t\EE[X_s^{-(q+1)}]ds + q \Lg\int_0^t \EE[X_s^{-q + \alphab-1}\vee X_s^{-q}]ds +q(q+1) \tfrac{\diffc^2}2\int_0^t \EE[X_s^{-q + 2(\alpha-1)}]ds.
\end{aligned}
\end{equation}
All the terms on the right-hand side can be easily bounded in terms of $\int_0^t\EE[X^{-q}_s]ds$. By applying the Gronwall inequality, we can establish the finiteness of all the negative moments. However, this rough estimation has a dependency on the exponent $q$, growing as $\exp(q^2)$.

To address this issue, we show now  that the dependency on the exponent $q$ does not grow faster than $\exp(q \log(q))$. This is crucial for controlling the exponential moments later on. We balance the dependence on $q$ (where only large values of $q$ matter) as follows: for all powers $q> \theta$, we can use Young's inequality to obtain $y^{q-\theta}\leq \frac{\theta}q +\, y^q$.
 Applying this  inequality two times  with $y = X^{-1}$ and $\theta \in\{\alphab -1, 2(\alpha-1)\}$,
in \eqref{eq:Ito_negative_moments}, we obtain
\begin{equation*}
\EE[X_{t}^{-q}]\leq  x_0^{-q} + a_1\ t  + a_2 \ q \int_0^t \EE[X_s^{-q}]ds  - b(0)q\int_0^t\EE[X_s^{-(q+1)}]ds + q^2 \tfrac{\diffc^2}2\int_0^t \EE[X_s^{-q + 2(\alpha-1)}]ds
\end{equation*}
with $a_1=\Lg(\alphab-1)+(\alpha-1)\diffc^2$ and $a_2 = 2\Lg+\tfrac{\diffc^2}2$. 

\paragraph*{When $b(0)\geq \frac{\diffc^2}{2}$, } for all powers $q> \theta$, by
Young's inequality again, we have 
$ q \ y^{q+1-\theta}\leq q^{\tfrac{q+1}{\theta}} +y^{q+1}$. 
Applying this  inequality with  $\theta=2\alpha-1$,  we obtain
\begin{equation*}
\EE[X_{t}^{-q}]\leq  x_0^{-q} + a_1\,t  +q a_2 \int_0^t \EE[X_s^{-q}]ds + \tfrac{\diffc^2}2q^{\tfrac{q+1}{2\alpha-1} + 1} t +(\tfrac{\diffc^2}2- b(0))q\int_0^t\EE[X_s^{-(q+1)}]ds. 
\end{equation*}
From Gronwall inequality we get the estimation of the negative moments
\begin{equation}\label{eq:NegativesM estimation}
\sup_{t\in[0,T]}\EE[{X}_{t}^{-q}]\leq \Big(x_0^{-q} +a_1 T+\tfrac{\diffc^2}2T~q^{\frac{q+2\alpha}{2\alpha-1}}\Big)\exp\{a_2T~q\}, \quad\mbox{for all }q>\alphab-1.
\end{equation}
\paragraph*{When $0\leq b(0)<\diffc^2/2$,} we use the 
Young's inequality $ q \ y^{q-\theta}\leq q^{\tfrac{q}{\theta}} +y^{q}$, with   $\theta = 2(\alpha-1)$ to get 
\begin{equation*}
\EE[X_{t}^{-q}]\leq  x_0^{-q} + a_1\,t  +q a_2' \int_0^t \EE[X_s^{-q}]ds  + \tfrac{\sigma^2}2q^{\tfrac{q}{2(\alpha-1)} + 1} t, 
\end{equation*}
with $a_2' = a_2 + \tfrac{\diffc^2}{2}$. From Gronwall inequality we get the estimation of the negative moments
\begin{equation}\label{eq:NegativesM estimation2}
\sup_{t\in[0,T]}\EE[{X}_{t}^{-q}]\leq \Big(x_0^{-q} +a_1 T+\tfrac{\diffc^2}2T~q^{\frac{q+2(\alpha-1)}{2(\alpha-1)}}\Big)\exp\{
a_2' T~q\}, \quad\mbox{for all }q>\alphab-1.
\end{equation}

\subsubsection*{Exponential moment bound}
By Itô's formula, for all $t\in[0,T]$, 
\begin{equation}\label{eq:Ito_log}
\log\left(\frac{X_t}{x_0}\right) = \left(\int_0^t\frac{b(X_s)}{X_s}ds +\int_0^t \frac{\sigma(X_s)}{X_s}dW_s- \frac{1}{2}\int_0^t\frac{\sigma^2(X_s)}{X_s^2}ds  \right).
\end{equation}
Using  \ref{H:control} to make appear $\int_0^t  X_s^{\alphab -1} ds$  from $\int_0^t\frac{b(X_s)}{X_s}ds$ and multiplying by $\mu>0$, we get
\begin{equation*}
\int_0^t \mu  X_s^{\alphab -1} ds \leq \frac{\mu}{B_2} B_1 t  -    \frac{\mu}{2B_2} \int_0^t \frac{\sigma^2(X_s)}{X_s^2} ds  +  \frac{\mu}{B_2} \int_0^t\frac{\sigma(X_s)}{X_s}dW_s- \frac{\mu}{B_2} \log\left(\frac{X_t}{x_0}\right) + \frac{\mu}{B_2} \int_0^t \frac{b(0)}{X_s} ds.
\end{equation*}
Then, taking the expectation of the $\exp$ of it
\small{
\begin{equation*}
\begin{aligned}
& \EE\Big[\exp\{\mu\int_0^tX_s^{\alphab -1}ds\}\Big]
\leq e^{\frac{\mu B_1 t}{B_2}}  \EE\Big[\big(\tfrac{X_t}{x_0}\big)^{-\tfrac{\mu}{B_2}} 
\exp\Big\{\tfrac{\mu}{B_2} \int_0^t \tfrac{b(0)}{X_s} ds\Big\} 
\exp\Big\{ \tfrac{\mu}{B_2}\Big( \int_0^t\tfrac{\sigma(X_s)}{X_s}dW_s -  \tfrac{1}{2} \int_0^t \tfrac{\sigma^2(X_s)}{X_s^2}ds \Big)\Big\}
\Big]. 
\end{aligned}
\end{equation*}
}
Assuming $\mu < B_2$, and  applying H\"{o}lder's inequality for  $q,p$ such that $1=\frac{1}{q} +\frac{1}{p}+\frac\mu{B_2}$,  we have 
\begin{equation*}
\begin{aligned}
& \EE\Big[\exp\{\mu\int_0^tX_s^{\alphab -1}ds\}\Big]
\leq e^{\frac{\mu B_1 t }{B_2}}  \  \EE^{\frac{1}{q}}\Big[\left(\frac{X_t}{x_0}\right)^{-\frac{\mu}{B_2} q}\Big]  \ 
\EE^{\frac{1}{p}}\Big[\exp\{\frac{\mu}{B_2} p \int_0^t \frac{b(0)}{X_s} ds\}\Big]. 
\end{aligned}
\end{equation*}
\paragraph*{When $b(0)=0$,}we get our bound. If not, expanding the last term in series with parameter $\theta = \frac{\mu b(0) p}{B_2}$, and using Jensen's inequality, 
\begin{equation}\label{eq:expomoments}
\begin{aligned}
\EE^{\frac{1}{p}}\Big[\exp\{\theta \int_0^t \frac{1}{X_s} ds\}\Big]
&\leq C \ \EE^{\frac{1}{p}}\Big[\sum_{k\geq 0}\frac{1}{k!}{\left(\theta\int_0^t\frac{1}{X_s}ds\right)^k}\Big] \leq C \Big(\sum_{k\geq 0}\frac{(\theta~T)^k}{k!}\sup_{t\in[0,T]}\EE[X_t^{-k}]\Big)^{\frac{1}{p}}
\end{aligned}
\end{equation}
where the last inequality above is obtained from the Fubini-Tonelli Theorem. 
It remains to prove the finiteness of the series in the expression above, for which we use the control in $k$ of negative moments in \eqref{eq:NegativesM estimation} and \eqref{eq:NegativesM estimation2}.

\paragraph*{When $0<b(0)<\sigma^2/2$,}we use the estimation \eqref{eq:NegativesM estimation2}
\begin{equation}\label{eq:seriesEstim}
\begin{aligned}
\sum_{k\geq \alphab} \frac{1}{k!}\left(\theta\,T \right)^k\sup_{t\in[0,T]}\EE[X_t^{-k}]&\leq \sum_{k\geq\alphab} \frac{1}{k!}\big(\theta\,T e^{\frac{a_2'}{2}T} \big)^k(x_0^{-k}+a_1 T)+T\tfrac{\diffc^2}2\sum_{k\geq\alphab}\frac{1}{k!}\big(\theta\,T e^{\frac{a_2'}{2}T}\big)^k k^{1+\frac{k}{2(\alpha-1)}}.
\end{aligned}
\end{equation}
The first sum in the right-hand side is clearly finite. For the second, by setting $\zeta(k)=\tfrac{1}{k!}\left(\theta~T \exp\{\frac{a_2'}{2}T\}\right)^k ~k^{\frac{k}{2(\alpha-1)}+1}$, we observe that 
\begin{equation*}
\lim_{k\rightarrow+\infty}\frac{\zeta(k+1)}{\zeta(k)} =\theta\,T e^{\frac{a_2}{2}T}\lim_{k\rightarrow+\infty}\left(\frac{k+1}{k}\right)^{\frac{k}{2(\alpha-1)}+1}~\frac1{(k+1)^{\frac{2\alpha-3}{2(\alpha-1)}}},
\end{equation*}
which converges to zero whenever $\alpha>\frac32$, and thus 
$\sum_{k=0}^{\infty}\zeta(k)$ is finite. 
This ends this part of the proof by substituting \eqref{eq:seriesEstim} bound in \eqref{eq:expomoments}. 

\paragraph*{When $b(0)\geq \tfrac{\diffc^2}2$,}we use instead  \eqref{eq:NegativesM estimation}, obtaining the same relation with updated $\zeta(k)=\tfrac{1}{k!}\left(\theta~T \exp\{\frac{a_2}{2}T\}\right)^k ~k^{\frac{k+2\alpha}{2\alpha-1}}$. In this particular case, we observe that 
\begin{equation*}
\lim_{k\rightarrow+\infty}\frac{\zeta(k+1)}{\zeta(k)}=\theta\,T e^{\frac{a_2}{2}T}\lim_{k\rightarrow+\infty}\left(\frac{k+1}{k}\right)^{\frac{k}{2\alpha-1}+\frac{2\alpha}{2\alpha-1}}~\frac1{(k+1)^{\frac{2\alpha-2}{2\alpha-1}}},
\end{equation*}
which converges to zero for all $\alpha>1$, and thus 
$\sum_{k=0}^{\infty}\zeta(k)$ is also finite in this case.

Next, in order to bound $\sup_{t\in[0,T]}\EE\big[\exp\{\nu \int_0^t \frac{1}{X_s} ds\} \big]$, we just  re-use the computation starting to the right-hand side of \eqref{eq:expomoments} for $\theta =\nu$.

Finally,  we  analyse  $\EE\Big[\exp\{-\upsilon\int_0^t\frac{b(X_s)}{X_s}ds\}\Big]$, for any   $\upsilon>0$. From \eqref{eq:Ito_log}, for any $\epsilon_0>1$.
\begin{equation*}
\begin{aligned}
& \exp\left\{ -\upsilon\int_0^t\frac{b(X_s)}{X_s}ds\right\}= \left(\frac{X_t}{x}\right)^{-\upsilon}\exp\left\{\upsilon\diffc\int_0^tX_s^{\alpha-1}dW_s- \upsilon\frac{\diffc^2}{2}\int_0^tX_s^{2\alpha-2}ds \right\}\\
& \quad = \left(\frac{X_t}{x_0}\right)^{-\upsilon}\exp\left\{\upsilon\diffc\int_0^tX_s^{\alpha-1}dW_s - \epsilon_0\upsilon^2\frac{\diffc^2}2\int_0^tX_s^{2(\alpha-1)}ds\right\} \exp\left\{\upsilon \left(\epsilon_0\upsilon - 1\right)\frac{\diffc^2}2\int_0^tX_s^{2(\alpha-1)}ds \right\}.
\end{aligned}
\end{equation*}
Then, for $\epsilon_{1,2}>1$, such that $\frac{1}{\epsilon_0} +\frac{1}{\epsilon_1}+\frac{1}{\epsilon_2}=1$, we take expectation and apply H\"older inequality, obtaining 
\begin{equation*}
\EE\left[\exp\left\{ -\upsilon\int_0^t\frac{b(X_s)}{X_s}ds\right\}\right]\leq 
\EE^{1/\epsilon_1}\left[\exp\left\{\upsilon\epsilon_1 \left(\epsilon_0\upsilon - 1\right)\frac{\sigma^2}2\int_0^tX_s^{2(\alpha-1)}ds \right\}\right] \EE^{1/\epsilon_2}\left[\left(\frac{X_t}{x_0}\right)^{-\epsilon_2\upsilon}\right].
\end{equation*}
The second  expectation on the right is  finite for any $\upsilon \epsilon_2>0$. The first one is finite whenever $\upsilon\epsilon_1 \left(\epsilon_0\upsilon - 1\right)\frac{\diffc^2}2< B_2$. 
In particular, taking $\epsilon_i = 3$, for $i=0,1,2$ and assuming $3\upsilon \left(3\upsilon - 1\right)\frac{\diffc^2}2< B_2$, we have
\begin{equation*}
\EE\left[\exp\left\{ -\upsilon\int_0^t\tfrac{b(X_s)}{X_s}ds\right\}\right]<+\infty,
\end{equation*}
provided that $\alpha>3/2$ when $0< b(0)<\diffc^2/2$.

\begin{thebibliography}{24}
\providecommand{\natexlab}[1]{#1}
\providecommand{\url}[1]{\texttt{#1}}
\expandafter\ifx\csname urlstyle\endcsname\relax
\providecommand{\doi}[1]{doi: #1}\else
\providecommand{\doi}{doi: \begingroup \urlstyle{rm}\Url}\fi

\bibitem[Ait-Sahalia(1996)]{Ait96}
Y.~Ait-Sahalia.
\newblock {Testing continuous-time models of the spot interest rate}.
\newblock \emph{Rev. Financial Stud.}, 9:\penalty0 385--426, 1996.

\bibitem[Berkaoui et~al.(2008)Berkaoui, Bossy, and Diop]{BBD08}
A.~Berkaoui, M.~Bossy, and A.~Diop.
\newblock {Euler scheme for SDEs with non-Lipschitz diffusion coefficient:
  strong convergence}.
\newblock \emph{ESAIM: Probab. Stat}, 12:\penalty0 1--12, 2008. \doi{10.1051/ps:2007030}. 
\bibitem[Bernardin et~al.(2009)Bernardin, Bossy, Martinez, and Talay]{bernardin2009passage}
F. Bernardin, M. Bossy, M. Martinez, and D. Talay.
\newblock {On mean numbers of passage times in small balls of discretized Itô processes}.
\newblock \emph{Electronic Communications in Probability}, 14:302 -- 316, 2009.
\newblock \doi{10.1214/ECP.v14-1479}.

\bibitem[Bossy et~al.(2021)Bossy, Jabir, and Mart\'inez]{BoJaMa2021}
M. Bossy, J-F. Jabir, and K. Mart\'inez.
\newblock {On the weak convergence rate of an exponential Euler scheme for SDEs governed by coefficients with superlinear growth}.
\newblock \emph{Bernoulli}, 27\penalty0 (1):\penalty0 312--347, 2021.
\newblock \doi{10.3150/20-BEJ1241}.

\bibitem[Bossy et~al.(2022)Bossy, Jabir, and Rodriguez]{bossy2022instantaneous}
M. Bossy, J-F. Jabir, and K.~Mart\'inez. 
\newblock Instantaneous turbulent kinetic energy modelling based on Lagrangian stochastic approach in CFD and application to wind energy.
\newblock \emph{Journal of Computational Physics}, 464:\penalty0 110929, 2022. \doi{10.1016/j.jcp.2021.110929}. 

\bibitem[Bréhier et~al.(2023)Bréhier, Cohen, and
  Ulander]{BDU_positivity_2023}
C-E. Bréhier, D. Cohen, and J. Ulander.
\newblock Analysis of a positivity-preserving splitting scheme for some nonlinear stochastic heat equations, 2023.  \doi{10.48550/arXiv.2302.08858}. 

\bibitem[Cai et~al.(2023)Cai, Guo, and Mao]{cai2023positivity}
Y.~Cai, Q.~Guo, and X.~Mao.
\newblock {Positivity preserving truncated scheme for the stochastic
  Lotka--Volterra model with small moment convergence}.
\newblock \emph{Calcolo}, 60\penalty0 (24), 2023.
\newblock \doi{10.1007/s10092-023-00521-9}.

\bibitem[Erdogan and Lord(2023)]{erdogan2023weak}
U. Erdogan and G.~J. Lord.
\newblock {Weak Convergence of Tamed Exponential Integrators for Stochastic Differential Equations}, 2023.

\bibitem[Geiss and Scheutzow(2021)]{geiss_scheutzow_2023}
S. Geiss and M. Scheutzow.
\newblock {Sharpness of Lenglart's domination inequality and a sharp monotone version}.
\newblock \emph{Electronic Communications in Probability}, 26: 1 -- 8, 2021.
\newblock \doi{10.1214/21-ECP413}.

\bibitem[Gray et~al.(2011)Gray, Greenhalgh, Hu, Mao, and
  Pan]{gray2011stochastic}
A. Gray, D. Greenhalgh, L. Hu, X. Mao, and J. Pan.
\newblock A stochastic differential equation {SIS} epidemic model.
\newblock \emph{SIAM Journal on Applied Mathematics}, 71\penalty0 (3):\penalty0
  876--902, 2011.

\bibitem[Greenhalgh et~al.(2016)Greenhalgh, Liang, and Mao]{GREENHALGH2016218}
D.~Greenhalgh, Y.~Liang, and X.~Mao.
\newblock {SDE SIS epidemic model with demographic stochasticity and varying population size}.
\newblock \emph{Applied Mathematics and Computation}, 276:\penalty0 218--238, 2016.
\newblock \doi{10.1016/j.amc.2015.11.094}.

\bibitem[Higham et~al.(2010)Higham, Mao, Pan, and Szpruch]{Spr}
D.~Higham, X.~Mao, J.~Pan, and L.~Szpruch.
\newblock {Numerical simulation of a strongly nonlinear Ait-Sahalia-type interest rate model}.
\newblock \emph{BIT Numer Math.}, 51:\penalty0 405--425, 2010.

\bibitem[Hutzenthaler et~al.(2010)Hutzenthaler, Jentzen, and Kloeden]{Kloeden}
M.~Hutzenthaler, A.~Jentzen, and P.~Kloeden.
\newblock {Strong and weak divergence in finite time of Euler's method for stochastic differential equations with non-globally Lipschitz continuous coefficients}.
\newblock \emph{Proceedings of the Royal Society}, 467:\penalty0 1563--1576,  2010.

\bibitem[Hutzenthaler et~al.(2012)Hutzenthaler, Jentzen, and Kloeden]{Kloeden2}
M.~Hutzenthaler, A.~Jentzen, and P.~Kloeden.
\newblock {Strong convergence of an explicit numerical method for SDEs with non-globally Lipschitz continuous coefficients}.
\newblock \emph{The Annals of Applied Probability}, 22:\penalty0 1611--1641, 2012.

\bibitem[Hutzenthaler and Jentzen(2020)]{Jentzen-b}
M. Hutzenthaler and A. Jentzen.
\newblock {On a perturbation theory and on strong convergence rates for stochastic ordinary and partial differential equations with nonglobally monotone coefficients}.
\newblock \emph{The Annals of Probability}, 48\penalty0 (1):\penalty0 53 -- 93, 2020.
\newblock \doi{10.1214/19-AOP1345}.

\bibitem[Ikeda and Watanabe(1981)]{IkedaWatanabe}
N.~Ikeda and S.~Watanabe.
\newblock \emph{{Stochastic differential equations and diffusion processes}}.
\newblock North-Holland Publishing Company, 1981.

\bibitem[Karatzas and Shreve(1988)]{KarShr-88}
I.~Karatzas and S.~Shreve.
\newblock \emph{Brownian Motion and Stochastic Calculus}.
\newblock Springer-Verlag, Berlin, 1988.

\bibitem[Kelly and Lord(2022)]{kelly2022adaptive}
C.~Kelly and G.~J. Lord.
\newblock {Adaptive Euler methods for stochastic systems with non-globally Lipschitz coefficients}.
\newblock \emph{Numerical Algorithms}, 89:\penalty0 721--747, 2022.
\newblock \doi{10.1007/s11075-021-01131-8}.

\bibitem[Lenglart(1977)]{Lenglart_77}
E.~Lenglart.
\newblock Relation de domination entre deux processus.
\newblock \emph{Annales de l'Institut Henri Poincar\'e. Section B. Calcul des probabilit\'es et statistiques}, 13\penalty0 (2):\penalty0 171--179, 1977.

\bibitem[Leobacher and Sz{\"o}lgyenyi(2017)]{leobacher_2017}
G. Leobacher and M. Sz{\"o}lgyenyi.
\newblock A strong order {1/2} method for multidimensional {SDEs} with discontinuous drift.
\newblock \emph{The Annals of Applied Probability}, 27\penalty0 (4),  2017.
\newblock \doi{10.1214/16-aap1262}.

\bibitem[Mao(2011)]{mao2007stochastic}
X. Mao.
\newblock \emph{Stochastic differential equations and applications}.
\newblock Second Edition. Elsevier, 2011.

\bibitem[Mao et~al.(2021)Mao, Wei, and Wiriyakraikul]{mao2021positivity}
X. Mao, F. Wei, and T. Wiriyakraikul.
\newblock Positivity preserving truncated Euler--Maruyama method for stochastic Lotka--Volterra competition model.
\newblock \emph{Journal of Computational and Applied Mathematics},
  394:\penalty0 113566, 2021.
\newblock \doi{10.1016/j.cam.2021.113566}.

\bibitem[Müller-Gronbach et~al.(2022)Müller-Gronbach, Sabanis, and
  Yaroslavtseva]{mullergronbach2022existence}
T. Müller-Gronbach, S. Sabanis, and L. Yaroslavtseva.
\newblock {Existence, uniqueness and approximation of solutions of SDEs with superlinear coefficients in the presence of discontinuities of the drift coefficient}, 2022.

\bibitem[Revuz and Yor(1999)]{Revuz_Yor_1999}
D. Revuz and M. Yor.
\newblock \emph{Continuous Martingales and Brownian Motion}.
\newblock Springer, Berlin, 3rd edition, 1999.
\end{thebibliography}
\end{document}